\documentclass[twoside,11pt,reqno]{amsart}
\usepackage{amsmath,amssymb,amscd,mathrsfs,epic,wasysym,latexsym,tikz,mathrsfs,cite,hyperref,enumerate}
\usepackage{pb-diagram}
\usepackage{tikz-cd}
\usepackage[all]{xy}
\usepackage{mathdots}

\makeatletter

\numberwithin{equation}{section}
\allowdisplaybreaks

\newtheorem{Proposition}[equation]{Proposition}
\newtheorem{Lemma}[equation]{Lemma}
\newtheorem{Theorem}[equation]{Theorem}

\newtheorem{MainTheorem}{Theorem}

\theoremstyle{definition}  %% makes all of the theorem environments which follow appear in \rm

\newtheorem{Remark}[equation]{Remark}

\let\<\langle
\let\>\rangle

\newcommand\Comment[2][\relax]{\space\par\medskip\noindent%
   \fbox{\begin{minipage}{\textwidth}\textbf{Comment\ifx\relax#1\else---#1\fi}\newline%
        #2\end{minipage}}\medskip
}

\def\bone{\text{\boldmath$1$}}

\def\pmod#1{\text{ }(\text{\rm mod } #1)\,}

\newcommand{\Hom}{\operatorname{Hom}}

\newcommand{\End}{\operatorname{End}}

\newcommand{\im}{\operatorname{im}}

\def\sgn{\mathtt{sgn}}

\newcommand{\res}{\operatorname{res}}
\newcommand{\soc}{\operatorname{soc}}
\newcommand{\head}{\operatorname{head}}

\newcommand{\Z}{\mathbb{Z}}

\def\eps{{\varepsilon}}
\def\phi{{\varphi}}

\newcommand\dbl{\operatorname{dbl}}

\newcommand{\F}{{\mathbb F}}

\newcommand{\la}{\lambda}

\newcommand{\al}{\alpha}
\newcommand{\be}{\beta}

\newcommand{\si}{\sigma}

\newcommand{\Om}{\Omega}

\newcommand{\de}{\delta}

\newcommand{\Mull}{{\tt M}}
\newcommand{\JS}{{\tt JS}}

\newcommand{\EE}{{\mathcal E}}

\newcommand{\SSS}{{\sf S}}
\newcommand{\AAA}{{\sf A}}

\newcommand{\da}{{{\downarrow}}}
\newcommand{\ua}{{\uparrow}}

\newcommand{\I}{{\mathcal I}}
\newcommand{\J}{{\mathcal J}}

\def\Parp{{\mathscr P}_p}
\def\Par{{\mathscr P}}
\def\Parinv{{\mathscr P}^\AAA}

\def\k{\Bbbk}

\def\spa{\operatorname{span}}

\def\im{{\mathrm{im}\,}}

\def\col{{\tt col}}
\def\row{{\tt row}}

\def\v{\bar v}

{\catcode`\|=\active
  \gdef\set#1{\mathinner{\lbrace\,{\mathcode`\|"8000%
  \let|\midvert #1}\,\rbrace}}
}
\def\midvert{\egroup\mid\bgroup}

\colorlet{darkgreen}{green!50!black}
\tikzset{dots/.style={very thick,loosely dotted},
         greendot/.style={fill,circle,color=darkgreen,inner sep=1.5pt,outer sep=0},
         blackdot/.style={fill,circle,color=black,inner sep=1.5pt,outer sep=0},
         graydot/.style={fill,circle,color=gray,inner sep=1.1pt,outer sep=0}
}
\def\greendot(#1,#2){\node[greendot] at(#1,#2){}}
\def\blackdot(#1,#2){\node[blackdot] at(#1,#2){}}
\def\graydot(#1,#2){\node[graydot] at(#1,#2){}}

\newenvironment{braid}{% sets defaults for the braid diagrams
  \begin{tikzpicture}[baseline=6mm,black,line width=1pt, scale=0.32,
                      draw/.append style={rounded corners},
                      every node/.append style={font=\fontsize{5}{5}\selectfont}]%
  }{\end{tikzpicture}
}

\def\Grid(#1,#2){%  draws a coordinate grid inside a braid diagram
  \draw[very thin,gray,step=2mm] (0,0)grid(#1,#2);
  \draw[very thin,darkgreen,step=10mm] (0,0)grid(#1,#2);
}

\newcommand\Tableau[2][\relax]{
  \begin{tikzpicture}[scale=0.5,draw/.append style={thick,black}]
    \ifx\relax#1\relax%
    \else % shade the boxes in #1
      \foreach\box in {#1} { \filldraw[blue!30]\box+(-.5,-.5)rectangle++(.5,.5); }
    \fi
    \newcount\row\newcount\col
    \row=0
    \foreach \Row in {#2} {
       \col=1
       \foreach\k in \Row {
          \draw(\the\col,\the\row)+(-.5,-.5)rectangle++(.5,.5);
          \draw(\the\col,\the\row)node{\k};
          \global\advance\col by 1
       }
       \global\advance\row by -1
    }
  \end{tikzpicture}
}

\newcommand\YoungDiagram[2][\relax]{
  \begin{tikzpicture}[scale=0.5,draw/.append style={thick,black}]
    \ifx\relax#1\relax%
    \else % shade the boxes in #1
    \foreach\box in {#1} {
      \filldraw[blue!30]\box rectangle ++(1,1);
    }
    \fi
    \newcount\row
    \row=0
    \foreach \col in {#2} {
       \draw(1,\the\row)grid ++(\col,1);
       \global\advance\row by -1
    }
  \end{tikzpicture}
}

\newdimen\hoogte    \hoogte=12pt    
\newdimen\breedte   \breedte=14pt  
\newdimen\dikte     \dikte=0.5pt 

\newenvironment{Young}{\begingroup
       \def\vr{\vrule height0.89\hoogte width\dikte depth 0.2\hoogte}
       \def\fbox##1{\vbox{\offinterlineskip
                    \hrule height\dikte
                    \hbox to \breedte{\vr\hfill##1\hfill\vr}
                    \hrule height\dikte}}
       \vbox\bgroup \offinterlineskip \tabskip=-\dikte \lineskip=-\dikte
            \halign\bgroup &\fbox{##\unskip}\unskip  \crcr }
       {\egroup\egroup\endgroup}
\def\Youngdiagram#1{\relax\ifmmode\vcenter{\,\begin{Young}#1\end{Young}\,}\else%
              $\vcenter{\,\begin{Young}#1\end{Young}\,}$\fi}

\begin{document}

\title[Irreducible restrictions of representations of symmetric groups]{{\bf Irreducible restrictions of representations of alternating groups in small characteristics: Reduction theorems}}

\author{\sc Alexander Kleshchev}
\address{Department of Mathematics\\ University of Oregon\\Eugene\\ OR 97403, USA}
\email{klesh@uoregon.edu}

\author{\sc Lucia Morotti}
\address
{Institut f\"{u}r Algebra, Zahlentheorie und Diskrete Mathematik\\ Leibniz Universit\"{a}t Hannover\\ 30167 Hannover\\ Germany} 
\email{morotti@math.uni-hannover.de}

\author{\sc Pham Huu Tiep}
\address
{Department of Mathematics\\ Rutgers University\\ Piscataway\\ NJ~08854, USA} 
\email{tiep@math.rutgers.edu}


\thanks{The first author was supported by the NSF grant DMS-1700905 and the DFG Mercator program through the University of Stuttgart. The second author was supported by the DFG grant MO 3377/1-1, the DFG Mercator program through the University of Stuttgart, and is also grateful to the University of Oregon for hospitality. The third author was supported by the NSF-grants DMS-1839351 and DMS-1840702. This work was also supported by the NSF grant DMS-1440140 and Simons Foundation while all three authors were in residence at the MSRI during the Spring 2018 semester.}

\begin{abstract}
We study irreducible restrictions from modules over alternating groups to subgroups. 
%A   irreducible modules over alternating groups which do not extend to symmetric groups. 
We get reduction results which substantially restrict the classes of subgroups and modules for which this is possible. This is known when the characteristic of the ground field is greater than $3$, but the small characteristics cases require a substantially more delicate analysis and new ideas. In view of our earlier work on symmetric groups we may consider only the restriction of irreducible modules  over alternating groups which do not extend to symmetric groups. This work fits into the Aschbacher-Scott program on maximal subgroups of finite classical groups. 
\end{abstract}

%}\fi

\maketitle

\section{Introduction}

Let $\F$ be an algebraically closed field of characteristic $p\geq 0$. Denote by $\AAA_n$ the alternating group on $n$ letters. We always assume that $n\geq 5$. 
%, and $H$ be a finite almost quasi-simple group. % with $\soc(H/Z(H))=\AAA_n$. Suppose that $n\geq 8$ to avoid small special cases. Then $H$ is one of $\AAA_n,\SSS_n$ or their double covers. 
In this paper we are concerned with the following problem 

\vspace{2mm}
\noindent
{\bf Problem 1 (Irreducible Restriction Problem for Alternating Groups).}
{\em
Classify the subgroups $G <\AAA_n$ and $\F \AAA_n$-modules $V$ of dimension greater than $1$ such that the restriction $V\da_G$ is irreducible.
}
\vspace{2mm}

%of classifying subgroups $G <\AAA_n$ and $\F \AAA_n$-modules $V$ of dimension greater than $1$ such that the restriction $V\da_G$ is irreducible. 
This is a special case of the general Irreducible Restriction Problem where we have an arbitrary almost quasi-simple group in place of $\AAA_n$. 
A major application of the Irreducible Restriction Problem is to the Aschbacher-Scott program on maximal subgroups of finite classical groups, see \cite{Asch,Scott,Magaard,KlL,BDR} for more details on this.  For the purposes of the applications to the Aschbacher-Scott program we may assume that $G$ is also almost quasi-simple, but  we will not be making this additional assumption. % for now. % and consider the Irreducible Restriction Problem as stated. 

For the case $p=0$, Problem 1 has been solved by Saxl \cite{Saxl}. 
%If $p=0$ and $H$ is a double cover of symmetric or alternating groups, the problem was essentially solved by Kleidman and Wales \cite{KlW}. 
Let us assume from now on that $p>0$. Indeed, it is the positive characteristic case which is important for the Aschbacher-Scott program. For $p>3$, Problem 1 is solved in  
\cite{KSAlt}. It is important to extend this to the case of characteristics $2$ and $3$. 
However, there are formidable technical obstacles which make the small characteristics  cases much more complicated. The most  serious difficulty is that the submodule structure of certain important permutation modules over symmetric groups gets very complicated for $p=2$ and $3$. This in turn necessitates a rather  detailed study of branching for symmetric groups. 

Let $V$ be an irreducible $\F\AAA_n$-module. If $V$ lifts to the symmetric group $\SSS_n$ then the problem reduces to the Irreducible Restriction Problem for Symmetric Groups, which is studied in \cite{BK}, where the problem is completely solved for $p>3$, and \cite{KMTOne}, where reduction theorems are obtained for the small characteristics cases. So in this paper we are concerned mostly with the case where $V$ does not lift to $\SSS_n$, and prove major reduction theorems for that case. These reduction theorems, together with the ones in \cite{KMTOne}, will play a key role in our future work \cite{KMTTwo}, which will complete the solution of the Irreducible Restriction 
Problem for both $\SSS_n$ and $\AAA_n$ (and $G$ a maximal subgroup) in all characteristics.  

To formulate our main result we recall some facts about irreducible representations of symmetric and alternating groups referring the reader to the main body of the paper for more details. The irreducible $\F \SSS_n$-modules are labeled by the $p$-regular partitions of $n$. If $\la$ is such a partition, we denote by $D^\la$ the corresponding irreducible $\F\SSS_n$-module. We refer the reader to \cite[\S11.1]{KBook} for the definitions of combinatorial notions of a residue of a node and of a normal node. 

It is known that $D^\la\da_{\SSS_{n-1}}$ is  irreducible  if and only if $\la$ is in the explicitly defined class of {\em Jantzen-Seitz} (or {\em JS}) partitions which go back to \cite{JS,k2}. 
There is a special irreducible $\F\SSS_n$-module in characteristic $2$ called the {\em basic spin module $D^{\be_n}$}.

We denote by $\Parinv_p(n)$ the set of all $p$-regular partitions of $n$ such that $D^\la\da_{\AAA_n}$ is reducible. If $\la\in \Parinv_p(n)$ we have $D^\la\da_{\AAA_n}\cong E^\la_+\oplus E^\la_-$ for irreducible $\F\AAA_n$-modules $E^\la_+\not\cong E^\la_-$. The set of partitions $\Parinv_p(n)$ is well understood---if $p=2$ it is described explicitly in \cite{Benson} while for $p>2$ these are exactly the partitions which are fixed by the Mullineux bijection, see \cite{Mull,FK,BO}. 

We formulate our main results for all characteristics, although they are  only new for $p=2$ and $3$:

\begin{MainTheorem}\label{TA}
Let $n\geq 5$, $\la\in\Parinv_p(n)$ and $G\leq \AAA_n$. If $E^\la_\pm\da_{G}$ is irreducible then one of the following statements holds.
\begin{enumerate}
\item[{\rm (i)}] $G$ is primitive.
\item[{\rm (ii)}] $G\leq \AAA_{n-1}$, and either 
\begin{enumerate}
\item[{\rm (a)}] $\la$ is JS, or
\item[{\rm (b)}] $\la$ has exactly two normal nodes, both of residue different from $0$.
\end{enumerate} 

\item[{\rm (iii)}] $G\leq \AAA_{n-2,2}$ and $\la$ is JS.

\item[{\rm (iv)}] $p=2$, $n\not\equiv 2\pmod{4}$ and $\la=\be_n$.
\end{enumerate}
\end{MainTheorem}

The exceptional case (i) in Theorem~\ref{TA} will be treated in \cite{KMTTwo}, and the exceptional cases (ii), (iii), (iv) are addressed in Theorems~\ref{TB}, \ref{TC}, \ref{TD}, respectively. 

\begin{MainTheorem}\label{TB}
Let $n\geq 5$, $\la\in\Parinv_p(n)$. Then  $E^\la_\pm\da_{\AAA_{n-1}}$ is irreducible if and only if one of the following statements holds.
\begin{enumerate}
\item[{\rm (a)}] $\la$ is JS.
\item[{\rm (b)}] $\la$ has exactly two normal nodes, both of residue different from $0$.
\end{enumerate} 
\end{MainTheorem}

We point out that the irreducible restrictions of the form $E^\la_\pm\da_{\AAA_{n-1}}$ for $p>2$ have been classified in \cite[Theorem 5.10]{BO2}, see also \cite[Proposition 3.7]{KSAlt}. For $p=2$ partial information is available in  \cite[Theorem 6.5 and Proposition 6.6]{BO2}, but Theorem~\ref{TB} says more. The irreducible restrictions $E^\la_\pm\da_{\AAA_{n-2,2}}$ for $p>2$ have been classified in \cite[Theorem 3.6]{KSAlt}, but for $p=2$ the following theorem is new. 

%namely if $E^\la_\pm\da_{A_{n-1}}$ is irreducible then either $\la$ is JS and $\la\neq \be_n$, or $\la$ is of a very special form described in \cite[Proposition 6.6]{BO2}. 

\begin{MainTheorem}\label{TC}
Let $n\geq 5$ and $\la\in\Parinv_p(n)$. If $p = 2$, assume in addition that $\la \neq \be_n$. Then the following are equivalent:
\begin{enumerate}
\item[{\rm (i)}] $E^\la_\pm\da_{\AAA_{n-2,2}}$ is irreducible;
\item[{\rm (ii)}] $E^\la_\pm\da_{\AAA_{n-2}}$ is irreducible;
\item[{\rm (iii)}] $\la$ is JS.
\end{enumerate} 
\end{MainTheorem}

The case  $\la=\be_n$, excluded in Theorems \ref{TA}(iv) and \ref{TC}, is handled in the following theorem (note that the condition $n\not\equiv 2\pmod{4}$ is equivalent to $\be_n\in\Parinv_2(n)$). 

\begin{MainTheorem}\label{TD}
Let $p=2$, $n\not\equiv 2\pmod{4}$ and $G\leq \AAA_n$. If  $E^{\be_n}_\pm\da_{G}$ is irreducible then one of the following statements 
holds.
\begin{enumerate}
\item[{\rm (i)}] $G$ is primitive.
\item[{\rm (ii)}] $G\leq \AAA_{n-k,k}$ for some $1\leq k<n$, and either 
$n\equiv 0\pmod{4}$ and $k$ is odd, or $k\equiv 2\pmod{4}$. Moreover, in all of these cases $E^{\be_n}_\pm\da_{\AAA_{n-k,k}}$ is irreducible.
\item[{\rm (iii)}] $G\leq (\SSS_a\wr \SSS_b)\cap \AAA_n$ for $a,b>1$ with $n=ab$,  and either $a$ is odd or $a\equiv 2\pmod{4}$ and $b=2$. Moreover, in all of these cases $E^\la_\pm\da_{(\SSS_a\wr \SSS_b)\cap \AAA_n}$ is irreducible.
\end{enumerate} 
\end{MainTheorem}

Theorems~\ref{TA}, \ref{TB}, \ref{TC}, \ref{TD} are proved in \S\S~\ref{SSTA}, \ref{SSTB}, \ref{SSTC}, \ref{SSTD} respectively.

\section{Preliminaries}

\subsection{Groups and modules}
\label{SSGM}
Throughout the paper we work over a fixed algebraically closed ground field $\F$ of characteristic $p>0$. We do not yet assume that $p=2$ or $3$ but will do this when necessary. 

For a group $G$, we denote by $\bone_G$  the trivial $\F G$-module. For an $\F G$-module $V$, we denote by $V^G$ the set of {\em $G$-invariant vectors} in $V$. If $L_1,\dots,L_a$ are irreducible $\F G$-modules, we denote by $L_1|\cdots| L_a$ a {\em uniserial} $\F G$-module with composition factors $L_1,\dots,L_a$ listed from socle to head. If $V$ is an $\F G$-module, we use the notation 
$$
V\cong L_1|\cdots| L_a\ \oplus\,\cdots\, \oplus\ K_1|\cdots| K_b
$$
to indicate that $V$ is isomorphic to a direct sum of the uniserial modules $L_1|\cdots| L_a$, \dots, $K_1|\cdots| K_b$. On the other hand, if $V_1,\dots,V_a$ are any $\F G$-modules, we write 
$$
V\sim V_1|\dots|V_a
$$
to indicate that $V$ has a filtration with subquotients $V_1,\dots,V_a$ listed from bottom to top. 
We use the notation 
$$
V\sim V_1|\cdots| V_a\ \oplus\,\cdots\, \oplus\ W_1|\cdots| W_b
$$
to indicate that $V\cong X\oplus\dots\oplus Y$ for $X\sim L_1|\cdots| L_a$, \dots, $Y\sim K_1|\cdots| K_b$.

%For a finite group $G$, we denote by $\mod{\F G}$ the category of finite dimensional $\F G$-modules. For $U,V\in\mod{\F G}$ we denote by $\Hom_G(U,V)$ the space of all $\F G$-module homomorphisms from $U$ to $V$, and by $\Hom_\F(U,V)$ the space of all linear maps considered as an $\F G$-module via $(g\cdot f)(u)=gf(g^{-1}u)$ for all $f\in \Hom_\F(U,V)$, $u\in U$ and $g\in G$. 

%Let $G$ be a subgroup of a group $H$, $V$ be an $\F H$-module and $W$ be an $\F G$-module. We denote by $V{\da}_G$ or $V{\da}^H_G$ the {\em restriction} of $V$ from $H$ to $G$, and by $W{\ua}^H$ or $W{\ua}^H_G$ the {\em induction} of $W$ from $G$ to $H$. As a special case, for 

For an odd element $\si\in\SSS_n$ and an $\F\AAA_n$-module $V$, we denote by $V^\si$ the $\F\AAA_n$-module which is $V$ as a vector space with the twisted action $g\cdot v=\si g\si^{-1}v$ for $g\in \AAA_n,v\in V$. 
If $G$ is a subgroup of $\SSS_n$ (resp. $\AAA_n$), we consider the induced modules
$$
\I(G):=\bone_G{\ua}^{\SSS_n}\qquad(\text{\rm resp.}\ \J(G):=\bone_G{\ua}^{\AAA_n}).
$$
For $G\leq \AAA_n$ we have 
$\I(G)\da_{\AAA_n}\cong \J(G)\oplus \J(G)^\si$.

For a composition $\mu=(\mu_1,\dots,\mu_r)$ of $n$ and positive integers $a,b$ with $ab=n$, we have the subgroups 
\begin{align*}
&\SSS_\mu:= \SSS_{\mu_1}\times\dots\times \SSS_{\mu_r}\leq \SSS_n, &&\SSS_a\wr\SSS_b\leq \SSS_n,
\\
&\AAA_\mu:= \SSS_{\mu}\cap \AAA_{n}\leq \AAA_n, &&
G_{a,b}:=(\SSS_a\wr\SSS_b)\cap \AAA_{n}\leq \AAA_n.
\end{align*}

%We write $\soc V$ and $\head V$ for the socle and head of $V$, respectively.

\subsection{Partitions}\label{SSPar}
We denote by $\Par(n)$ the set of all {\em partitions} of $n$ and by $\Parp(n)$ the set of all {\em $p$-regular} partitions of $n$, see \cite[10.1]{JamesBook}. 
We identify a partition $\la=(\la_1,\la_2,\dots)$ with its {\em Young diagram} $\{(r,s)\in\Z_{>0}\times\Z_{>0}\mid s\leq\la_r\}.$ 
%We have a {\em dominance order} $\unrhd$ on partitions, see \cite[3.2]{JamesBook}. 
The number of non-zero parts of a partition $\la$ is denoted by $h(\la)$. %We usually identify a partition and its Young diagram.
The following partition will play a special role in this paper:
\begin{eqnarray}\label{ESpin}
\be_n&:=&
\left\{
\begin{array}{ll}
(n/2+1,n/2-1) &\hbox{if $n$ is even,}\\
((n+1)/2,(n-1)/2) &\hbox{if $n$ is odd.}
\end{array}
\right.
\end{eqnarray}

We denote by $\la\mapsto \la^\Mull$ the {\em Mullineux bijection} on $\Par_p(n)$, see \cite{FK,BO,Mull}. If $p=2$, the Mullineux bijection is the identity map. 

For partitions $\mu^1=(\mu^1_1,\dots,\mu^1_{h_1})\in\Par(n_1), \dots, \mu^k=(\mu^k_1,\dots,\mu^k_{h_k})\in\Par(n_k)$, we define the composition 
$$
(\mu^1,\dots,\mu^k):=(\mu^1_1,\dots,\mu^1_{h_1},\dots,\mu^k_1,\dots,\mu^k_{h_k})
$$
of $n_1+\dots+n_k$. 
For a partition $\la=(\la_1,\dots,\la_h)$ of $n$, we now define its {\em double}
$$
\dbl(\la):=(\be_{\la_1},\dots,\be_{\la_h}).
$$
Following \cite{Benson}, we set 
$$
\Parinv_2(n):=\Par_2(n)\cap \{\dbl(\la)\mid \la\in\Par_2(n),\ \la_r\not\equiv 2\pmod{4}\ \text{for}\ 1\leq r\leq h(\la)\}.
$$
On the other hand, if $p>2$, we set 
$$
\Parinv_p(n):=\{\la\in\Par_p(n)\mid \la=\la^\Mull\}.
$$

\begin{Lemma} \label{L080518} %{\rm \cite{}}
Suppose that $n\geq 5$ and $\la\in \Parinv_p(n)$. Then $h(\la)\geq 3$ unless $p=2$, $n\not\equiv 2\pmod{4}$ and $\la=\be_n$. 
\end{Lemma}
\begin{proof}
For $p=2$ this is clear from the definition.  
For $p>2$ the result is contained in \cite[Lemma 1.8(i)]{KSAlt}. \end{proof}

Let
$
I:=\Z/p\Z
$
identified with $\{0,1,\dots,p-1\}$. 
Given a node $A=(r,s)$ in row $r$ and column $s$, we consider its {\em residue}
$
\res A:=s-r\pmod{p}\in I.
$
%The {\em residue content} of a partition $\la$ is the tuple 
%$$\operatorname{cont}(\la):=(a_i)_{i\in I}$$ such that $\la$ has exactly $a_i$ nodes of residue $i$ for each $i\in I$. For $j\in I$, let $\ga_j$ be the tuple $(a_i)_{i\in I}$ with $a_i=\de_{i,j}$. We consider the tuples $(a_i)_{i\in I}$ as elements of $\Theta:=\sum_{i\in I}\Z\cdot\ga_i$, the free $\Z$-module with basis $\{\ga_i\mid i\in I\}$. Let 
%\begin{equation}\label{ETheta}
%\Theta_n:=\left\{\theta=\sum_{i\in I}a_i\ga_i\in\Theta\ \bigl{|}\ a_i\geq 0,\ \sum_{i\in I} a_i=n\right\}.
%\end{equation}
%Partitions $\la,\mu\in\Par(n)$ have the same residue contents if and only if they have the same $p$-cores, see \cite[2.7.41]{JK}. 
Let $i\in I$ and $\la\in\Par(n)$. 
A node $A 
\in \la$ (resp. $B\not\in\la$) is called {\em $i$-removable} (resp. {\em $i$-addable}) for $\la$ 
if $\res A=i$ and $\la_A:=\la\setminus\{A\}$ (resp. $\la^B:=\la\cup\{B\}$) is a Young diagram of a partition. 
We refer the reader to \cite[\S11.1]{KBook} for the definition of {\em $i$-normal}, {\em $i$-conormal}, {\em $i$-good}, and {\em $i$-cogood} nodes for $\la$. 
We denote 
\begin{align*}
\eps_i(\la):=\sharp\{\text{$i$-normal nodes for $\la$}\},
\quad
\phi_i(\la):=\sharp\{\text{$i$-conormal nodes for $\la$}\}.
%\label{EPhi}
\end{align*}
If $\eps_i(\la)>0$, let $A$ be the $i$-good node of $\la$ and set
$
\tilde e_i \la:=\la_A.
$
If $\phi_i(\la)>0$, let $B$ be the $i$-cogood node for $\la$ and set
$
\tilde f_i\la:=\la^B.
$
Then $\tilde e_i \la$ and $\tilde f_i\la$ are $p$-regular, whenever $\la$ is so. 

We call $\la\in\Par_p(n)$ a {\em JS partition} if $\la$ has only one normal node, equivalently $\sum_{i\in I}\eps_i(\la)=1$. We will need the following technical result on JS partitions for $p=3$:

\begin{Lemma}\label{L110518_2}
Let $\la\in\Parinv_3(n)$ be a JS partition. Then one of the following holds:
\begin{enumerate}
\item[{\rm (i)}] $\la_1\geq\la_2+9$, $\la_3\geq 7$, $\la_1\leq(n+2)/2$, and  $n\geq 4h(\la)$.
\item[{\rm (ii)}] $\la_1\geq\la_2+7\geq\la_3+10$, $\la_4\geq 6$, $\la_1+\la_2\leq (n+8)/2$, $h(\la)\geq 6$ and  $n\geq 6h(\la)$.
\item[{\rm (iii)}] $\la_1\geq\la_2+4\geq\la_3+8$, $\la_4\geq 4$, $\la_1+\la_2\leq (n+8)/2$, $h(\la)\geq 6$ and  $n\geq 6h(\la)$.
\item[{\rm (iv)}] $\la$ is one of the following: $(1)$, $(4,1^2)$, $(7,3,2)$, $(10,4^2)$, $(13,6,5)$, $(7,3,2,1)$, $(10,4^2,1)$, $(13,6,5,1)$, $(10,6,3^2,1^2)$, $(13,6,5,4,1^2)$, $(13,9,5,4,3,2,1)$.
\end{enumerate}
\end{Lemma}
\begin{proof}
Let $\left(\begin{array}{ccc}a_0&\dots&a_k\\r_0&\dots&r_k\end{array}\right)$ be the Mullineux symbol of $\la$, and let $\la^{(0)}=\la,\la^{(1)},\ldots,\la^{(k)}$ be obtained by recursively removing the $3$-rim. 
From \cite[Theorem 4.1]{BO3} we have that $\left(\begin{array}{c}a_k\\r_k\end{array}\right)=\left(\begin{array}{c}1\\1\end{array}\right)$ and that for $0\leq j<k$:
\begin{itemize}
\item if $\left(\begin{array}{c}a_{j+1}\\r_{j+1}\end{array}\right)=\left(\begin{array}{c}6c+1\\3c+1\end{array}\right)$ then $\left(\begin{array}{c}a_j\\r_j\end{array}\right)=\left(\begin{array}{c}6(c+1)-1\\3(c+1)\end{array}\right)$,

\item if $\left(\begin{array}{c}a_{j+1}\\r_{j+1}\end{array}\right)=\left(\begin{array}{c}6c-1\\3c\end{array}\right)$ or $\left(\begin{array}{c}a_{j+1}\\r_{j+1}\end{array}\right)=\left(\begin{array}{c}6c\\3c\end{array}\right)$ then $\left(\begin{array}{c}a_j\\r_j\end{array}\right)=\left(\begin{array}{c}6c\\3c\end{array}\right)$ or $\left(\begin{array}{c}a_j\\r_j\end{array}\right)=\left(\begin{array}{c}6c+1\\3c+1\end{array}\right)$.
\end{itemize}
Note in particular that $h(\la)=r_0$ is of the form $3c$ or $3c+1$ for some $c\geq 0$. 

\vspace{2mm}
\noindent
{\sf Claim 1:} {\em if $h(\la)$ is of the form $3c$ or $3c+1$ then $n\geq 2ch(\la)$. }

\vspace{1 mm}
\noindent
Indeed, if $h(\la)=3c$ then $\left(\begin{array}{c}6j-5\\3j-2\end{array}\right)=\left(\begin{array}{c}6(j-1)+1\\3(j-1)+1\end{array}\right)$ and $\left(\begin{array}{c}6j-1\\3j\end{array}\right)$ appear as column of the Mullineux symbol for each $1\leq j\leq c$. So
\[n=a_0+\dots+a_k\geq\sum_{j=1}^c(6j-5+6j-1)=6c^2=2ch(\la),\]
while if $h(\la)=3c+1$ then similarly 
\[n\geq6c+1+\sum_{j=1}^c(6j-5+6j-1)=6c^2+6c+1\geq2ch(\la).\]

\vspace{2mm}
\noindent
{\sf Claim 2:} {\em if $n\geq 42$ then $n\geq 6h(\la)$.}

\vspace{1 mm}
\noindent
Indeed, if $c\geq 3$ then $n\geq 6h(\la)$ by Claim 1, so we may assume that $c\leq 2$, in which case $h(\la)\leq 7$ and if $n\geq 42$ then $n\geq 6h(\la)$.
\vspace{2mm}

The next two claims are easy to see. 

\vspace{2mm}
\noindent
{\sf Claim 3:} {\em  If $\la^{(j)}_1\geq\la^{(j)}_2+3\geq\ldots\geq\la^{(j)}_m+3(m-1)$ 
for some $1\leq j\leq k$ and $m\geq 2$ 
then $\la^{(j-1)}_s-\la^{(j-1)}_{s+1}\geq\la^{(j)}_s-\la^{(j)}_{s+1}$ for all  $1\leq s<m$. }

\vspace{2mm}
\noindent
{\sf Claim 4:} {\em $\la^{(j-1)}_l\geq\la^{(j)}_l$ for all $1\leq j\leq k$ and $l\geq 1$.}

\vspace{2mm}
\noindent
{\sf Claim 5:} {\em if $h(\la^{(j)})\geq 3c$ and $\la^{(j)}_1+\ldots+\la^{(j)}_c\leq (|\la^{(j)}|+b)/2$ for some $1\leq j\leq k$,  $c\in\Z_{>0}$ and $b\in \Z$, then $\la^{(j-1)}_1+\ldots+\la^{(j-1)}_c
\leq (|\la^{(j-1)}|+b)/2$}.

\vspace{1 mm}
\noindent
Indeed, using the fact that $h(\la^{(j)})=r_j$ and $\la^\Mull=\la$, we deduce that $a_{j-1}\geq 6c$. So 
\begin{align*}
\la^{(j-1)}_1+\ldots+\la^{(j-1)}_c
&\leq\la^{(j)}_1+\ldots+\la^{(j)}_c+3c
\leq (|\la^{(j)}|+6c+b)/2
\\
&=(|\la^{(j-1)}|-a_{j-1}+6c+b)/2
\leq (|\la^{(j-1)}|+b)/2.
\end{align*}

\vspace{2mm}
\noindent
{\sf Claim 6:} {\em  if $1\leq j\leq k$, $|\la^{(j)}|\geq 42$ and $\la^{(j)}$ satisfies (i) (resp. (ii), resp. (iii)), then so does $\la$. }

\vspace{1 mm}
\noindent
We provide the proof for the condition (i), the conditions (ii) and (iii) are treated similarly. The condition $\la_1\geq \la_2+9$ is deduced using Claim 3 with $m=2$. The condition $\la_3\geq 7$ comes from Claim 4. The condition $\la_1\leq(n+2)/2$ comes from Claim 5 with $c=1$ and $b=2$ since $\la_3\geq 7$ of course implies $h(\la)\geq 3$. The condition $n\geq 4h(\la)$ comes from Claim 2. 

\vspace{2mm}
For $n<42$ the lemma holds by inspection (the exceptional cases are listed in part (iv)). Assume that $n\geq 42$. Pick $j$ maximal such that $|\la^{(j)}|\geq 42$. Then $|\la^{(j+1)}|<42$ and by inspection again we see that (i), (ii) or (iii) holds for $\la^{(j)}$. The proof is completed using Claim 6. 
\end{proof}

\subsection{Irreducible modules over symmetric and alternating groups}
We use James' notation 
$\{D^\la\mid \la\in\Par_p(n)\}$ for the set of the irreducible $\F\SSS_n$-modules up to isomorphism, see \cite[\S11]{JamesBook}. 
For example, $D^{(n)}\cong\bone_{\SSS_n}$. By \cite{JS} and \cite{k2}, $D^\la\da_{\SSS_{n-1}}$ is irreducible if and only if $\la$ is JS. The following much more general result is contained in \cite[Theorems 11.2.10]{KBook} and \cite[Theorem 1.4]{KDec}. % and \cite[Theorem E$'$(iv)]{BrK1}.

\begin{Lemma}\label{Lemma39}
Let $\lambda\in\Parp(n)$, $i\in I$ and $r\in\Z_{\geq 0}$. Then:
\begin{enumerate}
\item[{\rm (i)}] $e_i^rD^\lambda\cong(e_i^{(r)}D^\lambda)^{\oplus r!}$;
\item[{\rm (ii)}]  $e_i^{(r)}D^\lambda\not=0$ if and only if $r\leq \eps_i(\lambda)$, in which case $e_i^{(r)}D^\lambda$ is a self-dual indecomposable module with socle and head both isomorphic to $D^{\tilde e_i^r\la}$.  
\item[{\rm (iii)}]  $[e_i^{(r)}D^\lambda:D^{\tilde e_i^r\la}]=\binom{\eps_i(\lambda)}{r}=\dim\End_{\SSS_{n-r}}(e_i^{(r)}D^\lambda)$;
\item[{\rm (iv)}] if $D^\mu$ is a composition factor of $e_i^{(r)}D^\lambda$ then $\eps_i(\mu)\leq \eps_i(\lambda)-r$, with equality holding if and only if $\mu=\tilde e_i^r\la$;
\item[{\rm (v)}] 
$\dim\End_{\SSS_{n-1}}(D^\lambda\da_{\SSS_{n-1}})=\sum_{j\in I}\eps_j(\lambda)$.
\item[{\rm (vi)}] Let $A$ be a removable node of $\la$ such that $\la_A$ is $p$-regular. Then $D^{\la_A}$ is a composition factor of $e_i D^\la$ if and only if $A$ is $i$-normal, in which case $[e_i D^\la:D^{\la_A}]$ is one more than the number of $i$-normal nodes for $\la$ above $A$. 
\end{enumerate}
\end{Lemma}

To describe the irreducible $\F\AAA_n$-modules, let us first suppose that $p=2$. For $\la\in\Par_2(n)$, by \cite[Theorem 1.1]{Benson}, 
we have $D^\la\da_{\AAA_n}$ is irreducible if and only if $\la\not\in\Parinv_2(n)$. In this case, we denote $E^\la=D^\la\da_{\AAA_n}$. On the other hand, if $\la\in\Parinv_2(n)$, then $D^\la\da_{\AAA_n}\cong E^\la_+\oplus E^\la_-$ for irreducible $\F \AAA_n$-modules $E^\la_+\not\cong E^\la_-$. Moreover
$$
\{E^\la\mid \la\in\Par_2(n)\setminus \Parinv_2(n)\}\cup\{E^\la_\pm\mid \la\in\Parinv_2(n)\}
$$
is a complete set of irreducible $\F\AAA_n$-modules up to isomorphism.

%By \cite[11.5]{JamesBook}, we have $(D^\la)^*\cong D^\la$ for all $\la\in\Parp(n)$. 

Now, let $p>2$. %The Mullineux involution acts on $\Par_p(n)$
We denote by $\sgn$ the sign module over $\SSS_n$. Then by \cite{FK} (see also \cite{BO}), we have $D^\la \otimes \sgn \cong D^{\la^\Mull}$, and $E^\la:=D^\la\da_{\AAA_n} \cong D^{\la^\Mull}\da_{\AAA_n}$ is irreducible if and only if $\la\neq\la^\Mull$. If $\la=\la^\Mull$, i.e. $\la\in\Parinv_p(n)$, we have $D^\la\da_{\AAA_n}\cong E^\la_+\oplus E^\la_-$ for irreducible $\F \AAA_n$-modules $E^\la_+\not\cong E^\la_-$. By Clifford theory, 
$$
\{E^\la\mid \la\in\Par_p(n)\setminus \Parinv_p(n)\}\cup\{E^\la_\pm\mid \la\in\Parinv_p(n)\}
$$
is a complete set of irreducible $\F\AAA_n$-modules, and 
%every irreducible $\F\AAA_n$-module is isomorphic to either $E^\la$ with $\la\in\Par_p(n)\setminus\Parinv_p(n)$ or to $E^\la_\pm$ with $\la\in\Parinv_p(n)$, and 
$E^\la\cong E^{\la^\Mull}$ for $\la\in\Par_p(n)\setminus \Parinv_p(n)$ are the only non-trivial isomorphisms among these.

For any $p$ we now have that $(E^\la_\pm)^\si\cong E^\la_{\mp}$ for $\si\in\SSS_n\setminus \AAA_n$ and $\la\in\Parinv_p(n)$. 
It follows that if $G=\si G\si^{-1}$ is a subgroup of $\AAA_n$ then $E^\la_+\da_G$ is irreducible if and only if $E^\la_-\da_G$ is irreducible. For example, this applies to the subgroups of the form  $\AAA_\mu$ and $G_{a,b}$.

\begin{Lemma} \label{L080518_4} %{\rm \cite{}}
Let $V$ be an $\F\SSS_n$-module, $W$ be an $\F\AAA_n$-module and $\mu\in\Par_p(n)\setminus \Parinv_p(n)$. 
\begin{enumerate}
\item[{\rm (i)}] If there is $\psi\in\Hom_{\SSS_n}(W\ua^{\SSS_n}, V)$ such that $[\im \psi:D^\mu]\neq 0$ then there exists $\psi'\in\Hom_{\AAA_n}(W, V\da_{\AAA_n})$ such that $[\im \psi':E^\mu]\neq 0$.
\item[{\rm (ii)}] If there is $\psi\in\Hom_{\SSS_n}( V, W\ua^{\SSS_n})$ such that $[\im \psi:D^\mu]\neq 0$ then there exists $\psi'\in\Hom_{\AAA_n}(V\da_{\AAA_n},W)$ such that $[\im \psi':E^\mu]\neq 0$.
\end{enumerate}
\end{Lemma}
\begin{proof}
We prove (i), the proof of (ii) being similar. Since $W\ua^{\SSS_n}\da_{\AAA_n}\cong W\oplus W^\si$, there exists $\psi'$ as required or there exists  $\psi''\in\Hom_{\AAA_n}(W^\si, V\da_{\AAA_n})$ such that $[\im \psi'':E^\mu]\neq 0$. In the second case, twisting $\psi''$ with $\si$ yields the required $\psi'$. 
\end{proof}

\begin{Lemma} \label{LKZ} %{\rm \cite{}}%{\bf ()}
Let $n\geq 8$, $\la\in\Par_p(n)$, and $\SSS_4\times \SSS_4\leq \SSS_8\leq \SSS_n$ be natural subgroups. Then $D^\la\da_{\SSS_4\times \SSS_4}$ has a composition factor of the form $D^\mu\boxtimes D^\nu$ with $\dim D^\mu>1$ and $\dim D^\nu>1$, unless $\la$ or $\la^\Mull$ belongs to $\{(n),(n-1,1)\}$. 
\end{Lemma}
\begin{proof}
If $n=8$ this is an easy explicit check. Now the result follows by induction using \cite[Proposition 2.3]{KZ}. 
\end{proof}

\subsection{Some special permutation modules}\label{SSPerm}
For a $2$-row partition $(n-k,k)$, we use the special notation
$$
S_k:=S^{(n-k,k)}\quad\text{and}\quad M_k:=M^{(n-k,k)}.
$$
(when it is clear what $n$ is). 
If $(n-k,k)\in\Par_p(n)$, we also denote
$$
D_k:=D^{(n-k,k)}\quad\text{and}\quad E_k:=D_k\da_{\AAA_n}.
$$
By Lemma~\ref{L080518}, we almost always have $(n-k,k)\not\in\Parinv_p(n)$, 
in which case 
$E_k\cong E^{(n-k,k)}$ is irreducible. %The case $k=0,1,2,3$ will play a special role in this paper. 

Let $0\leq k\leq n/2$ and $G\leq \SSS_n$. We denote by $i_k(G)$ the number of $G$-orbits on $\Om_k$. Note that 
\begin{equation}\label{EItoM}
i_k(G)=\dim M_k^G=\dim\Hom_{\SSS_n}(\I(G),M_k). 
\end{equation}

We will need the following information on the structure of some special permutation modules.

\begin{Theorem} \label{Tp=3}
{\rm \cite[Lemmas 4.3, 4.4, 4.5]{KMTOne}} %{\bf ()}
Let $p=3$ and $n\geq 6$. Then
$$M_1\sim D_0|S_1^*\quad\text{and}\quad M_2\sim M_1|S_2^*.$$
Further
\begin{enumerate}
\item[{\rm (i)}] If  $n\equiv 0\pmod{3}$ then $M_3\sim S_2^*\oplus ((D_0\oplus S_1^*)|S_3^*)$.

\item[{\rm (ii)}] If $n\equiv 1\pmod{3}$ then $M_3\sim S_1^*\oplus ((D_0\oplus S_2^*)|S_3^*)$.

\item[{\rm (iii)}] If $n\equiv 2\pmod{3}$ then $M_3\sim M_2|S_3^*$.
\end{enumerate}
\end{Theorem}

\begin{Theorem} \label{Tp=2}
{\rm \cite[Lemmas 4.6, 4.7, 4.9]{KMTOne}}
Let $p=2$ and $n\geq 7$. Then $M_1\sim D_0|S_1^*$. Further
\begin{enumerate}
\item[{\rm (i)}] If  $n\equiv 0\pmod{4}$ then
$$
M_2\sim(D_0\oplus S_1^*)|S_2^*\quad\text{and}\quad 
M_3\sim M_1\oplus (S_2^*|S_3^*).
$$
\item[{\rm (ii)}] If  $n\equiv 1\pmod{4}$ then
$$
M_2\sim M_1|S_2^*\quad\text{and}\quad
M_3\cong M_2\oplus S_3^*.
$$
\item[{\rm (iii)}] If $n\equiv 2\pmod{4}$, then
$$
M_2\sim(D_0\oplus S_1^*)|S_2^*\quad\text{and}\quad
M_3\sim M_1|S_2^*|S_3^*.
$$
\item[{\rm (iv)}] If $n\equiv 3\pmod{4}$, then 
$$
M_2\cong M_1\oplus S_2^*\quad\text{and}\quad
M_3\sim M_2|S_3^*.
$$
\end{enumerate}
\end{Theorem}

\begin{Lemma} \label{C180418} 
Let $n\geq 5$, and $\la\in\Par_p(n)$ be such that $\dim D^\la>1$. If $p=2$ assume further that $\la\neq \be_n$. Then there exists $\zeta_2\in\Hom_{\SSS_n}(M_2, \End_\F(D^\la))$ with $[\im\zeta_2:D_2]\neq 0$.  
\end{Lemma}
\begin{proof}
This follows from \cite[Corollary 6.4]{KMTOne} and \cite[Lemma 3.8]{KS2Tran}.
\end{proof}

\begin{Lemma} \label{LM_3E} 
Let $n\geq 7$, and $\la\in\Par_p(n)$ with $h(\la),h(\la^\Mull)\geq 3$. Then there exists $\zeta_3\in\Hom_{\SSS_n}(M_3, \End_\F(D^\la))$ with $[\im\zeta_3:D_3]\neq 0$.  
\end{Lemma}
\begin{proof}
This follows from \cite[Corollaries 6.7, 6.10]{KMTOne} and 
\cite[Lemmas 3.1,3.2 and Corollary 3.9]{BK}.
\end{proof}

\iffalse{
\begin{Theorem} \label{Tp=3}
{\rm \cite[Lemmas 4.3, 4.4, 4.5]{KMTOne}} %{\bf ()}
Let $p=3$ and $n\geq 6$. 
\begin{enumerate}
\item[{\rm (i)}] If  $n\equiv 0\pmod{3}$ then 
$$M_1\cong D_0|D_1|D_0,\quad M_2\cong D_2\oplus M_1\quad\text{and}\quad M_3\sim D_2\oplus ((D_0\oplus S_1^*)|S_3^*).$$

\item[{\rm (ii)}] If $n\equiv 1\pmod{3}$ then %S_2\cong D_0|D_2
$$
M_1\cong D_0\oplus D_1, \quad M_2\cong D_1\oplus D_0|D_2|D_0,\quad\text{and}\quad
M_3\sim D_1\oplus ((D_0\oplus S_2^*)|S_3^*).
$$

\item[{\rm (iii)}] If $n\equiv 2\pmod{3}$ then 
$$M_1\cong D_0\oplus D_1,\quad M_2\cong D_0\oplus D_1|D_2|D_1,\quad\text{and}\quad M_3\sim M_2|S_3^*.$$
\end{enumerate}
\end{Theorem}

\begin{Theorem} \label{Tp=2}
{\rm \cite[Lemmas 4.6, 4.7, 4.9]{KMTOne}}
Let $p=2$ and $n\geq 7$. 
\begin{enumerate}
\item[{\rm (i)}] If  $n\equiv 0\pmod{4}$ then
$$
M_1\cong D_0|D_1|D_0,\quad 
M_2\sim(D_0\oplus S_1^*)|S_2^*,\quad\text{and}\quad 
M_3\cong M_1\oplus (S_2^*|S_3^*).
$$
\item[{\rm (ii)}] If  $n\equiv 1\pmod{4}$ then
$$
M_1\cong D_0\oplus D_1,\quad
M_2\cong D_1\oplus D_0|D_2|D_0,\quad\text{and}\quad
M_3\cong D_3\oplus M_2.
$$
\item[{\rm (iii)}] If $n\equiv 2\pmod{4}$, then
$$
M_1\cong D_0|D_1|D_0,\quad
M_2\sim(D_0\oplus S_1^*)|S_2^*,\quad\text{and}\quad
M_3\sim M_1|S_2^*|S_3^*.
$$
\item[{\rm (iv)}] If $n\equiv 3\pmod{4}$, then 
$$
M_1\cong D_0\oplus D_1,\quad
M_2\cong D_0\oplus D_1\oplus D_2,\quad\text{and}\quad
M_3\cong D_0\oplus D_2\oplus D_1|D_3|D_1.
$$
\end{enumerate}
\end{Theorem}

}\fi

\subsection{Invariants}
In this subsection we will compute some invariants $(S_k^*)^G$ for small $k$. We use the standard basis $v_1,\dots,v_n$ in $M_1$ and the corresponding elements $\bar v_1,\dots\bar v_n\in S_1^*=M_1/\langle \sum_{j=1}^nv_j\rangle$ so that $\{\bar v_1,\dots, \bar v_{n-1}\}$ is a basis of $S_1^*$.

Let $\Om_n$ be the set of all $2$-element subsets of $\{1,\dots,n\}$ . 
We  use the standard basis 
$\{v_A\mid A\in \Om_{n}\}$ in $M_2$ and write $v_{i,j}:=v_{\{i,j\}}$ for $\{i,j\}\in\Om_n$. It is easy to check that $S_2^*\cong M_2/K$, where 
$$
K:=\spa\Big(\sum_{A\in \Om_{n}} v_A,\ \sum_{j\not=i}v_{i,j}\mid 1\leq i\leq n\Big).
$$
Set $\bar v_A:=v_A+K\in M_2/K=S_2^*$. Then 
$$
\{\v_A\mid A\in\Om_{n-2}\}\cup\{\v_{i,n-1}\mid 1\leq i\leq n-3\}
$$
is a basis of $S_2^*$, and
\begin{align*}
\v_{i,n}&=-\sum_{j\in[1,n-1]\setminus\{i\}}\v_{i,j}\qquad(1\leq i\leq n-3),
\\
\v_{n-2,n-1}&=-\sum_{A\in\Om_{n-2}}\v_A-\sum_{i=1}^{n-3}\v_{i,n-1},
\\
\v_{n-2,n}&=\sum_{A\in\Om_{n-3}}\v_A+\sum_{i=1}^{n-3}\v_{i,n-1},
\\
\v_{n-1,n}&=\sum_{A\in\Om_{n-2}}\v_A.
\end{align*}

\begin{Lemma}\label{L100518_2}
Let $n\geq 5$, $1\leq k\leq n/2$, and  $G=A_{n-k,k}$. Then $\dim(S_1^*)^G=\dim(S_2^*)^G=1$, with the only exception $(S_2^*)^{A_{n-1}}=0$.
\end{Lemma}
\begin{proof}
For $S_1^*$ this is an easy explicit check left to the reader. For $S_2^*$, assume first that $k=1$. Then, acting with $A_{n-3}$, we deduce that $x\in (S_2^*)^G$ must be of the form
$$
x=c\sum_{A\in\Om_{n-3}}\v_A+d\sum_{i=1}^{n-3}\v_{i,n-2}+e\sum_{i=1}^{n-3}\v_{i,n-1}
$$
for $c,d,e\in\F$. Acting with $(1,2)(n-2,n-1)$ gives $e=d$ and acting  with $(1,2)(n-3,n-2)$ then gives $c=d=0$. 
The case $k=2$ is handled similarly, giving $e=0$ and $c=d$, which gives a non-trivial invariant if $c=1$.

Let $k> 2$. Note that $\sum_{A\in\Om_{n-k}}\v_A\in(S_2^*)^G$. If $k\geq 3$, then acting with $A_{n-k,k-3}$, we deduce that $x\in (S_2^*)^G$ must be of the form
\begin{align*}
x\,\,=\,\,&a\sum_{A\in\Om_{n-k}}\v_A+b\sum_{j=n-k+1}^{n-3}\sum_{i=1}^{n-k}\v_{i,j}+c\sum_{i=1}^{n-k}\v_{i,n-2}+d\sum_{i=n-k+1}^{n-3}\sum_{j=i+1}^{n-3}\v_{i,j}
\\&+e\sum_{i=n-k+1}^{n-r}\v_{i,n-2}
+f\sum_{i=1}^{n-k}\v_{i,n-1}+g\sum_{i=n-k+1}^{n-3}\v_{i,n-1}
\end{align*}
for some $a,b,c,d,e,f,g\in\F$. In view of the invariant already found, we may assume that $a=0$ and then prove that $x=0$. Acting with $(1,2)(n-1,n)$ we get $f=g=0$. Then acting with $(1,2)(n-2,n-1)$ we get $c=e=0$. In the case $k=3$ we are done since then the two remaining sums are empty, and we done. Otherwise, 
acting with $(1,2)(n-3,n-2)$ we get $b=d=0$. 
\end{proof}

\begin{Lemma} \label{L100518_3}%{\rm \cite{}}
Let $G_{a,b}\leq \SSS_n$ for $n=ab\geq 6$. Then: 
\begin{enumerate}
\item[{\rm (i)}] if $a,b\geq 2$ then $\dim (S_1^*)^{G_{a,b}}=0$, unless $p=b=2$ in which case $\dim (S_1^*)^{G_{a,b}}\neq0$
\item[{\rm (ii)}] if $a,b\geq 3$ then $\dim (S_2^*)^{G_{a,b}}=1$.
\end{enumerate}
\end{Lemma}
\begin{proof}
(i) is an easy explicit calculation similar to the proof of \cite[Lemma 2.35]{KMTOne}. 

(ii) For $r=1,\dots,b$, we set $B_r=[(r-1)a+1,ra]$. For $1\leq i\neq j\leq n$ we write $i\sim j$ if $i,j\in B_r$ for some $r$. Starting with the invariant vector $v:=\sum_{i\sim j}v_{i,j}\in M_2$, we express $\v:=v+K\in S_2^*$ as
$$
\v=2\sum_{\stackrel{1\leq i<j\leq n-a}{i\sim j}}\v_{i,j}+\sum_{\stackrel{1\leq i<j\leq n-a}{i\not\sim j}}\v_{i,j}
$$
which yields a non-zero vector in $(S_2^*)^{G_{a,b}}$. Let now $x\in (S_2^*)^{G_{a,b}}$. Acting with $((\SSS_a\wr \SSS_{b-1})\times\SSS_{\{n-a+1,\dots,n-3\}})\cap\AAA_n$ and subtracting a multiple of $\v$, we may assume that
\begin{align*}
x\,\,=\,\,
&a\sum_{\stackrel{1\leq i<j\leq n-a}{i\sim j}}\v_{i,j}
+b\sum_{j=n-a+1}^{n-3}\sum_{i=1}^{n-a}\v_{i,j}
+c\sum_{i=1}^{n-a}\v_{i,n-2}+
d\sum_{i=1}^{n-a}\v_{i,n-1}
\\
&+e\sum_{A\subseteq [n-a+1,n-3]}\v_A
+f\sum_{i=n-a+1}^{n-3}\v_{i,n-2}
+g\sum_{i=n-a+1}^{n-3}\v_{i,n-1}.
\end{align*}
Acting with $(1,2)(n-1,n)$, we get that $d=g=0$. Acting with $(1,2)(n-2,n-1)$, we get $c=f=0$. If $a=3$, then the sums with coefficients $b$ and $e$ are empty. Otherwise, acting with $(1,2)(n-3,n-2)$ yields $b=e=0$. Now using the permutation which swaps the last two blocks $B_b$ and $B_{b-1}$ (possibly multiplied with $(1,2)$), we get $a=0$. 
\end{proof}

\section{Branching lemmas}

In this section we prove some technical lemmas on branching for symmetric groups. 

\subsection{Some special composition factors}
\label{SSSCF}
In this subsection we prove some technical results concerning special composition factors in $D^\la\da_{\SSS_k}$. 

 \begin{Lemma}\label{L110518}
Let $n,m\in\Z_{>0}$, and $\la=(a_1^{b_1},\ldots,a_k^{b_k})\in\Par_p(n+m)$ with $a_1>\ldots>a_k>0$ and $b_1,\dots,b_k>0$. 
Set $h_r:=b_1+\ldots+b_r$ for $0\leq r\leq k$. 
Suppose that there is $1\leq j\leq k$ and a composition $\nu=(\nu_1,\ldots,\nu_{h_j})$ of $n$ such that $\mu:=\la-\nu$ is a $p$-regular partition of $m$ and 
$(\la_1,\ldots,\la_{h_j})$ is JS. 
If 
$$
\nu_{h_1}\geq\dots\geq \nu_1\geq 
\nu_{h_2}\geq\dots\geq \nu_{h_1+1}\geq \dots\geq \nu_{h_j}\geq \dots\geq \nu_{h_{j-1}+1}.
$$
and 
$$
\nu_{h_r}\leq\nu_{h_{r-1}+1}+p-b_r\qquad(\text{for all}\ 1\leq r\leq j),
$$
then $D^\mu$ is a composition factor of $D^\la\da_{\SSS_m}$.
\end{Lemma}
\begin{proof}
\iffalse{
If $\nu_{h_{j-1}+1}>0$ then $\nu=\overline{\nu}+(1,\ldots,1)$ with $\overline{\nu}=(\nu_1-1,\ldots,\nu_{h_j}-1)$. Notice that $\overline{\mu}=\la-(1^{h_j})$ is $p$-regular, since $\mu$ is and that $\mu=\overline{\mu}-\overline{\nu}$. Further $(\overline{\mu}_1,\ldots,\overline{\mu}_{h_j})$ is JS and either $\overline{\mu}_{h_j}>\overline{\mu}_{h_j+1}$ or $\overline{\nu}_r=0$ for $r>h_{j-1}$.

If $\nu_{h_{j-1}+1}=0$ then $\nu=\overline{\nu}+(\nu_{h_j}^{h_{j-1}},0,\nu_{h_{j-1}+2},\ldots,\nu_{h_j})$ with
$$\overline{\nu}=(\nu_1-\nu_{h_j},\ldots,\nu_{h_{j-1}}-\nu_{h_j},0,\ldots,0)=(\nu_1-\nu_{h_j},\ldots,\nu_{h_{j-1}}-\nu_{h_j}).$$
Notice that $\overline{\mu}=\la-(\nu_{h_j}^{h_{j-1}},0,\nu_{h_{j-1}+2},\ldots,\nu_{h_j})$ is $p$-regular, since $\mu$ is and that $\mu=\overline{\mu}-\overline{\nu}$. Further $(\overline{\mu}_1,\ldots,\overline{\mu}_{h_{j-1}})$ is JS and either $\overline{\mu}_{h_{j-1}}>\overline{\mu}_{h_{j-1}+1}$ or $\overline{\nu}_r=0$ for $r>h_{j-2}$.

So by induction on $n$ it is enough to prove the lemma for $\nu=(1,\dots,1)$ and for $\nu=(\nu_{h_j}^{h_j-1},0,\nu_{h_{j-1}+2},\ldots,\nu_{h_j})$. 
}\fi
In view of Lemma \ref{Lemma39}(v), it suffices to see that there exists a sequence $A_1,\dots,A_n$ of nodes such that $A_r$ is normal for $\la\setminus\{A_1,\dots,A_{r-1}\}$ and $\la\setminus\{A_1,\dots,A_{r}\}$ is $p$-regular for $r=1,\dots n$, and $\la\setminus\{A_1,\dots,A_{n}\}=\mu$. 
Such a sequence is obtained by  removing the nodes of $\la\setminus\mu$ in the ends of rows 
$$h_1,h_1-1,\dots,1,h_2,h_2-1,\dots,h_1+1,\dots,h_j,\dots,h_{j-1}+1$$
if there are any, then starting over in the row $h_1$ and proceeding in the same order until all the nodes of $\la\setminus\mu$ are exhausted.
\end{proof}

\begin{Remark}
For $p=2$ the assumptions on $\nu$ in Lemma \ref{L110518} are equivalent to $\nu\in\Par(n)$, and so Lemma~\ref{L110518} generalizes \cite[Lemma 3.14]{KMTOne}. 
\end{Remark}

Let $p=2$. Then $\la\in\Par_2(n)$ is JS if and only if all its parts $\la_r$ are of the same parity. We now describe a procedure which assigns to every partition $\la\in\Par_2(n)$ a JS-partition $\la^\JS\in\Par_2(m)$ for $m\leq n$. If $\la$ is already JS then $\la^\JS$ will return $\la$. We begin by setting $\la^1:=\la$. For $r\geq 1$, as long as $\la^r_{r+1}>0$ define $\la^{r+1}$ as follows. If $\la^r_{r+1}\equiv \la^r_r\pmod{2}$ then $\la^{r+1}:=\la^r$. If instead $\la^r_{r+1}\not\equiv \la^r_r\pmod{2}$ let $l\geq r+1$ be minimal such that $\la^r_{l+1}>\la^r_l+1$ or such that $\la^r_{l+1}=0$ and define
\[\la^{r+1}:=(\la^r_1,\ldots,\la^r_r,\la^r_{r+1}-1,\ldots,\la^r_l-1,\la^r_{l+1},\la^r_{l+2},\ldots).\]
Let $s$ be minimal with $\la^s_{s+1}=0$. Take $\la^\JS:=\la^s$.

\begin{Lemma} \label{LJS} %{\rm \cite{}}%{\bf ()}
Let $p=2$ and $\la\in\Par_2(n)$. Then $\la^\JS$ is a $2$-regular JS partition of $m\leq n$. Moreover, denoting $h:=h$, we have:
\begin{enumerate}
\item[{\rm (i)}] $\la^\JS=\la$ if and only if $\la$ is JS. 
\item[{\rm (ii)}] $D^{\la^\JS}$ is a composition factor of $D^\la\da_{\SSS_m}$.
\item[{\rm (iii)}] $\la^\JS_{j}-\la^\JS_{j+1}\leq 2\lceil (\la_j-\la_{j+1})/2\rceil$ for each $j\geq 1$.
\item[{\rm (iv)}] $0\leq\la_j-\la^\JS_j\leq j-1$. In particular $\la_{h+1}\leq h$.
\item[{\rm (v)}] if $k$ is maximal such that $\la_{2k-1}>0$ then $|\la^\JS|\geq(n+k)/2$.
\item[{\rm (vi)}] if $\la^\JS_{h}\geq 3$ then $\la_{h+1}\leq 1$ and if $\la_{h+1}=1$ then $\la_1$ is even. 
\item[{\rm (vii)}] if $\la^\JS_{h}=2$ and $\la^\JS_{h-1}\geq 6$ then $\la_{h}\leq 3$.
\item[{\rm (viii)}] if $\la^\JS_{h}=1$ and $\la^\JS_{h-1}\geq 5$ then $\la_{h}\leq 2$.
\end{enumerate} 
\end{Lemma}
\begin{proof}
Let $\la=\la^1,\dots,\la^s=\la^\JS$ be as in the construction. Then for all $r=1,\dots,s$, we have that $\la^r$ is $2$-regular, $(\la^r_1,\ldots,\la^r_r)$ is JS. Moreover, to go from $\la^r$ to $\la^{r+1}$ we remove normal nodes on each step. So by Lemma~\ref{Lemma39}, $D^{\la^{r+1}}$ is a composition factor of $D^{\la^r}\da_{\SSS_{|\la^{r+1}|}}$. The statement that $\la^\JS$ is a $2$-regular JS partition as well as (i) and (ii) follow by induction, while (iv) holds by construction. 

(iii) Notice that $\la^{r+1}_j-\la^{r+1}_{j+1}\leq \la^r_j-\la^r_{j+1}$ unless $j=r$ and $\la^r_j-\la^r_{j+1}$ is odd, in which case $\la^{r+1}_j-\la^{r+1}_{j+1}=\la^r_j-\la^r_{j+1}+1$.

(v) Using (iv) we have that 
$\la^\JS_i\geq\la_i-i+1\geq\la_{2i-1}$ for all $i$. So, by definition of $k$,
 \[|\la^\JS|=\sum_i\la^\JS_{i}\geq \sum_i\la_{2i-1}\geq \sum_{i=1}^k(\la_{2i-1}+\la_{2i}+1)/2=(n+k)/2.\]

(vi) Note that $\la^h_h=\la^\JS_h\geq 3$ and $\la^\JS_{h+1}=0$. If some node had been removed from row $h+1$ of $\la$ to obtain $\la^h$ then $\la^h_{h+1}=\la^h_h-1\geq 2$ and then $\la^\JS_{h+1}\geq 1$, leading to a contradiction. So no node was removed from row $h+1$ of $\la$ and then $\la^\JS_{h+1}=0$ implies $\la_{h+1}\leq 1$. Further if $\la_{h+1}=1$ then since $\la^\JS_{h+1}=0$ the node $(h+1,1)$ needs to be removed on step $h+1$ of the construction, so $(\la^h_1,\dots,\la^h_h)$ is JS, while $(\la^h_1,\dots,\la^h_h,1)$ is not. 
Hence $\la_1=\la^h_1$ is even.

(vii) and (viii) are proved similarly to (vi).
\end{proof}

\begin{Lemma}\label{l1}
Let $n\geq12$ be even, $p=2$ and $\la\in\Parinv_2(n)$. Exclude the cases where $\la$ is the double of one of the following partitions:
\begin{align*}
&(11, 1),\, (9,3),\,(9,5),\,(11, 5),\,( 11, 7 ),\,(13, 8, 3),\,( 13, 9, 4),\,
( 13, 9, 5, 1),\,(15, 11, 5, 1),
\\
&(15, 11, 7, 1 ),\,(15, 11, 7, 3),\,( 17, 13, 9, 3),\,(17, 13, 9, 5),\,
(19, 15, 11, 7),\,(21, 17, 13, 9, 4),
\\
&(21, 17, 13, 9, 5, 1 ),\, (23, 19, 15, 11, 7, 1 ),\, (23, 19, 15, 11, 7, 3),\, (25, 21, 17, 13, 9, 5),
\\
&(29, 25, 21, 17, 13, 9, 5, 1 ),\,( 31, 27, 23, 19, 15, 11, 7, 3).
\end{align*}
Then $|\la^\JS|\geq n/2+5$ and $\la^\JS_{2j-1}-\la^\JS_{2j}\leq 2$ for each $j\geq 1$.
\end{Lemma}

\begin{proof}
The inequalities $\la^\JS_{2j-1}-\la^\JS_{2j}\leq 2$ come from the assumption that $\la\in\Parinv_2(n)$ and Lemma~\ref{LJS}(iii). 

If $h(\la)\geq 17$ then $|\la^\JS|\geq n/2+5$ by Lemma \ref{LJS}(v). Moreover, by Lemma \ref{LJS}(iv), 
\[n-|\la^\JS|\leq \sum_{i=1}^{h(\la)}(i-1)=h(\la)(h(\la)-1)/2.\]
So the lemma holds  if $h(\la)(h(\la)-1)/2\leq n/2-5$. 
So we may assume that $h(\la)\leq 16$ and $h(\la)(h(\la)-1)/2> n/2-5$. This leaves only a finite number of partitions to be considered. For these it can be checked, using GAP \cite{GAP}, that $|\la^\JS|\geq n/2+5$, unless we are in one of the exceptional cases.
\end{proof}

\subsection{Non-isomorphic composition factors} 
In this subsection we obtain various technical results which guarantee the presence of several non-isomorphic composition factors in the restriction $D^\la\da_{\SSS_k}$. 
 
\begin{Lemma}\label{L240818_2}
Let $p=3$, $n$ be even and $\la\in\Parinv_3(n)$ be a JS partition with $\la\not=(4,1,1)$. Then $D^\la\da_{\SSS_{n/2}}$ has at least $5$ non-isomorphic composition factors.
\end{Lemma}
\begin{proof}
By Lemma~\ref{L110518_2}, $\la$ is belongs to one of the families (i)-(iv) of that lemma. By the assumptions the only partitions from the family (iv) that need to be considered are $(7,3,2)$, $(10,4^2)$, $(13,6,5)$, $(10,6,3^2,1^2)$ and $(13,6,5,4,1^2)$. For  
 $\la=(7,3,2)$ the lemma can be checked using \cite[Tables]{JamesBook} and branching in characteristic $0$. For the remaining ones, the lemma can be proved using Lemma \ref{L110518}. So we may assume that we are in one of the cases  (i), (ii), (iii) of Lemma \ref{L110518_2}. Let $\la=(a_1^{b_1},\ldots,a_m^{b_m})$ with $a_1>\ldots>a_m>0$ and all  $b_j>0$. Define $h_j:=b_1+\ldots+b_j$. % (notice that $h_m=h(\la)$).

Set $\la^0:=\la$ and then recursively define $\la^i:=(\la^{i-1}_1-1,\ldots,\la^{i-1}_{h(\la^{i-1})}-1)$. %, so that $\la^i$ is obtained from $\la^{i-1}$ by removing the first column. 
By Lemma \ref{L110518}, $D^{\la^i}$ is a composition factor of $D^\la\da_{\SSS_{|\la^i|}}$ for all $i$. Let $k$ maximal such that $|\la^k|\geq n/2$. Let $\nu:=\la-\la^k$. For any composition $\al=(\al_1,\ldots,\al_{h_m})$ define $\overline{\al}:=(\al_{h_1},\ldots,\al_1,\ldots,\al_{h_m},\ldots,\al_{h_{m-1}+1})$, and let $\nu^1=\nu+\overline{(1^{n/2-|\nu|})}$. Note that $\overline{\nu^1}\in\Par(n/2)$. % and if $\la_j=\la_{j+1}$ and $\nu^1_j<\nu^1_{j+1}$ then $\nu^1_j=\nu^1_{j+1}-1$.
%If $|\la^k|=n/2$ set $\mu:=\la^k$. If $|\la^k|>n/2$ and $\la^k_{n/2-|\la^k|}>\la^k_{n/2-|\la^k|+1}$ set $\mu:=\la^k-(1^{n/2-|\la^k|})$. If $|\la^k|>n/2$ and $\la^k_{n/2-|\la^k|}=\la^k_{n/2-|\la^k|+1}$ then $\la^k_{n/2-|\la^k|-1}>\la^k_{n/2-|\la^k|}$. In this case set $\mu:=\la^k-(2,1^{n/2-|\la^k|-2})$. Then $\mu$ is $3$-regular, since $\la_1\geq\la_2+4$ and since if $\la_j=\la_{j+1}+1>1$ then $\la_j$ and $\la_{j+1}$ both have multiplicity $1$ in $\la$ (this follows from $\la$ being JS). So from Lemma \ref{L110518} we have that $D^\mu$ is a composition factor of $D^\la\da_{\SSS_{n/2}}$. Also let $\nu^1:=\la-\mu$. Notice that $\nu^1\in\Par(n/2)$.
By Lemma~\ref{L110518}, $D^{\la-\nu^1}$ is a composition factor of $D^\la\da_{\SSS_{n/2}}$. We will now construct $\nu^2,\dots,\nu^5$ and apply the same lemma to see that $D^{\la-\nu^1},\dots,D^{\la-\nu^5}$ are distinct composition factors of $D^\la\da_{\SSS_{n/2}}$.

%We consider cases corresponding to cases (i) and (ii) or (iii) of Lemma \ref{L110518_2}. It could be that we are in more than one case, but at least one of them must hold.

{\sf Case 1.} $\la$ is as in case (i) of Lemma \ref{L110518_2}. Note that $b_1=1$ and $k\geq 2$ since $n/2\geq 2h(\la)$. 
Now, $k\geq 2$ and $\la_3\geq 7$ imply $\nu^1_1,\nu^1_2,\nu^1_3\geq 2$. As $\la_1\leq n/2+1$ we have $\la_1-\nu^1_1\leq n/2+1-2<n/2$, so $\la_2-\nu^1_2 >0$.
We write $\nu^1$ in the `canonical' form $\overline{\nu^1}=(c_{1,1},c_{1,2}^{d_{1,2}},c_{1,3}^{d_{1,3}},\ldots)$ with $c_{1,1}\geq c_{1,2}>c_{1,3}>\ldots$ and $d_{1,j}>0$, and then proceed to define 
$\overline{\nu^2},\dots,\overline{\nu^5}$ in the `canonical' form 
$\overline{\nu^i}=(c_{i,1},c_{i,2}^{d_{i,2}},c_{i,3}^{d_{i,3}},\ldots)$ recurrently according to the cases.

{\sf Case 1.1.} $d_{1,2}\geq 2$ or $c_{1,2}=c_{1,3}+1$. For $1\leq i<5$ define $\nu^{i+1}$ recurrently by setting $\overline{\nu^{i+1}}:=(c_{i,1}+1,c_{i,2}^{d_{i,2}-1},c_{i,2}-1,c_{i,3}^{d_{i,3}},c_{i,4}^{d_{i,4}},\ldots)$ (this form is not necessarily `canonical', so we might have to rewrite into the `canonical' form before the next recurrent step).

%It then follows that,  for $1\leq j\leq 5$, $\overline{\nu^j}\in\Par(n/2)$ and  if $\la_j=\la_{j+1}$ and $\nu^i_j<\nu^i_{j+1}$ then $\nu^i_j=\nu^i_{j+1}-1$. We have that $(\la-\nu^i)_{d_{i,2}+d_{i,3}+2}=0$ and if $c_{i,2}>c_{i,3}+1$ also that $(\la-\nu^i)_{d_{i,2}+2}=0$. From $d_{1,2}\geq 2$ or $c_{1,2}=c_{1,3}+1$. we have that $\nu^i_2\geq\nu^5_2\geq\nu^1_2-2\geq \nu^1_1-3=\nu^i_1-2-i$. Since $\la$ is JS, so that if $\la_j=\la_{j+1}+1>1$ then $\la_j$ and $\la_{j+1}$ both have multiplicity $1$ in $\la$ and since $\la_1\geq\la_2 +9$, it follows that $\la-\nu^i\in\Par_3(n/2)$ for $1\leq i\leq 5$ and so the lemma holds by Lemma \ref{L110518}.

{\sf Case 1.2.} $d_{1,2}=1$, $c_{1,2}>c_{1,3}+1$, and either of the following conditions holds: $d_{1,3}\geq 2$, $c_{1,3}=c_{1,4}+1$. In this case let $\nu^2:=\nu^1+(1,-1,0,0,\ldots)$, $\nu^3:=\nu^2+(1,0,-1,0,0,\ldots)$, $\nu^4:=\nu^3+(1,0,0,-1,0,0,\ldots)$ and $\nu^5:=\nu^4+(1,-1,0,0,\ldots)$. %The assumptions give that,  for $1\leq j\leq 5$, $\overline{\nu^j}\in\Par(n/2)$ and  if $\la_j=\la_{j+1}$ and $\nu^i_j<\nu^i_{j+1}$ then $\nu^i_j=\nu^i_{j+1}-1$. Further from $d_{1,2}=1$ and $c_{1,2}>c_{1,3}+1$ we have that $(\la-\nu^1)_3=0$. Since $(\la-\nu^1)_2>0$ we have that $\nu^1_1\leq \nu^1_2+1$ and so from $\la_1\geq\la_2+9$ it follows that $\la-\nu^i\in\Par_3(n/2)$ for each $1\leq i\leq 5$ and so we can conclude by  Lemma \ref{L110518}.

{\sf Case 1.3.} $d_{1,2}=d_{1,3}=1$, $c_{1,2}>c_{1,3}+1>c_{1,4}+2$. In this case let $\nu^2:=\nu^1+(1,-1,0,0,\ldots)$, $\nu^3:=\nu^2+(1,0,-1,0,0,\ldots)$, $\nu^4:=\nu^3+(1,-1,0,0,\ldots)$ and $\nu^5:=\nu^4+(1,0,-1,0,0,\ldots)$.  %The assumptions give that,  for $1\leq j\leq 5$, $\overline{\nu^j}\in\Par(n/2)$ and  if $\la_j=\la_{j+1}$ and $\nu^i_j<\nu^i_{j+1}$ then $\nu^i_j=\nu^i_{j+1}-1$. Further from $d_{1,2}=1$ and $c_{1,2}>c_{1,3}+1$ we have that $(\la-\nu^1)_3=0$. Since $(\la-\nu^1)_2>0$, $\nu^1_1\leq \nu^1_2+1$ and $\la_1\geq\la_2+9$, it follows that $\la-\nu^i\in\Par_3(n/2)$ for each $1\leq i\leq 5$ and so we can conclude by  Lemma \ref{L110518}.

{\sf Case 2.} $\la$ is as in case (ii) or (iii) of Lemma \ref{L110518_2}. Then $b_1=b_2=1$, and since $n/2\geq 3h(\la)$, we have $k\geq 3$. Now, $h(\la)\geq 6$, $\la_4\geq 4$ and $k\geq 3$ imply that $\nu_1,\nu_2,\nu_3,\nu_4\geq 3$. As $\la_1+\la_2\leq n/2+4$ it then follows that $\la_3-\nu^1_3>0$.
We will be writing $\overline{\nu^i}$ in the `canonical' form
$\overline{\nu^i}=(c_{i,1},c_{i,2},c_{i,3}^{d_{i,3}},c_{i,4}^{d_{i,4}},\ldots)$ with $c_{i,1}\geq c_{i,2}\geq c_{i,3}>c_{i,4}>\ldots$ and all $d_{i,j}>0$.

{\sf Case 2.1.} $d_{1,3}\geq 2$ or $c_{1,3}=c_{1,4}+1$. For $1\leq i<5$ define 
$$
\overline{\nu^{i+1}}:=
\left\{
\begin{array}{ll}
(c_{i,1}+1,c_{i,2},c_{i,3}^{d_{i,3}-1},c_{i,3}-1,c_{i,4}^{d_{i,4}},c_{i,5}^{d_{i,5}},\ldots) &\hbox{if $\la_1-c_{i,1}-1\geq\la_2-c_{i,2}$,}\\
(c_{i,1},c_{i,2}+1,c_{i,3}^{d_{i,3}-1},c_{i,3}-1,c_{i,4}^{d_{i,4}},c_{i,5}^{d_{i,5}},\ldots) &\hbox{if $\la_1-c_{i,1}-1<\la_2-c_{i,2}$.}
\end{array}
\right.
$$

%If $\la_1\geq\la_2+7\geq\la_3+10$ then $\nu^i_2=\nu^1_2$ and, from $d_{1,3}\geq 2$ or $c_{1,3}=c_{1,4}+1$, $\nu^i_3\geq \nu^5_3\geq\nu^1_3+2$. If $\la_1\geq\la_2+4\geq\la_3+8$ then $\nu^i_2\leq \nu^1_2+1$ and, from $d_{1,3}\geq 2$ or $c_{1,3}=c_{1,4}+1$, $\nu^i_3\geq \nu^5_3\geq\nu^1_3-2$. In either case $(\la-\nu^i)_1\geq(\la-\nu^i)_2>(\la-\nu^i)_3$ for $1\leq i\leq 5$.

%Further by definition,  for $1\leq j\leq 5$, $\overline{\nu^j}\in\Par(n/2)$ and  if $\la_j=\la_{j+1}$ and $\nu^i_j<\nu^i_{j+1}$ then $\nu^i_j=\nu^i_{j+1}-1$. We have that $(\la-\nu^i)_{d_{i,3}+d_{i,4}+3}=0$ and if $c_{i,3}>c_{i,4}+1$ also that $(\la-\nu^i)_{d_{i,3}+3}=0$. Further $\la_1-c_{i,1}\geq \la_2-c_{2,i}$ and $c_{i,2}\leq c_{i,3}+i+1$. Since $\la$ is JS, so that if $\la_j=\la_{j+1}+1>1$ then $\la_j$ and $\la_{j+1}$ both have multiplicity $1$ in $\la$ and since $\la_1\geq\la_2 +4\geq\la_3+7$, it follows that $\la-\nu^i\in\Par_3(n/2)$ for $1\leq i\leq 5$ and so the lemma holds by Lemma \ref{L110518}.

{\sf Case 2.2.} $d_{1,3}=1$, $c_{1,3}>c_{1,4}+1$, and either of the following conditions holds: $d_{1,4}\geq 2$, $c_{1,4}=c_{1,5}+1$. In this case let $\nu^2:=\nu^1+(1,0,-1,0,0,\ldots)$, $\nu^3:=\nu^2+(1,0,0,-1,0,0,\ldots)$, $\nu^4:=\nu^3+(1,0,0,0,-1,0,0,\ldots)$ and $\nu^5:=\nu^4+(1,0,-1,0,0,\ldots)$ if $\la_1-\la_2\geq 5$ or $\nu^5:=\nu^4+(0,1,-1,0,0,\ldots)$ if $\la_1-\la_2=4$. %The assumptions give that, for $1\leq j\leq 5$, $\overline{\nu^j}\in\Par(n/2)$ and  if $\la_j=\la_{j+1}$ and $\nu^i_j<\nu^i_{j+1}$ then $\nu^i_j=\nu^i_{j+1}-1$. Further from $d_{1,3}=1$ and $c_{1,3}>c_{1,4}+1$ we have that $(\la-\nu^1)_4=0$. Since $(\la-\nu^1)_3>0$ we have that $\nu^1_2\leq \nu^1_3+1$. As $\la_2\geq\la_3+3$, it follows that $\la-\nu^i\in\Par_3(n/2)$ for each $1\leq i\leq 5$ and so we can conclude by  Lemma \ref{L110518}.

{\sf Case 2.3.} $d_{1,3}=d_{1,4}=1$, $c_{1,3}>c_{1,4}+1>c_{1,5}+2$. In this case let $\nu^2:=\nu^1+(1,0,-1,0,0,\ldots)$, $\nu^3:=\nu^2+(1,0,0,-1,0,0,\ldots)$, $\nu^4:=\nu^3+(1,0,-1,0,0,\ldots)$ and $\nu^5:=\nu^4+(1,0,0,-1,0,0,\ldots)$ if $\la_1-\la_2\geq 5$ or $\nu^5:=\nu^4+(0,1,0-1,0,0,\ldots)$ if $\la_1-\la_2=4$.  %The assumptions give that,  for $1\leq j\leq 5$, $\overline{\nu^j}\in\Par(n/2)$ and  if $\la_j=\la_{j+1}$ and $\nu^i_j<\nu^i_{j+1}$ then $\nu^i_j=\nu^i_{j+1}-1$. Further from $d_{1,3}=1$ and $c_{1,3}>c_{1,4}+1$ we have that $(\la-\nu^1)_4=0$. Since $(\la-\nu^1)_3>0$ and then $\nu^1_2\leq \nu^1_3+1$ and since $\la_2\geq\la_3+3$, it follows that $\la-\nu^i\in\Par_3(n/2)$ for each $1\leq i\leq 5$ and so we can conclude by  Lemma \ref{L110518}.
\end{proof}

\begin{Lemma}\label{l25}
Let $p=2$, $\la\in\Par_2(n)$ be a JS-partition and $5\leq k\leq n/2$.  Then $D^\la\da_{\SSS_{n-k}}$ has at least three non-isomorphic composition factors, unless possibly one of the following holds:
\begin{itemize}
\item
$\la=(n)$

\item
$n$ is even and $\la=(n-1,1)$,

\item
$n$ is even and $\la=(n/2+2,n/2-2)$,

\item
$n$ is even and $\la=(n/2+1,n/2-1)$,

\item
$n$ is odd and $\la=((n+1)/2,(n-3)/2,1)$,

\item
$n\equiv 0\pmod{3}$ and $\la=(n/3+2,n/3,n/3-2)$,

\item
$n\geq 14$ with $n\equiv 2\pmod{3}$, $\la=((n-2)/3+4,(n-2)/3,(n-2)/3-2)$ and $k\equiv 1\pmod{3}$,

\item
$n\geq 19$ with $n\equiv 1\pmod{3}$, $\la=((n+2)/3+2,(n+2)/3,(n+2)/3-4)$ and $k\equiv 2\pmod{3}$,

\item
$h(\la)=3$, $\la_1=\la_2+2$, $\la_2\geq\la_3+4$ and $k=5$,

\item
$n\geq 22$ with $n\equiv 4\pmod{6}$, $\la=((n-1)/3+2,(n-1)/3,(n-1)/3-2,1)$ and $k\not\equiv 0\pmod{3}$,

\item
$\la=(\la_1,\la_1-2,\la_1-4,\la_4)$ and $k=5$,

\item
$n\geq 20$ with $n\equiv 0\pmod{4}$ and $\la=(n/4+3,n/4+1,n/4-1,n/4-3)$.
\end{itemize}
\end{Lemma}

\begin{proof}
Let $\la^0:=\la$ and then recursively define $\la^j:=(\la^{j-1}_1-1,\ldots,\la^{j-1}_{h(\la^{j-i})}-1)$. Note that $\la^j$ is JS. 
Let $a_j:=|\la^j|-n+k$.  Since $|\la^j|-a_j=n-k\geq n/2$, we have $n\geq 2(n-|\la^j|+a_j)$. Moreover, $|\la^j|/2\leq n/2\leq n-k=|\la^j|-a_j$ implies $a_j\leq |\la^j|/2$. There is a unique $i$ with $5\leq a_i\leq h(\la^i)+4$.

By Lemma \ref{L110518}, $D^{\la^j}$ is a composition factor of $D^\la\da_{\SSS_{|\la^j|}}$, so it suffices to show that for some $j$ such that $|\la^j|\geq n-k$ 
there exist distinct composition factors $D^\mu,D^\nu,D^\pi$ of $D^{\la^j}\da_{\SSS_{n-k}}$. We always assume that we are not in one of the excluded cases. We will repeatedly apply  Lemma \ref{L110518} without referring to it. We denote by $\de_m$ the composition $(0,\dots,0,1,0,\dots,0)$ with $1$ in the $m$th position. 

{\sf Case 1.} $a_i\leq h(\la^i)$. Then, using the fact that $a_i\geq 5$, we can take $\mu=\la^i-\de_1-\dots-\de_{a_i}$, $\nu=\la^i-2\de_1-\de_2-\dots-\de_{a_i-1}$,  and $\pi=\la^i-2\de_1-2\de_2-\de_3-\dots-\de_{a_i-2}$.

{\sf Case 2.} $a_i=h(\la^i)+1\geq 6$. Then we can take $\mu=\la^i-2\de_1-\de_2-\dots-\de_{a_i-1}$, $\nu=\la^i-2\de_1-2\de_2-\de_3-\dots-\de_{a_i-2}$ and $\pi=\la^i-3\de_1-2\de_2-\de_3-\dots-\de_{a_i-3}$.

{\sf Case 3.} $a_i=h(\la^i)+2\geq 7$. Then we can take $\mu=\la^i-2\de_1-2\de_2-\de_3-\dots-\de_{a_i-2}$, $\nu=\la^i-2\de_1-2\de_2-2\de_3-\de_4-\dots-\de_{a_i-3}$ and $\pi=\la^i-3\de_1-2\de_2-\de_3-\dots-\de_{a_i-3}$.

{\sf Case 4.} $a_i=h(\la^i)+3\geq 8$. Then we can take $\mu=\la^i-2\de_1-2\de_2-2\de_3-\de_4-\dots-\de_{a_i-3}$, $\nu=\la^i-3\de_1-2\de_2-2\de_3-\de_4-\dots-\de_{a_i-4}$ and $\pi=\la^i-3\de_1-2\de_2-\de_3-\dots-\de_{a_i-3}$.

{\sf Case 5.} $a_i=h(\la^i)+4\geq 9$. Then we can take $\mu=\la^i-2\de_1-2\de_2-2\de_3-2\de_4-\de_5-\dots-\de_{a_i-4}$, $\nu=\la^i-3\de_1-3\de_2-2\de_3-\de_4-\dots-\de_{a_i-5}$ and $\pi=\la^i-3\de_1-2\de_2-2\de_3-\de_4-\dots-\de_{a_i-4}$.

{\sf Case 6.} $h(\la^i)\leq 4$. Note that %From $|\la^i|\geq n-k+5$ and $5\leq k\leq n/2$ we have that 
$|\la^i|\geq n/2+5\geq 10$.

{\sf Case 6.1.} $h(\la^i)=1$. %Then $a_i=5$ and $\la^i=(\la^i_1)$ with $\la^i_1\geq 10$. 
By assumption $\la\not=(n),(n-1,1)$. Since $\la$ is JS, there exists $0\leq j\leq i-2$ with $\la^j=(\la^j_1,2)$. From $j<i$ we have that $a_j>a_i=5$. So we can take $\mu=(\la^j_1-(a_j-2))$, $\nu=(\la^j_1-(a_j-1),1)$ and $\pi=(\la^j_1-a_j,2)$ (notice that $\la^j_1-a_j>2$ since $n-k\geq n/2\geq 5$).

{\sf Case 6.2.} $h(\la^i)=2$. In this case $a_i=5$ or $6$. %Then $5\leq a_j\leq 6$. 

{\sf Case 6.2.1.} $\la^i_2=1$. Then $i\geq 1$ and so $\la^{i-1}=(\la^i_1+1,2)$, since $\la\not=(n-1,1)$ is JS. So we can take $\mu=(\la^i_1-(a_i-1))$, $\nu=(\la^i_1-a_i,1)$ and $\pi=(\la^i_1-(a_i+1),2)$.% (notice that $\la^i_1-(a_i+1)>2$ since $n-k\geq n/2\geq 5$).

{\sf Case 6.2.2.} $\la^i_2=2$. Then we can take $\mu=(\la^i_1-(a_i-2))$, $\nu=(\la^i_1-(a_i-1),1)$ and $\pi=(\la^i_1-a_i,2)$.% (notice that $\la^i_1-a_i>2$ since $n-k\geq n/2\geq 5$).

{\sf Case 6.2.3.} $\la^i_2\geq 3$ and $\la^i_1\geq\la^i_2+6$. Then we can take $\mu=(\la^i_1-\lceil a_i/2\rceil,\la^i_2-\lfloor a_i/2\rfloor)$, $\nu=(\la^i_1-\lceil a_i/2\rceil-1,\la^i_2-\lfloor a_i/2\rfloor+1)$ and $\pi=(\la^i_1-\lceil a_i/2\rceil-2,\la^i_2-\lfloor a_i/2\rfloor+2)$.

{\sf Case 6.2.4.} $\la^i_2\geq 3$ and $\la^i_1=\la^i_2+4$ (note that we cannot have $\la^i_1=\la^i_2+5$ since $\la^i$ is JS). 
Let $\mu=(\la^i_1-\lceil a_i/2\rceil,\la^i_2-\lfloor a_i/2\rfloor)$, $\nu=(\la^i_1-\lceil a_i/2\rceil-1,\la^i_2-\lfloor a_i/2\rfloor+1)$. To get the third composition factor, note that 
by assumption $\la\not=(n/2+2,n/2-2)$, so there exists $0\leq j<i$ with $\la^j=(\la^j_1,\la^j_1-4,1)$. In particular $n\geq |\la^i|+3\geq 13$ and then $n-k\geq 7$. Now we take $\pi=(\la^j_1-2-\lceil (a_j-2)/2\rceil,\la^j_1-4-\lfloor (a_j-2)/2\rfloor,1)$ (note that $\la^j_1-4-\lfloor (a_j-2)/2\rfloor>1$ since $2(\la^j_1-4-\lfloor (a_j-2)/2\rfloor)+3=n-k\geq 7$).

{\sf Case 6.2.5.} $\la^i_2\geq 3$ and $\la^i_1=\la^i_2+2$. 
We can take $\mu=(\la^i_1-\lceil a_i/2\rceil,\la^i_2-\lfloor a_i/2\rfloor)$. To get two more composition factors, note that 
since $\la\not=(n/2+1,n/2-1),((n+1)/2,(n-3)/2,1)$ is JS, we have that $i\geq 2$ and there exists $0\leq j\leq i-2$ with $\la^j=(\la^j_1,\la^j_1-2,2)$. In this case $n\geq 2(n-|\la^i|+a_i)\geq 22$ and so $n-k\geq 11$. So we can take $\nu=(\la^j_1-\lceil (a_j-1)/2\rceil,\la^j_1-2-\lfloor (a_j-1)/2\rfloor,1)$ and $\pi=(\la^j_1-\lceil a_j/2\rceil,\la^j_1-2-\lfloor a_j/2\rfloor,2)$ (note that $a_j>a_i\geq 5$ and $\la^j_1-2-\lfloor a_j/2\rfloor>2$ since $2(\la^j_1-2-\lfloor a_j/2\rfloor)+5\geq n-k\geq 11$).

{\sf Case 6.3.} $h(\la^i)=3$. In this case $5\leq a_i\leq 7$.

{\sf Case 6.3.1.} $\la^i_3=1$ and $\la^i_1\geq\la^i_2+6$. Then we can take $\mu=(\la^i_1-\lceil(a_i-1)/2\rceil,\la^i_2-\lfloor (a_i-1)/2\rfloor)$, $\nu=(\la^i_1-\lceil(a_i+1)/2\rceil,\la^i_2-\lfloor (a_i-3)/2\rfloor)$ and $\pi=(\la^i_1-\lceil(a_i+3)/2\rceil,\la^i_2-\lfloor (a_i-5)/2\rfloor)$.

{\sf Case 6.3.2.} $\la^i_3=1$ and $\la^i_1=\la^1_2+4$. Notice that $\la^i_1$ is odd in this case. If $\la^i_1\geq 9$ then $\la^i_2-\lfloor (a_i-2)/2\rfloor\geq\la^i_1-4-2>1$. If $\la^i_1\leq 7$ then $\la^i=(7,3,1)$ and $a_i=5$ and then $\la^i_2-\lfloor (a_i-2)/2\rfloor>1$ again. So we can take $\mu=(\la^i_1-\lceil(a_i-1)/2\rceil,\la^i_2-\lfloor (a_i-1)/2\rfloor)$, $\nu=(\la^i_1-\lceil(a_i+1)/2\rceil,\la^i_2-\lfloor (a_i-3)/2\rfloor)$ and $\pi=(\la^i_1-\lceil (a_i+2)/2\rceil,\la^i_2-\lfloor (a_i-2)/2\rfloor,1)$.

{\sf Case 6.3.3.} $\la^i_3=1$ and $\la^i_1=\la^1_2+2$. Since $|\la^i|\geq 10$ and $\la^i$ is JS, we have $\la^i_1\geq 7$ and so $\la^i_2\geq 5$. We can take $\mu=(\la^i_1-\lceil(a_i-1)/2\rceil,\la^i_2-\lfloor (a_i-1)/2\rfloor)$, $\nu=(\la^i_1-\lceil a_i/2\rceil,\la^i_2-\lfloor a_i/2\rfloor,1)$. To get the third composition factor, 
note that since $\la\not=((n+1)/3,(n-3)/2,1)$, we have $i\geq 1$ and $\la^{i-1}=(\la^i_1+1,\la^i_1-1,2)$. If $\la^i_1\geq 9$ then $\la^i_2-\lfloor (a_i+1)/2\rfloor\geq\la^i_1-2-4>2$. In this case we can take  $\pi=(\la^i_1-\lceil (a_i+1)/2\rceil,\la^i_2-\lfloor (a_i+1)/2\rfloor,2)=(\la^{i-1}_1-\lceil(a_i+3)/2\rceil,\la^{i-1}_2-\lfloor(a_i+3)/2\rfloor,\la^{i-1}_3)$.
If $\la^i_1=7$ then $n\geq|\la^{i-1}|=16$ and $a_{i-1}\geq 8$. It then follows that $\la=\la^{i-1}=(8,6,2)$ and $k=8$. Notice first $D^{(4,3,2)}$ that is a composition factor of $D^{(8,6,2)}\da_{\SSS_9}$. From \cite[Tables]{JamesBook} we have that $[D^{(4,3,2)}]=[S^{(4,3,2)}]-[S^{(8,1)}]$. Using branching rule in characteristic $0$ and \cite[Tables]{JamesBook}, it then follows that $D^{(4,3,1)}$, $D^{(5,3)}$ and $D^{(6,2)}$ are composition factors of $D^{(8,6,2)}\da_{\SSS_8}$.

{\sf Case 6.3.4.} $\la^i_3\geq 2$, $\la^i_1\geq\la^i_2+4$ and $a_i=5$. In this case we can take $\mu=(\la^i_1-2,\la^i_2-2,\la^i_3-1)$, $\nu=(\la^i_1-3,\la^i_2-1,\la^i_3-1)$ and $\pi=(\la^i_1-4,\la^i_2-1,\la^i_3)$.

{\sf Case 6.3.5.} $\la^i_3\geq 2$, $\la^i_1\geq\la^i_2+4$ and $a_i=6$. In this case we can take $\mu=(\la^i_1-2,\la^i_2-2,\la^i_3-2)$, $\nu=(\la^i_1-3,\la^i_2-2,\la^i_3-1)$ and $\pi=(\la^i_1-4,\la^i_2-1,\la^i_3-1)$.

{\sf Case 6.3.6.} $\la^i_3\geq 2$, $\la^i_1\geq\la^i_2+6$ and $a_i=7$. %, $\la^i_2\geq \la^i_3+4$ and $a_i=7$. 
In this case we can take $\mu=(\la^i_1-3,\la^i_2-2,\la^i_3-2)$, $\nu=(\la^i_1-4,\la^i_2-2,\la^i_3-1)$ and $\pi=(\la^i_1-5,\la^i_2-1,\la^i_3-1)$.

{\sf Case 6.3.7.} $\la^i_3\geq 2$, $\la^i_1=\la^i_2+4$, $\la^i_2\geq \la^i_3+4$ and $a_i=7$. In this case we can take $\mu=(\la^i_1-3,\la^i_2-2,\la^i_3-2)$, $\nu=(\la^i_1-4,\la^i_2-2,\la^i_3-1)$ and $\pi=(\la^i_1-3,\la^i_2-3,\la^i_3-1)$.

{\sf Case 6.3.8.} $\la^i_3\geq 2$, $\la^i_1=\la^i_2+4$, $\la^i_2=\la^i_3+2$ and $a_i=7$. If $h(\la)=3$ then $\la=((n-2)/3+4,(n-2)/3,(n-2)/3-2)$ (and $n\equiv 2\pmod{3}$ and $n\geq 14$ since $\la_3\geq\la^i_3\geq 2$) and $k=a_i+3i\equiv 1\pmod{3}$, which is one of the excluded cases. Else there exists $j<i$ with $\la^j=(\la^j_1,\la^j_1-4,\la^j_1-6,1)$. If $\la^i_1\geq 10$ then $\la^j_1=\la^i_1+i-j\geq 10+i-j$ and $a_j=a_i+3(i-j)+1=8+3(i-j)$, so that we can take $\mu=(\la^i_1-3,\la^i_2-2,\la^i_3-2)$, $\nu=(\la^i_1-4,\la^i_2-2,\la^i_3-1)$ and $\pi=(\la^j_1-3-(i-j),\la^j_1-7-(i-j),\la^j_1-8-(i-j),1)$.

If $\la^i_1\leq 9$ then $\la^i=(8,4,2)$ or $\la^i=(9,5,3)$. Since $n\leq 2(n-k)=2(|\la^i|-a_i)$ and $n\geq |\la^j|\geq |\la^i|+4$, this leads to a contradiction.

{\sf Case 6.3.9.} $\la^i_3\geq 2$, $\la^i_1=\la^i_2+2$, $\la^i_2\geq \la^i_3+4$, $a_i=5$ and $h(\la)=3$. In this case $i\geq 1$, as otherwise we are in one of the excluded cases. Further $\la^i_2\geq \la^i_3+6$, since else $n\equiv 1\pmod{3}$, $\la=((n+2)/3+2,(n+2)/3,(n+2)/3-4)$ and $k=a_i+3i=5+3i\equiv 2\pmod{3}$ (this is one of the excluded cases, since $\la_3>\la^i_3\geq 2$, so that $n\geq 19$). So we can take  $\mu=(\la^i_1-2,\la^i_2-2,\la^i_3-1)$, $\nu=(\la^i_1-3,\la^i_2-2,\la^i_3)$ and $\pi=(\la^i_1-3,\la^i_2-3,\la^i_3+1)=(\la^{i-1}_1-4,\la^{i-1}_2-4,\la^{i-1}_3)$.

{\sf Case 6.3.10.} $\la^i_3\geq 2$, $\la^i_1=\la^i_2+2$, $\la^i_2\geq \la^i_3+4$, $a_i=5$ and $h(\la)\geq4$. In this case there exists $0\leq j<i$ with $\la^j=(\la^j_1,\la^j_2,\la^j_3,1)$.  For this $j$ we have that $a_j=a_i+3(i-j)+1=3(i-j)+6$. We can take $\mu=(\la^i_1-2,\la^i_2-2,\la^i_3-1)$, $\nu=(\la^i_1-3,\la^i_2-2,\la^i_3)$ and $\pi=(\la^j_1-(i-j)-3,\la^j_2-(i-j)-3,\la^j_3-(i-j),1)$ (notice that $\la^j_3=\la^i_3+i-j\geq i-j+2$).

{\sf Case 6.3.11.} $\la^i_3\geq 2$, $\la^i_1=\la^i_2+2$, $\la^i_2\geq \la^i_3+4$ and $a_i=6$. Then we can take $\mu=(\la^i_1-2,\la^i_2-2,\la^i_3-2)$, $\nu=(\la^i_1-3,\la^i_2-2,\la^i_3-1)$ and $\pi=(\la^i_1-3,\la^i_2-3,\la^i_3)$.

{\sf Case 6.3.12.} $\la^i_3\geq 2$, $\la^i_1=\la^i_2+2$, $\la^i_2\geq \la^i_3+4$ and $a_i=7$. Then we can take $\mu=(\la^i_1-3,\la^i_2-2,\la^i_3-2)$, $\nu=(\la^i_1-3,\la^i_2-3,\la^i_3-1)$ and $\pi=(\la^i_1-4,\la^i_2-3,\la^i_3)$.

{\sf Case 6.3.13.} $\la^i_3\geq 2$, $\la^i_1=\la^i_2+2$, $\la^i_2=\la^i_3+2$. Then $h(\la)\geq 4$, as otherwise we are in one of the excluded cases. So there exists $0\leq j<i$ with $\la^j=(\la^j_1,\la^j_2,\la^j_3,1)=(\la^j_1,\la^j_1-2,\la^j_1-4,1)$. If $\la^i_3=2$ then $|\la^i|=12$ and $n/2\leq n-k=|\la^i|-a_i\leq 7$. So $n-|\la^i|\leq 2\leq h(\la^i)$ and then $\la=\la^i$, contradicting $h(\la)\geq 4$ and $h(\la^i)=3$. If $\la^i_3=3$ then $|\la^i|=15$ and $n/2\leq n-k=|\la^i|-a_i\leq 10$, so that $n\leq 20$. So $n-|\la^i|\leq 5$ and then $i=j+1=1$. This contradicts $\la$ being JS. So $\la^i_3\geq 4$. Then $\la^j_1=\la^i_3+i-j+4\geq i-j+8$. Further $a_j=a_i+3(i-j)+1$.

{\sf Case 6.3.13.1.} $a_i=5$. Then we can take $\mu=(\la^j_1-(i-j)-2,\la^j_1-(i-j)-4,\la^j_1-(i-j)-5)$, $\nu=(\la^j_1-(i-j)-3,\la^j_1-(i-j)-4,\la^j_1-(i-j)-5,1)$ and $\pi=(\la^j_1-(i-j)-2,\la^j_1-(i-j)-4,\la^j_1-(i-j)-6,1)$.

{\sf Case 6.3.13.2.} $a_i=6$. In this case $j\geq 1$, as otherwise, being $\la^j_3>\la^i_3\geq 4$ and $\la^j_3$ odd, it follows that $n\geq 22$ with $n\equiv 4\pmod{6}$, $\la=((n-1)/3+2,(n-1)/3,(n-1)/3-2,1)$ and $k=a_i+3i+1=3i+7\equiv 1\pmod{3}$. We have that $\la^{j-1}=(\la^j_1+1,\la^j_1-1,\la^j_1-3,2)$ (since $\la$ is a JS-partition) and $a_{j-1}=a_j+4=3(i-j)+11$. If $\la^i_3=4$ then $|\la^i|=18$ and $n/2\leq n-k=|\la^i|-a_i=12$, so that $n\leq 24$. Since $1\leq j<i$ we obtain $|\la^{j-1}|\geq|\la^i|+8>n$, leading to a contradiction. So $\la^i_3\geq 5$ (and then $\la^j_1\geq i-j+9$). In this case we can take $\mu=(\la^j_1-(i-j)-2,\la^j_1-(i-j)-4,\la^j_1-(i-j)-6)$, $\nu=(\la^j_1-(i-j)-3,\la^j_1-(i-j)-4,\la^j_1-(i-j)-6,1)$ and $\pi=(\la^j_1-(i-j)-3,\la^j_1-(i-j)-5,\la^j_1-(i-j)-6,2)=(\la^{j-1}_1-(i-j)-4,\la^{j-1}_2-(i-j)-4,\la^{j-1}_3-(i-j)-3,\la^{j-1}_4)$. 

{\sf Case 6.3.13.3.} $a_i=7$. In this case $j\geq 1$, as otherwise, being $\la^j_3>\la^i_3\geq 4$ and $\la^j_3$ odd, it follows that $n\geq 22$ with $n\equiv 4\pmod{6}$, $\la=((n-1)/3+2,(n-1)/3,(n-1)/3-2,1)$ and $k=a_i+3i+1=3i+7\equiv 2\pmod{3}$. We have that $\la^{j-1}=(\la^j_1+1,\la^j_1-1,\la^j_1-3,2)$ (since $\la$ is a JS-partition) and $a_{j-1}=a_j+4=3(i-j)+12$. If $\la^i_3=4$ then $|\la^i|=18$ and $n/2\leq n-k=|\la^i|-a_i=11$, so that $n\leq 22$. Since $1\leq j<i$ we obtain $|\la^{j-1}|\geq|\la^i|+8>n$, leading to a contradiction. If $\la^i_3=5$ then $|\la^i|=21$ and $n\leq 28$. In this case $|\la^{j-1}|\geq|\la^i|+8>n$, again leading to a contradiction. So $\la^i_3\geq 6$ (and then $\la^j_1\geq i-j+10$). We can take $\mu=(\la^j_1-(i-j)-3,\la^j_1-(i-j)-4,\la^j_1-(i-j)-6)$, $\nu=(\la^j_1-(i-j)-3,\la^j_1-(i-j)-5,\la^j_1-(i-j)-6,1)$ and $\pi=(\la^j_1-(i-j)-3,\la^j_1-(i-j)-5,\la^j_1-(i-j)-7,2)=(\la^{j-1}_1-(i-j)-4,\la^{j-1}_2-(i-j)-4,\la^{j-1}_3-(i-j)-4,\la^{j-1}_4)$. 

{\sf Case 6.4.} $h(\la^i)=4$. In this case $5\leq a_i\leq 8$. 

{\sf Case 6.4.1.} $a_i=5$. Then we can take $\mu=(\la^i_1-2,\la^i_2-1,\la^i_3-1,\la^i_4-1)$, $\nu=(\la^i_1-2,\la^i_2-2,\la^i_3-1,\la^i_4)$.

{\sf Case 6.4.1.1.} $\la^i_1\geq\la^i_2+4$.  Then we can take $\pi=(\la^i_1-3,\la^i_2-1,\la^i_3-1,\la^i_4)$.

{\sf Case 6.4.1.2.} $\la^i_2\geq\la^i_3+4$. Then we can take $\pi=(\la^i_1-3,\la^i_2-2,\la^i_3,\la^i_4)$.

{\sf Case 6.4.1.3.} $\la^i_1=\la^i_2+2$, $\la^i_2=\la^i_3+2$ and $h(\la)=4$. If $i=0$ then $\la=\la^i=(\la^i_1,\la^i_1-2,\la^i_1-4,\la^i_4)$ and $k=a_i=5$, which is one of the excluded cases. So we can assume that $i\geq 1$. If $\la_3=\la_4+2$ then $n\equiv 0\pmod{4}$ and $\la=(n/4+3,n/4+1,n/4-1,n/4-3)$, which is also an excluded case. So we can also assume that $\la_3\geq\la_4+2$. In this case we can take $\pi=(\la^i_1-2,\la^i_1-4,\la^i_1-6,\la^i_4+1)=(\la_1-i-2,\la_2-i-2,\la_3-i-2,\la_4-i+1)$.

{\sf Case 6.4.1.4.} $\la^i_1=\la^i_2+2$, $\la^i_2=\la^i_3+2$ and $h(\la)\geq 5$. In this case there exists $0\leq j<i$ with $\la^j=(\la^j_1,\la^j_2,\la^j_3,\la^j_4,1)=(\la^i_1+i-j,\la^i_1+i-j-2,\la^i_1+i-j-4,\la^i_4+i-j,1)$.

{\sf Case 6.4.1.4.1.} $\la^i_4\geq 3$. Then $\la^j_4\geq i-j+3$, so we can take $\pi=(\la^i_1-2,\la^i_1-4,\la^i_1-5,\la^i_4-1,1)=(\la^j_1-(i-j)-2,\la^j_2-(i-j)-2,\la^j_3-(i-j)-1,\la^j_4-(i-j)-1,\la^j_5)$.

{\sf Case 6.4.1.4.2.} $\la^i_4=2$. Then we can take $\pi=(\la^i_1-3,\la^i_1-4,\la^i_1-5,\la^i_4,1)=(\la^j_1-(i-j)-3,\la^j_2-(i-j)-2,\la^j_3-(i-j)-1,\la^j_4-(i-j),\la^j_5)$ (notice that $\la_1^i\geq\la^i_4+6$ since $\la$ is JS).

{\sf Case 6.4.1.4.3.} $\la^i_4=1$. If $\la^i_4=1$ and $\la^i_1=7$ then $\la^i=(7,5,3,1)$ and then $|\la^i|=16$. In this case $n\leq 22$. Since $\la$ is JS, we have that $j\leq i-2$, so that $n\geq|\la^j|\geq |\la^i|+9=25$, leading to a contradiction. So $\la^i_1\geq 9$ and then  we can take $\pi=(\la^i_1-3,\la^i_1-4,\la^i_1-6,\la^i_4+1,1)=(\la^j_1-(i-j)-3,\la^j_2-(i-j)-2,\la^j_3-(i-j)-2,\la^j_4-(i-j)+1,\la^j_5)$. 

{\sf Case 6.4.2.} $a_i=6$. Then we can take $\mu=(\la^i_1-2,\la^i_2-2,\la^i_3-1,\la^i_4-1)$ and $\nu=(\la^i_1-3,\la^i_2-2,\la^i_3-1,\la^i_4)$.

{\sf Case 6.4.2.1.} $\la^i_1\geq\la^i_2+4$. We can take $\pi=(\la^i_1-4,\la^i_2-1,\la^i_3-1,\la^i_4)$.

{\sf Case 6.4.2.2.} $\la^i_2\geq\la^i_3+4$. We can take $\pi=(\la^i_1-3,\la^i_2-3,\la^i_3,\la^i_4)$.

{\sf Case 6.4.2.3.} $\la^i_3\geq \la^i_4+4$. We can take $\pi=(\la^i_1-2,\la^i_2-2,\la^i_3-2,\la^i_4)$.

{\sf Case 6.4.2.4.} $\la^i_1=\la^i_2+2$, $\la^i_2=\la^i_3+2$ and $\la^i_3=\la^i_4+2$. Then $\la^i=(\la^i_1,\la^i_1-2,\la^i_1-4,\la^i_1-6)$ with $\la^i_1\geq 7$. If $h(\la)=4$ then $n\equiv 0\pmod{4}$ and $\la=(n/4+3,n/4+1,n/4-1,n/4-3)$, which is one of the excluded cases. So we can assume that there exists $0\leq j<i$ with $\la^j=(\la^j_1,\la^j_1-2,\la^j_1-4,\la^j_1-6,1)$. If $\la^i_1\geq 9$ we can take $\pi=(\la^i_1-2,\la^i_1-4,\la^i_1-6,\la^i_1-7,1)=(\la^j_1-(i-j)-2,\la^j_2-(i-j)-2,\la^j_3-(i-j)-2,\la^j_4-(i-j)-1,\la^j_5)$. If $7\leq\la^i\leq 8$ then $\la^i=(7,5,3,1)$ or $\la^i=(8,6,4,2)$ and then $n\leq 28$, so that $\la=(9,7,5,3,1)$ and $k=11$ (since $h(\la)\geq 5$, $n-k=|\la^i|-a_i$ and $k\leq n/2$). In this case notice first that $D^{(5,4,3,2,1)}$ is a composition factor of $D^{(9,7,5,3,1)}\da_{\SSS_{15}}$. Since $D^{(5,4,3,2,1)}\cong S^{(5,4,3,2,1)}$, using branching rule in characteristic 0 and decomposition matrices, we can see that $D^{(9,7,5,3,1)}\da_{\SSS_{14}}$ has more than 3 composition factors.

{\sf Case 6.4.3.} $a_i=7$. Then we can take $\mu=(\la^i_1-2,\la^i_2-2,\la^i_3-2,\la^i_4-1)$ and $\nu=(\la^i_1-3,\la^i_2-2,\la^i_3-1,\la^i_4-1)$. 

{\sf Case 6.4.3.1.} $\la^i_1\geq\la^i_2+4$. We can take $\pi=(\la^i_1-4,\la^i_2-2,\la^i_3-1,\la^i_4)$.

{\sf Case 6.4.3.2.} $\la^i_2\geq\la^i_3+4$. We can take $\pi=(\la^i_1-3,\la^i_2-3,\la^i_3-1,\la^i_4)$.

{\sf Case 6.4.3.3.} $\la^i_3\geq \la^i_4+4$. We can take $\pi=(\la^i_1-3,\la^i_2-2,\la^i_3-2,\la^i_4)$.

{\sf Case 6.4.3.4.} $\la^i_1=\la^i_2+2$, $\la^i_2=\la^i_3+2$ and $\la^i_3=\la^i_4+2$. Then $\la^i=(\la^i_1,\la^i_1-2,\la^i_1-4,\la^i_1-6)$ with $\la^i_1\geq 7$. Again we can assume that there exists $0\leq j<i$ with $\la^j=(\la^j_1,\la^j_1-2,\la^j_1-4,\la^j_1-6,1)$. If $\la^i_1\geq 9$ we can take $\pi=(\la^i_1-3,\la^i_1-4,\la^i_1-6,\la^i_1-7,1)=(\la^j_1-(i-j)-3,\la^j_2-(i-j)-2,\la^j_3-(i-j)-2,\la^j_4-(i-j)-1,\la^j_5)$. If $7\leq\la^i\leq 8$ then $\la^i=(7,5,3,1)$ or $\la^i=(8,6,4,2)$ and $n\leq 26$, so that $\la=(9,7,5,3,1)$ and $k=12$ (since $h(\la)\geq 5$, $n-k=|\la^i|-a_i$ and $k\leq n/2$). In this case notice first that $D^{(5,4,3,2,1)}$ is a composition factor of $D^{(9,7,5,3,1)}\da_{\SSS_{15}}$. Since $D^{(5,4,3,2,1)}\cong S^{(5,4,3,2,1)}$, using branching rule in characteristic 0 and \cite[Tables]{JamesBook}, we can see that $D^{(9,7,5,3,1)}\da_{\SSS_{13}}$ has more than 3 composition factors.

{\sf Case 6.4.4.} $a_i=8$ and $\la^i_4\geq 2$. Then we can take $\mu=(\la^i_1-2,\la^i_2-2,\la^i_3-2,\la^i_4-2)$ and $\nu=(\la^i_1-3,\la^i_2-2,\la^i_3-2,\la^i_4-1)$. 

{\sf Case 6.4.4.1.} $\la^i_1\geq\la^i_2+4$. We can take $\pi=(\la^i_1-4,\la^i_2-2,\la^i_3-1,\la^i_4-1)$.

{\sf Case 6.4.4.2.} $\la^i_2\geq\la^i_3+4$. We can take $\pi=(\la^i_1-3,\la^i_2-3,\la^i_3-1,\la^i_4-1)$.

{\sf Case 6.4.4.3.} $\la^i_3\geq \la^i_4+4$. We can take $\pi=(\la^i_1-3,\la^i_2-3,\la^i_3-2,\la^i_4)$.

{\sf Case 6.4.4.4.} $\la^i_1=\la^i_2+2$, $\la^i_2=\la^i_3+2$ and $\la^i_3=\la^i_4+2$. Then $\la^i=(\la^i_1,\la^i_1-2,\la^i_1-4,\la^i_1-6)$ with $\la^i_1\geq 8$. Again we can assume that $0\leq j<i$ with $\la^j=(\la^j_1,\la^j_1-2,\la^j_1-4,\la^j_1-6,1)$. If $\la^i_1\geq 9$ we can take $\pi=(\la^i_1-2,\la^i_1-4,\la^i_1-6,\la^i_1-7,1)=(\la^j_1-(i-j)-2,\la^j_2-(i-j)-2,\la^j_3-(i-j)-2,\la^j_4-(i-j)-1,\la^j_5)$. If $\la^i=8$ then $\la^i=(8,6,4,2)$ and $n\leq 24$ and it can be checked that no $\la$ exists (since $h(\la)\geq 5$).

{\sf Case 6.4.5.} $a_i=8$ and $\la^i_4=1$. 

{\sf Case 6.4.5.1.} $\la^i_1\geq\la^i_2+4$. Then we can take $\mu=(\la^i_1-3,\la^i_2-2,\la^i_3-2,\la^i_4-1)$, $\nu=(\la^i_1-4,\la^i_2-2,\la^i_3-1,\la^i_4-1)$ and $\pi=(\la^i_1-5,\la^i_2-2,\la^i_3-1,\la^i_4)$.

{\sf Case 6.4.5.2.} $\la^i_2\geq\la^i_3+4$. Then we can take $\mu=(\la^i_1-3,\la^i_2-2,\la^i_3-2,\la^i_4-1)$, $\nu=(\la^i_1-3,\la^i_2-3,\la^i_3-1,\la^i_4-1)$ and $\pi=(\la^i_1-4,\la^i_2-3,\la^i_3-1,\la^i_4)$.

{\sf Case 6.4.5.3.} $\la^i_1=\la^i_2+2$ and $\la^i_2=\la^i_3+2$. Then $\la^i=(\la^i_1,\la^i_1-2,\la^i_2-4,1)$ with $\la^i_1\geq 7$ odd. If $i\geq 1$ then $\la^{i-1}=(\la^i_1+1,\la^i_1-1,\la^i_1-3,2)$ and $a_{i-1}=a_i+4=12$. Then $3\la^i_1-1=|\la^{i-i}|\geq 2a_{i-1}=24$ and so $\la^i_1\geq 9$. So we can take $\mu=(\la^i_1-3,\la^i_1-4,\la^i_1-6)$, $\nu=(\la^i_1-3,\la^i_1-5,\la^i_1-6,1)$ and $\pi=(\la^i_1-4,\la^i_1-5,\la^i_1-6,2)=(\la^{i-1}_1-5,\la^{i-1}_2-4,\la^{i-1}_3-3,\la^{i-1}_4)$. If $i=0$ then $\la=\la^i$ and so $n\equiv 4\pmod{6}$, $\la=((n-1)/3+2,(n-1)/3,(n-1)/3-2,1)$ and $k=a_i=8$. If $n=16$ then $\la=(7,5,3,1)$, else $n\geq 22$, so that we are in one of the excluded cases. If $\la=(7,5,3,1)$ notice first that $D^{(4,3,2,1)}$ is a composition factor of $D^{(7,5,3,1)}\da_{\SSS_{10}}$. Since $(4,3,2,1)$ is a 2-core, we have that $D^{(4,3,2,1)}\cong S^{(4,3,2,1)}$. From \cite[Tables]{JamesBook} and branching rule in characteristic 0, it then follows that $D^{(7,5,3,1)}\da_{\SSS_8}$ has more than 3 non-isomorphic composition factors.
\end{proof}

\begin{Lemma}\label{L221118}
Let $p=2\mid n$, and $\la\in\Parinv_2(n)$. Then  $D^\la\da_{\SSS_{n/2}}$ has at least 3 non-isomorphic composition factors unless $\la$ has one 
%of the exceptional partitions from Lemma \ref{l1} or $\la$ is of one 
of the following forms:
\begin{enumerate}
\item[{\rm (i)}] $\be_n$ with $n\equiv 0\pmod{4}$,

\item[{\rm (ii)}] $(\be_{n-1},1)$, 

\item[{\rm (iii)}] $n\geq 24$ with $n\equiv 0\pmod{8}$ and $\la=(n/4+3,n/4+1,n/4-1,n/4-3)$,

\item[{\rm (iv)}] $n\geq 10$ with $n\equiv 2\pmod{4}$ and $\la=((n+6)/4,(n+2)/4,(n-2)/4,(n-6)/4)$,

\item[{\rm (v)}] $n\geq 24$ with $n\equiv 0\pmod{4}$ and $\la=(n/4+2,n/4+1,n/4-1,n/4-2)$,

\item[{\rm (vi)}] $n\geq 14$ with $n\equiv 2\pmod{4}$ and $\la=((n+10)/4,(n+6)/4,(n-6)/4,(n-10)/4)$,
\end{enumerate}
\end{Lemma}

\begin{proof}
If $n\leq 10$ then $\la\in\Parinv_2(n)$ implies that $\la=\be_n$, $(\be_{n-1},1)$ or $(4,3,2,1)$, in particular, $\la$ is of the exceptional forms (i), (ii) or (iv), and so we may assume that $n\geq 12$. 

If $\la$ is the double of $(11,1)$, $(9,5)$ or $(11,7)$ then $\la$ is of the exceptional forms (ii) or (iv). So we do not need to consider them. Let $E_1$ be the set of the doubles of the remaining exceptional partitions appearing in Lemma \ref{l1}. Moreover, let 
\begin{align*}
E_2:=\{&(7,5,4,3,2,1), (7,6,5,3,1), (8,7,5,3,2,1), (8,7,5,4,3,1), (8,7,5,4,3,2,1),\\ 
&(8,7,6,5,3,1), (8,7,6,5,3,2,1), (8,7,6,5,4,3,1), (8,7,6,5,4,3,2,1)\}.
\end{align*}
(there is an overlap between $E_1$ and $E_2$). Finally let 
\begin{align*}
E_3:=\{&(7,5,3,1), (7,5,3,2,1), (7,6,2,1), (8,7,5,3,1), (9,7,3,2,1)\}.
\end{align*}
By Lemma~\ref{Lemma39}, $D^{(h(\la),h(\la)-1,\ldots,1)}$ is a composition factor of $D^\la\da_{\SSS_{h(\la)(h(\la)+1)/2}}$. Since $D^{(h(\la),h(\la)-1,\ldots,1)}\cong S^{(h(\la),h(\la)-1,\ldots,1)}$ is an irreducible Specht module, using branching rule for Specht modules and known decomposition matrices, the lemma can be checked for all $\la\in E_1\cup E_2\cup E_3$.

Recalling the partition $\la^\JS$ from \S\ref{SSSCF}, we can now assume that $n\geq 12$ and that we are not in one of the exceptional cases of Lemma \ref{l1}. Then by Lemma \ref{l1}, $m:=|\la^\JS|\geq n/2+5$ and $\la^\JS_{2j-1}-\la^\JS_{2j}\leq 2$ for all  $j\geq 1$. By Lemma \ref{LJS} we have that $D^{\la^\JS}$ is a composition factor of $D^\la\da_{\SSS_m}$. Moreover, by Lemma \ref{l25}, $D^{\la^\JS}\da_{\SSS_{n/2}}$ has at least three non-isomorphic composition factors, unless $\la^\JS$ in one of the exceptional cases listed in Lemma \ref{l25}. Since $\la^\JS_{2j-1}-\la^\JS_{2j}\leq 2$ for all  $j\geq 1$, we are left only with the following cases:
\begin{enumerate}
\item[(a)]
$m$ is even and $\la^\JS=\be_m$,

\item[(b)]
$m$ is odd and $\la^\JS=(\be_{m-1},1)$,

\item[(c)]
$\la^\JS=(6,4,2)$,

\item[(d)]
$m\geq 16$ with $m\equiv 0\pmod{4}$, $\la^\JS=(\be_{m-2},2)$ and $m=n/2+5$,

\item[(e)]
$\la^\JS=(7,5,3,1)$ and $n=22$,

\item[(f)]
$m\geq 20$ with $m\equiv 0\pmod{4}$ and $\la^\JS=(m/4+3,m/4+1,m/4-1,m/4-3)$.
\end{enumerate}
%From Lemma \ref{LJS} we have that $0\leq\la_j-\la^\JS_j\leq j-1$ and $\la_{h(\la^\JS)+1}\leq h(\la^\JS)$. Further if $\mu_{h(\la^\JS)}\geq 3$ then $\la_{h(\la^\JS)+1}\leq 1$ and if $\la_{h(\la^\JS)+1}=1$ then $\la_1=\la^\JS_1$ is even. If $\la^\JS_{h(\la^\JS)}=2$ and $\la^\JS_{h(\la^\JS)-1}\geq 6$ then $\la_{h(\la^\JS)}\leq 3$. If $\la^\JS_{h(\la^\JS)}=1$ and $\la^\JS_{h(\la^\JS)-1}\geq 5$ then $\la_{h(\la^\JS)}\leq 2$.
So, using Lemma \ref{LJS} and since $n$ is even, in each of the corresponding above cases the following holds:
\begin{enumerate}
\item[(a)]
$\la$ is $\be_n$ or $(\be_{n-1},1)$.

\item[(b)]
$\la$ is $(\be_{n-1},1)$, $(\be_{n-2},2)$, $(\be_{n-3},3)$, $(\be_{n-3},2,1)$, $(\be_{n-4},3,1)$, $(\be_{n-5},3,2)$ or $(\be_{n-6},3,2,1)$.

\item[(c)]
$\la$ is $(6,4,2)$, $(6,4,3,1)$, $(6,4,3,2,1)$, $(6,5,2,1)$, $(6,5,3)$, $(6,5,3,2)$, $(6,5,4,1)$, $(6,5,4,2,1)$, $(6,5,4,3)$ or $(6,5,4,3,2)$.

\item[(d)]
In this case $\la$ is one of the partitions in (b).

\item[(e)]
$\la$ is $(7,5,4,3,2,1)$, $(7,6,4,3,2)$, $(7,6,5,3,1)$ or $(7,6,5,4)$.

\item[(f)]
\begin{enumerate}
\item[(f1)]$\la_1\geq 8$, $\la_1-\la_2\leq 2$, $\la_1-\la_3\leq 4$, $\la_1-\la_4\leq 6$, $\la_5=0$ if $\la_1$ is odd or $\la_5\leq 1$ if $\la_1$ is even;
\item[(f2)] $\la$ is one of the following: $(8,6,4,3,2,1)$, $(8,6,5,3,2)$, $(8,6,5,4,2,1)$, $(8,6,5,4,3)$, $(8,6,5,4,3,2)$, $(8,7,4,3,2)$, $(8,7,5,3,2,1)$, $(8,7,5,4,2)$,\\ $(8,7,5,4,3,1)$, $(8,7,5,4,3,2,1)$, $(8,7,6,3,2)$, $(8,7,6,4,2,1)$, $(8,7,6,4,3)$, $(8,7,6,5,2)$, $(8,7,6,5,3,1)$, $(8,7,6,5,3,2,1)$, $(8,7,6,5,4)$, $(8,7,6,5,4,2)$, $(8,7,6,5,4,3,1)$ or $(8,7,6,5,4,3,2,1)$.
\end{enumerate}
\end{enumerate}
Since $\la\in\Parinv_2(n)$ we then have that one of the following holds: % (some of the exceptional partitions belong to some of the other families of partitions):
\begin{itemize}
\item[{\rm (1)}] $n\equiv 0\pmod{4}$ and $\la=\be_n$,

\item[{\rm (2)}] $\la=(\be_{n-1},1)$,

\item[{\rm (3)}] $\la=(\be_{n-3},2,1)$,

\item[{\rm (4)}] $n\equiv 0\pmod{4}$ and $\la=(\be_{n-4},3,1)$,

\item[{\rm (5)}] $\la=(\be_{n-5},3,2)$,

\item[{\rm (6)}] $n\equiv 2\pmod{4}$ and $\la=(\be_{n-6},3,2,1)$,

\item[{\rm (7)}] $n\geq 24$ with $n\equiv 0\pmod{8}$ and $\la=(n/4+3,n/4+1,n/4-1,n/4-3)$,

\item[{\rm (8)}] $n\geq 18$ with $n\equiv 2\pmod{4}$ and $\la=((n+6)/4,(n+2)/4,(n-2)/4,(n-6)/4)$,

\item[{\rm (9)}] $n\geq 24$ with $n\equiv 0\pmod{4}$ and $\la=((n+8)/4,(n+4)/4,(n-4)/4,(n-8)/4)$,

\item[{\rm (10)}] $n\geq 14$ with $n\equiv 2\pmod{4}$ and $\la=((n+10)/4,(n+6)/4,(n-6)/4,(n-10)/4)$,

\item[{\rm (11)}] $n\geq 24$ with $n\equiv 0\pmod{8}$ and $\la=(n/4+2,n/4+1,n/4-1,n/4-3,1)$,

\item[{\rm (12)}] $\la\in E_2$.
\end{itemize}
Taking into account that we have already dealt with $\la\in E_2$ and that in the statement of the lemma we have excluded certain classes of partitions, it remains to deal with the cases (3)--(6), (11). We will repeatedly use Lemma~\ref{Lemma39}. 

To deal with (3)--(6), first note that if $\nu=(\be_l,\overline{\nu})$ with $\overline{\nu}\not=0$ and $\overline{\nu}_1\leq(\be_l)_2-2$, then $D^{(\be_{l-1},\overline{\nu})}$ and $D^{(\be_{l-2},\overline{\nu})}$ are composition factors of $D^\nu\da_{\SSS_{|\nu|-1}}$ and $D^\nu\da_{\SSS_{|\nu|-2}}$ respectively and at least one of $\nu$, $(\be_{l-1},\overline{\nu})$ or $(\be_{l-2},\overline{\nu})$ has a good node below the second row.

(3) If $n\geq 18$ it then follows  that $D^{(\be_{n/2-2},2)}$, $D^{(\be_{n/2-1},1)}$ and $D^{\be_{n/2}}$ are composition factors of $D^{(\be_{n-3},2,1)}\da_{\SSS_{n/2}}$. If $n=10$ then $\la$ is in the exceptional family (iv), if $n=12$ then $\la\in E_1$, if $n=14$ then $\la$ is in the exceptional family (vi), while if $n=16$ then $\la\in E_3$.

(4) If $n\geq 20$ that $D^{(\be_{n/2-3},2,1)}$, $D^{(\be_{n/2-2},2)}$ and $D^{(\be_{n/2-1},1)}$ are composition factors of $D^{(\be_{n-4},3,1)}\da_{\SSS_{n/2}}$. If $n=16$ then $\la\in E_3$.

(5) If $n\geq 24$ that $D^{(\be_{n/2-3},2,1)}$, $D^{(\be_{n/2-2},2)}$ and $D^{(\be_{n/2-1},1)}$ are composition factors of $D^{(\be_{n-5},3,2)}\da_{\SSS_{n/2}}$ (this can be checked also for $n=20$ and $n=22$). If $n=14$ then $\la$ is in the exceptional family (iv), if $n=16$ then $\la\in E_1$, if $n=18$ then $\la$ is in the exceptional family (vi).

(6) If $n\geq 26$ that $D^{(\be_{n/2-4},3,1)}$, $D^{(\be_{n/2-3},2,1)}$ and $D^{(\be_{n/2-2},2)}$ are composition factors of $D^{(\be_{n-6},3,2,1)}\da_{\SSS_{n/2}}$. If $n=18$ or $n=22$ then $\la\in E_3$.

(11) If $n\geq 24$ with $n\equiv 0\pmod{8}$ notice that $D^{(n/4+2,n/4,n/4-2,n/4-4,1)}$ is a composition factor of $D^{(n/4+2,n/4+1,n/4-1,n/4-3,1)}\da_{\SSS_{n-3}}$ and $D^{(n/4+2,n/4,n/4-2,n/4-4)}$ is a composition factor of $D^{(n/4+2,n/4+1,n/4-1,n/4-3,1)}\da_{\SSS_{n-4}}$. Using Lemma \ref{L110518} it follows that $D^{(n/8+3,n/8+1,n/8-1,n/8-3)}$, $D^{(n/8+2,n/8+1,n/8-1,n/8-2)}$ and $D^{(n/8+3,n/8+1,n/8-1,n/8-2,1)}$ are composition factor of $D^{(n/4+2,n/4+1,n/4-1,n/4-3,1)}\da_{\SSS_{n/2}}$ if $n\geq 32$. If $n=24$ then $\la\in E_3$.

\end{proof}

\begin{Lemma}\label{L200818_2}
Let $p=2$, $n\equiv 0\pmod{4}$ and $\la=(n/4+2,n/4+1,n/4-1,n/4-2)$. Assume that $2\leq k\leq n-9$ and let $\mu$ and $\nu$ be given by
\begin{itemize}
\item[{\rm (i)}]
$\mu=(n/4-m+2,n/4-m+1,n/4-m-1,n/4-m-2)$, $\nu=(n/4-m+3,n/4-m+1,n/4-m-1,n/4-m-3)$ if $k=4m$,

\item[{\rm (ii)}]
$\mu=(n/4-m+2,n/4-m,n/4-m-1,n/4-m-2)$, $\nu=(n/4-m+2,n/4-m+1,n/4-m-1,n/4-m-3)$ if $k=4m+1$,

\item[{\rm (iii)}]
$\mu=(n/4-m+2,n/4-m,n/4-m-1,n/4-m-3)$, $\nu=(n/4-m+1,n/4-m,n/4-m-1,n/4-m-2)$ if $k=4m+2$,

\item[{\rm (iv)}]
$\mu=(n/4-m+1,n/4-m,n/4-m-1,n/4-m-3)$, $\nu=(n/4-m+2,n/4-m,n/4-m-2,n/4-m-3)$ if $k=4m+3$.
\end{itemize}
Then $D^\mu$ and $D^\nu$ are composition factor of $D^\la\da_{\SSS_{n-k}}$.
\end{Lemma}

\begin{proof}
Notice that repeatedly applying Lemma \ref{Lemma39} we have that
\begin{itemize}
\item
$D^{(n/4+2,n/4,n/4-1,n/4-2)}$ is a composition factor of $D^\la\da_{\SSS_{n-1}}$,

\item
$D^{(n/4+2,n/4,n/4-1,n/4-3)}$ and $D^{(n/4+1,n/4,n/4-1,n/4-2)}$ are composition factors of $D^\la\da_{\SSS_{n-2}}$,

\item
$D^{(n/4+1,n/4,n/4-1,n/4-3)}$ and $D^{(n/4+2,n/4,n/4-2,n/4-3)}$ are composition factors of $D^\la\da_{\SSS_{n-3}}$,

\item
$D^{(n/4+1,n/4,n/4-2,n/4-3)}$ and $D^{(n/4+2,n/4,n/4-2,n/4-4)}$ are composition factors of $D^\la\da_{\SSS_{n-4}}$,

\item
$D^{(n/4+1,n/4-1,n/4-2,n/4-3)}$ and $D^{(n/4+1,n/4,n/4-2,n/4-4)}$ are composition factors of $D^\la\da_{\SSS_{n-5}}$.
\end{itemize}

Since $\mu=((n-4)/4+2,(n-4)/4+1,(n-4)/4-1,(n-4)/4-2)$ for $k=4$, the lemma then follows by induction.
\end{proof}

\section{Reduction Theorems}

\subsection{Criterion for reducibility of restrictions}
For $\la\in\Par_p(n)$ we denote
$$
\EE(\la):=\End_\F(D^\la).
$$
Note that $\EE(\la)$ is naturally an $\F\SSS_n$-module. 

If $\la\in\Parinv_p(n)$, then upon restriction to $\AAA_n$, we have the $\F\AAA_n$-module decomposition
$$
\EE(\la)=\EE_{+,+}(\la)\oplus\EE_{-,-}(\la)\oplus \EE_{+,-}(\la)\oplus \EE_{-,+}(\la),
$$
where
$$\EE_{\de,\eps}(\la):=\Hom_F(E^\la_\de,E^\la_\eps)\qquad(\de,\eps\in\{+,-\}).$$

\begin{Lemma} \label{Lemma4.13} %{\rm \cite{}}%{\bf ()}
Let $\la\in\Parinv_p(n)$ and $V$ be an $\F\SSS_n$-module. Then
\[\Hom_{\SSS_n}(V,\EE(\la))\cong \Hom_{\AAA_n}(V\da_{\AAA_n},\EE_{\pm,\pm}(\la)\oplus\EE_{\pm,\mp}(\la)).\]
\end{Lemma}
\begin{proof}
Using $D^\la\cong E^\la_\pm\uparrow^{\SSS_n}$, it is easy to check that $\EE(\la)\cong\Hom_F(E^\la_\pm,E^\la_+\oplus E^\la_-)\ua^{\SSS_n}$,
%\begin{align*}
%\Hom_{\SSS_n}(V,\EE(\la))\cong\Hom_{\SSS_n}(V,\Hom_F(E^\la_\pm,E^\la_+\oplus E^\la_-)\ua^{\SSS_n}),
%\cong\Hom_{\AAA_n}(V\da_{\AAA_n},\EE_{\pm,\pm}(\la)\oplus\EE_{\pm,\mp}(\la)).
%\end{align*}
which implies the lemma using Frobenious reciprocity.
\end{proof}

Let 
\begin{equation}\label{EProjA}
\pi_{\de,\eps}:\EE(\la)\to \EE_{\de,\eps}(\la)\qquad(\de,\eps\in\{+,-\})
\end{equation}
be the corresponding projections. 
 Note that $E^\la_\pm$ with $(E^\la_\pm)^\si\cong E^\la_{\mp}$ implies 
\begin{equation}\label{EEESi}
\EE_{+,\pm}(\la)^\si\cong \EE_{-,\mp}(\la). 
\end{equation}
Note also that 
$
\EE_{+,+}(\la)\oplus\EE_{-,-}(\la)
$
is an $\F \SSS_n$-submodule of $\EE(\la)$, as is $\EE_{+,-}(\la)\oplus \EE_{-,+}(\la)$. We then have the corresponding projections
\begin{equation}\label{EProj}
\pi_0:\EE(\la)\to \EE_{+,+}(\la)\oplus\EE_{-,-}(\la)
\quad\text{and}\quad
\pi_1:\EE(\la)\to \EE_{+,-}(\la)\oplus \EE_{-,+}(\la).
\end{equation}

Recall the notation $\J(G)$ from \S\ref{SSGM}. 
The following is an analogue of \cite[Lemma~2.17]{KMTOne}:

\begin{Lemma} \label{LBasicAlt} %{\rm \cite{}}
Let $\la\in\Parinv_p(n)$, $\de\in\{+,-\}$ and $G\leq \AAA_n$ be a subgroup such that $E^\la_\de{\da}_G$ is irreducible. Then $
\dim\Hom_{\AAA_n}(\J(G),\EE_{\de,\eps}(\la))\leq 1
$
for all $\eps\in\{+,-\}$.
\end{Lemma}
\begin{proof}
We have 
\begin{align*}
\Hom_{\AAA_n}(\J(G),\EE_{\de,\eps}(\la))&=\Hom_{\AAA_n}(\bone_G{\ua}^{\AAA_n},\EE_{\de,\eps}(\la))
\\
&\cong \Hom_{G}(\bone_G,\EE_{\de,\eps}(\la)\da_G)
\\
&\cong\Hom_G(E^\la_\de{\da}_G,E^\la_\eps{\da}_G).
\end{align*}
Since $E^\la_\de{\da}_G$ is irreducible and $\dim E^\la_\de=\dim E^\la_\eps$ the lemma follows. 
%If $\eps=\de$, we are done by Schur's lemma. If $\eps\neq\de$, 
\end{proof}

\begin{Lemma} \label{L090518} %{\rm \cite{}}
Let $k\in\Z_{\geq 0}$ and $n\geq \max(5,2k)$, and exclude the cases $(p,n-2k)=(2,\leq 2),(3,\leq 1)$. Suppose that $\la\in\Parinv_p(n)$, $G\leq \AAA_n$ and there is $\psi\in\Hom_{\SSS_n}(\I(G), M_k)$ such that $[\im\psi:D_k]\neq 0$. 
\begin{enumerate}
\item[{\rm (i)}] If 
there is $\zeta\in\Hom_{\SSS_n}(M_k,\EE_{+,+}(\la)\oplus \EE_{-,-}(\la))$ such that $[\im\zeta:D_k]\neq 0$ then there exist 
$$
\xi'\in\Hom_{\AAA_n}(\J(G),\EE_{+,+}(\la))\quad\text{and}\quad 
 \xi''\in\Hom_{\AAA_n}(\J(G),\EE_{-,-}(\la))
$$ 
such that $[\im\xi':E_k]\neq 0$ and $[\im\xi'':E_k]\neq 0$. 
\item[{\rm (ii)}] If 
there is $\zeta\in\Hom_{\SSS_n}(M_k,\EE_{+,-}(\la)\oplus \EE_{-,+}(\la))$ such that $[\im\zeta:D_k]\neq 0$ then there exist 
$$
\xi'\in\Hom_{\AAA_n}(\J(G),\EE_{+,-}(\la))\quad\text{and}\quad 
 \xi''\in\Hom_{\AAA_n}(\J(G),\EE_{-,+}(\la))
$$ 
such that $[\im\xi':E_k]\neq 0$ and $[\im\xi'':E_k]\neq 0$. 
\end{enumerate}
\end{Lemma}
\begin{proof}
We prove (i), the proof of (ii) being similar. By Lemma~\ref{L080518}, the assumptions $n\geq \max(5,2k)$ and $(p,n-2k)\neq(2,\leq 2),(3,\leq 1)$ guarantee that $E_k$ is irreducible and appears with multiplicity $1$ in $M_k\da_{\AAA_n}$. 
Note that $\I(G)\cong \J(G)\ua^{\SSS_n}$. By Lemma~\ref{L080518_4}(i), there is $\psi'\in\Hom_{\AAA_n}(\J(G),M_k\da_{\AAA_n})$ with $[\im\psi':E_k]\neq 0$. 
Furthermore,  
$$\EE_{+,+}(\la)\ua^{\SSS_n}\cong \EE_{+,+}(\la)\oplus \EE_{-,-}(\la)\cong \EE_{-,-}(\la)\ua^{\SSS_n}. 
$$
By Lemma~\ref{L080518_4}(ii), there are 
$$\zeta'\in\Hom_{\AAA_n}(M_k\da_{\AAA_n}, \EE_{+,+}(\la))\quad \text{and} \quad
\zeta''\in\Hom_{\AAA_n}(M_k\da_{\AAA_n}, \EE_{-,-}(\la))$$
 with $[\im\zeta':E_k]\neq 0$ and $[\im\zeta'':E_k]\neq 0$. Since $E_k$ appears in $M_k$ with multiplicity $1$, we can now take $\xi':=\zeta'\circ \psi'$ and $\xi'':=\zeta''\circ \psi'$.
\end{proof}

\begin{Theorem} \label{CFirst} %{\rm \cite{}}
Let $k\geq 1$, $n\geq \max(5,2k)$, and exclude the cases $(p,n-2k)=(2,\leq 2),(3,\leq 1)$. Suppose that $\la\in\Parinv_p(n)$, $G\leq \AAA_n$ and there are 
$$\psi\in\Hom_{\SSS_n}(\I(G), M_k)
\quad\text{and}\quad
\zeta\in\Hom_{\SSS_n}(M_k,\EE_{+,+}(\la)\oplus \EE_{-,-}(\la))
$$
 such that $[\im\psi:D_k]\neq 0$ and $[\im\zeta:D_k]\neq 0$. Then $E^\la_{\pm}\da_G$ is reducible. 
\end{Theorem}
\begin{proof}
We prove that $E^\la_{+}\da_G$ is reducible, the argument for $E^\la_{-}\da_G$ being similar. Note that there always exists $\xi_0\in\Hom_{\AAA_n}(\J(G),\EE_{+,+}(\la))$ whose image is the trivial module $\bone_{\AAA_n}$. On the other hand, by Lemma~\ref{L090518}(i), there exists 
$
\xi_k\in\Hom_{\AAA_n}(\J(G),\EE_{+,+}(\la))
$
whose image contains $E_k$ as a composition factor. Since $\xi_0$ and $\xi_k$ are linearly independent, the theorem follows from Lemma~\ref{LBasicAlt}. 
\end{proof}

\begin{Theorem} \label{CSecond} %{\rm \cite{}}
Let $1\leq k<l$, $n\geq \max(5,2l)$, and exclude the cases $(p,n-2l)=(2,\leq 2),(3,\leq 1)$. Suppose that $\la\in\Parinv_p(n)$, $G\leq \AAA_n$ and for $j=k,l$ there are 
$$\psi_j\in\Hom_{\SSS_n}(\I(G), M_j) 
\quad\text{and}\quad
\zeta_j\in\Hom_{\SSS_n}(M_j,\EE_{+,-}(\la)\oplus \EE_{-,+}(\la))
$$
 such that $[\im\psi_j:D_j]\neq 0$ and $[\im\zeta_j:D_j]\neq 0$. Then $E^\la_{\pm}\da_G$ is reducible. 
\end{Theorem}
\begin{proof}
We prove that $E^\la_{+}\da_G$ is reducible, the argument for $E^\la_{-}\da_G$ being similar. By Lemma~\ref{L090518}(ii), for $j=k,l$, there exists 
$
\xi_j\in\Hom_{\AAA_n}(\J(G),\EE_{+,-}(\la))
$
whose image contains $E_j$ as a composition factor. Note that $\xi_k$ and $\xi_l$ are linearly independent since $M_k\da_{\AAA_n}$ does not have $E_l$ as a composition factor and therefore $\im \xi_k$ does not have $E_l$ as a composition factor. The theorem now follows from Lemma~\ref{LBasicAlt}. 
\end{proof}

\subsection{First reduction theorems for alternating groups}

Recall the projections $\pi_0,\pi_1%,\pi_{\de,\eps}
$ defined in %(\ref{EProjA}) and 
(\ref{EProj}) and the integers $i_k(G)$ from \S\ref{SSPerm}

\begin{Proposition} \label{PIntrans} %{\rm \cite{}}
Let $n\geq 8$ and exclude the case $(n,p)=(8,2)$. If $\la\in\Parinv_p(n)$ and $G\leq \AAA_n$ is a subgroup such that $1<i_1(G)<i_2(G)$ and $E^\la_+\da_G$ or $E^\la_-\da_G$ is irreducible then one of the following holds:
\begin{enumerate}
\item[{\rm (i)}] $i_2(G)=i_3(G)$ and $\la$ is JS.
%$G\leq \AAA_{n-1}$.
\item[{\rm (ii)}] $p=2$ and $\la=\be_n$.
\item[{\rm (iii)}] $p=2\mid n$ and $\dim(S_1^*)^G\geq i_2(G)-1$.
\item[{\rm (iv)}] $\la$ is JS, $p=2\mid n$ and $\dim(S_2^*)^G\geq i_3(G)-i_1(G)$.
\item[{\rm (v)}] $\la$ is JS, $p=3$, $n\equiv 0\pmod{3}$ and $\dim(S_1^*)^G\geq i_3(G)-i_2(G)+i_1(G)-1$.
\item[{\rm (vi)}] $\la$ is JS, $p=3$, $n\equiv 1\pmod{3}$ and $\dim(S_2^*)^G\geq i_3(G)-i_1(G)$.
\end{enumerate} 
\end{Proposition}
\begin{proof}
The result for $p>3$ follows from \cite[Main Theorem]{KSAlt}. So we may now assume that $p=2$ or $p=3$. Also, without loss of generality we assume that $E^\la_+\da_G$ is irreducible. Moreover, in view of (ii) if $p=2$ we further assume that $\la\not=\be_n$. From Lemma \ref{L080518} it then follows that $h(\la)\geq 3$.

In view of (\ref{EItoM}), the assumption $i_1(G)>1$ implies the existence of 
$$\psi_1\in\Hom_{\SSS_n}(\I(G),M_1)$$ with $[\im \psi_1:D_1]\neq 0$. There also exists 
$$\psi_2\in\Hom_{\SSS_n}(\I(G), M_2)$$ with $[\im \psi_2:D_2]\neq 0$. Indeed, if $p=2\mid n$, then 
in view of (iii), we may assume that $\dim (S_1^*)^G<i_2(G)-1$, whence $\psi_2$ exists by Theorem~\ref{Tp=2}. On the other hand if $p=3$ or $p=2\nmid n$, by Theorems~\ref{Tp=3} and \ref{Tp=2}, we have that $M_2\sim M_1|S_2^*$, so the assumption $i_2(G)>i_1(G)$ and (\ref{EItoM}) imply the existence of $\psi_2$ with the required properties. 

Furthermore, by Lemma~\ref{C180418}, there exists 
$$\zeta_2\in\Hom_{\SSS_n}(M_2, \EE(\la))$$ with $[\im \zeta_2:D_2]\neq 0$. By Lemma~\ref{LM_3E}, there exists 
$$\zeta_3\in\Hom_{\SSS_n}(M_3, \EE(\la))$$ with $[\im \zeta_3:D_3]\neq 0$.

{\sf Case 1.} $\la$ is not JS. Then 
$$
1<\dim\End_{\SSS_{n-1}}(D^\la\da_{\SSS_{n-1}})=\dim \Hom_{\SSS_n}(M_1,\EE(\la))
$$ 
and the well-known structure of $M_1$ (see for example Theorems~\ref{Tp=3} and \ref{Tp=2}) 
imply that  there is $\zeta_1\in\Hom_{\SSS_n}(M_1,\EE(\la))$ with $[\im \zeta_1:D_1]\neq 0$. 

{\sf Case 1.1.} $\zeta_1':=\pi_0\circ \zeta_1$ satisfies $[\im\zeta_1':D_1]\neq 0$. Then the proposition follows by Theorem~\ref{CFirst} (with $k=1$). 

{\sf Case 1.2.} $\zeta_1'$ satisfies $[\im\zeta_1':D_1]=0$. Then $\zeta_1'':=\pi_1\circ \zeta_1$ satisfies $[\im\zeta_1'':D_1]\neq 0$.

{\sf Case 1.2.1.} $\zeta_2':=\pi_0\circ \zeta_2$ satisfies $[\im\zeta_2':D_2]\neq 0$. Then the proposition follows by Theorem~\ref{CFirst} (with $k=2$).

{\sf Case 1.2.2.} $\zeta_2'$ satisfies $[\im\zeta_2':D_2]=0$. Then $\zeta_2'':=\pi_1\circ \zeta_2$ satisfies $[\im\zeta_2'':D_2]\neq 0$ and we can conclude by Theorem~\ref{CSecond} (with $k=1$ and $l=2$). 

{\sf Case 2.} $\la$ is JS. 

{\sf Case 2.1.} $\zeta_2':=\pi_0\circ \zeta_2$ satisfies $[\im\zeta_2':D_2]\neq 0$. Then the proposition follows by Theorem~\ref{CFirst} (with $k=2$).

{\sf Case 2.2.} $\zeta_2'$ satisfies $[\im\zeta_2':D_2]=0$. Then $\zeta_2'':=\pi_1\circ \zeta_2$ satisfies $[\im\zeta_2'':D_2]\neq 0$.
In view of (iv), (v) and (vi), using Theorems~\ref{Tp=3} and \ref{Tp=2}, we may assume that there exists $\psi_3\in\Hom_{\SSS_n}(\I(G), M_3)$ with $[\im \psi_3:D_3]\neq 0$.

{\sf Case 2.2.1.} $\zeta_3':=\pi_0\circ \zeta_3$ satisfies $[\im\zeta_3':D_3]\neq 0$. Then the proposition follows by Theorem~\ref{CFirst} (with $k=3$).

{\sf Case 2.2.2.} $\zeta_3'$ satisfies $[\im\zeta_3':D_3]=0$. Then $\zeta_3'':=\pi_1\circ \zeta_3$ satisfies $[\im\zeta_3'':D_3]\neq 0$ and we can conclude by Theorem~\ref{CSecond} (with $k=2$ and $l=3$).
\end{proof}

\begin{Proposition} \label{P100518}%{\rm \cite{}}
Let $n\geq 8$ and exclude the case $(n,p)=(8,2)$. Suppose that $\la\in\Parinv_p(n)$ and $G\leq \AAA_n$ is a transitive subgroup such that $E^\la_+\da_G$ or $E^\la_-\da_G$ is irreducible. Then one of the following holds:
\begin{enumerate}
\item[{\rm (i)}] $G$ is $2$-homogeneous.
\item[{\rm (ii)}] $p=2$ and $\la=\be_n$.
\item[{\rm (iii)}]  $n$ is even and $G\leq G_{n/2,2}$ or $G\leq  G_{2,n/2}$. 
%\item[{\rm (iv)}] $p=3$, $n\equiv 0\pmod{3}$ and $(S_1^*)^G\not=0$.
\item[{\rm (iv)}] $p=3$, $n\equiv 1\pmod{3}$ and $\dim(S_2^*)^G>i_2(G)-1$.
%\item[{\rm (vi)}] $p=2\mid n$ and $\dim(S_1^*)^G\geq i_2(G)-1$.
\item[{\rm (v)}] $p=2\mid n$ and $\dim(S_2^*)^G>i_2(G)-1$.

\end{enumerate} 
\end{Proposition}
\begin{proof} The result for $p>3$ follows from \cite[Theorem 3.13]{KSAlt}, so let $p=2$ or $3$. We will now prove the following

\vspace{2mm}
\noindent
{\sf Claim.} We are in one of the cases (i)-(v) or one of the following conditions holds:

(a) $p=3$, $n\equiv 0\pmod{3}$ and $(S_1^*)^G\not=0$;

(b) $p=2\mid n$ and $\dim(S_1^*)^G\geq i_2(G)-1$.

\vspace{2mm}

To prove the claim, we assume that we are not in the cases (i)-(v) or (a),(b) and apply Theorems~\ref{CFirst} and \ref{CSecond} to deduce that $E^\la_{\pm}\da_G$ are reducible, obtaining a contradiction. 

Since we are not in case (i), we have $i_2(G)>1$. Since we are not in case (ii) we have that $D^\la$ is not basic spin in characteristic $2$. So $h(\la)\geq 3$ by Lemma~\ref{L080518}. Since we are not in case (iii), using \cite[\S5, Corollary]{CNS}, we have that $i_2(G)<i_3(G)$. 

Since we are not in case (b) and $i_2(G)>1$, Theorems \ref{Tp=3} and \ref{Tp=2} imply that there exists
$$\psi_2\in\Hom_{\SSS_n}(\I(G), M_2)$$ with $[\im \psi_2:D_2]\neq 0$. Since we are not in cases (iv),(v),(a) and $i_2(G)<i_3(G)$, 
Theorems \ref{Tp=3} and \ref{Tp=2} imply that there exists
$$\psi_3\in\Hom_{\SSS_n}(\I(G), M_3)$$ with $[\im \psi_3:D_3]\neq 0$.

Furthermore, by Lemma~\ref{C180418}, there exists 
$$\zeta_2\in\Hom_{\SSS_n}(M_2, \EE(\la))$$ with $[\im \zeta_2:D_2]\neq 0$. By Lemma~\ref{LM_3E}, there exists 
$$\zeta_3\in\Hom_{\SSS_n}(M_3, \EE(\la))$$ with $[\im \zeta_3:D_3]\neq 0$. 

%We now consider several cases. 

{\sf Case 1.} $\zeta_2':=\pi_0\circ \zeta_2$ satisfies $[\im \zeta_2':D_2]\neq 0$. The proposition now follows from Theorem~\ref{CFirst} (with $k=2$).

{\sf Case 2.} $\zeta_2'$ satisfies $[\im \zeta_2':D_2]=0$. Then $\zeta_2'':=\pi_1\circ \zeta_2$ satisfies $[\im \zeta_2'':D_2]\neq 0$. 

{\sf Case 2.1.} $\zeta_3':=\pi_0\circ \zeta_3$ satisfies $[\im \zeta_3':D_3]\neq 0$. The proposition now follows from Theorem~\ref{CFirst} (with $k=3$).

{\sf Case 2.2.} $\zeta_3'$ satisfies $[\im \zeta_3':D_3]=0$. Then $\zeta_3'':=\pi_1\circ \zeta_3$ satisfies $[\im \zeta_3'':D_3]\neq 0$ and so the proposition  follows from Theorem~\ref{CSecond} (with $k=2$ and $l=3$).

This completes the proof of the claim.

\vspace{2mm}
We now eliminate the exceptional cases (a) and (b) in the Claim. Indeed, if we are in one of those cases then, in view of (i), we may assume that  $(S_1^*)^G\neq 0$. If $G$ is primitive then by \cite[Corollary~2.32 and  Lemma 2.33]{KMTOne} this implies that $O_p(G)\neq 1$, in which case $E^\la_\pm\da_G$ is reducible for example by \cite[Lemma~2.19]{KMTOne}. In the imprimitive case, $G\leq G_{a,b}$ for some $a,b>1$ with $ab=n$. In view of (iii), we may assume that $a,b\neq 2$, in which case $(S_1^*)^{G_{a,b}}=0$ by Lemma~\ref{L100518_3}, and so by the Claim, $E^\la_\pm\da_{G_{a,b}}$ is reducible and so is  $E^\la_\pm\da_{G}$.
\end{proof}

\section{Imprimitive subgroups}

In this section we analyze restrictions $E^\la_\pm\da_G$ for imprimitive subgroups $G\leq \AAA_n$ and prove Theorems~\ref{TA}, \ref{TB} and \ref{TC}. 

\subsection{Intransitive subgroups}
In this subsection we deal with maximal intransitive subgroups $G\leq \AAA_n$, i.e. subgroups of the form $\AAA_{n-k,k}$. The following is easy to check. 

\begin{Lemma} \label{L100518}%{\rm \cite{}}
Let $n\geq 5$, $1< k\leq n/2$, and  $G=A_{n-k,k}$. Then $i_1(G)=2$, $i_2(G)=3$, $i_3(G)=4$ for $k>2$, and $i_3(G)=3$ for $k=2$. 
\end{Lemma}

\begin{Proposition} \label{P100518_2}%{\rm \cite{}}
Let $n\geq 8$ and exclude the case $(n,p)=(8,2)$. If $\la\in\Parinv_p(n)$, $G\leq \AAA_n$ is an intransitive subgroup and $E^\la_+\da_G$ or $E^\la_-\da_G$ is irreducible then one of the following holds:
\begin{enumerate}
\item[{\rm (i)}] $G\leq \AAA_{n-1}$.

\item[{\rm (ii)}] $G\leq \AAA_{n-2,2}$ and $\la$ is JS.

\item[{\rm (iii)}] $p=2$ and $\la=\be_n$.
\end{enumerate} 
\end{Proposition}
\begin{proof}
We may assume that $G=\AAA_{n-k,k}$ with $1< k\leq n/2$, in which case the result follows from Proposition \ref{PIntrans} and Lemmas \ref{L100518_2} and \ref{L100518}.
\end{proof}

The next result deals with the case $k=1$ when $p=2$ (for $p>2$ the corresponding result is known, see \S\ref{SSTB} below). 

\begin{Theorem}\label{TAn-1}
Let $p=2$ and $\la\in\Parinv_2(n)$. Then $E^\la_\pm\da_{\AAA_{n-1}}$ is irreducible if and only if one of the following holds:
\begin{enumerate}
\item[{\rm (i)}] $\la$ is JS;
\item[{\rm (ii)}] $\la$ has exactly two normal nodes, $\la_1=\la_2+1$ and $\la_1$ is even.
\end{enumerate}
\end{Theorem}

\begin{proof}
If $E^\la_\pm\da_{\AAA_{n-1}}$ is irreducible then $D^\la\da_{\SSS_{n-1}}$ has at most two composition factors and if it has two composition factors, they are isomorphic and they do not split when restricted to $\AAA_{n-1}$. In particular $\la$ has at most two normal nodes.

{\sf Case 1.} $\la$ is JS. Then $D^\la\da_{\SSS_{n-1}}$ is irreducible, and so $E^\la_\pm\da_{\AAA_{n-1}}$ is irreducible.

%{\em Case 2.} $\la$ has $2$ normal nodes of different residues. Then $D^\la\da_{\SSS_{n-1}}$ has 2 non-isomorphic composition factors, so in this case $E^\la_\pm\da_{\AAA_{n-1}}$ is not irreducible.

{\sf Case 2.} $\la$ has two normal nodes. 
Let $A:=(1,\la_1)$ and $B:=(2,\la_2)$.  
% of same residue. 
If $\la_A$ is $2$-regular, then by Lemma~\ref{Lemma39}, $D^\la\da_{\SSS_{n-1}}$ has two non-isomorphic composition factors, so in this case $E^\la_\pm\da_{\AAA_{n-1}}$ is not irreducible. So we may assume that $\la_1=\la_2+1$. Let $\mu=\la_{B}$. Then $[D^\la\da_{\SSS_{n-1}}:D^\mu]=2$ by Lemma~\ref{Lemma39}.

{\sf Case 2.1.} $\la_1$ is odd. Notice that $\la_1=\mu_1=\mu_2+2$ and $\mu_1+\mu_2\equiv 0\pmod{4}$. Since $\la\in\Parinv_2(n)$, it then follows that $\mu\in\Parinv_2(n-1)$. So in this case $E^\la_\pm\da_{\AAA_{n-1}}$ is not irreducible.

{\sf Case 2.2.} $\la_1$ is even. In this case $\mu_1+\mu_2\equiv 2\pmod{4}$, so  $\mu\not\in\Parinv_2(n-1)$. Moreover, by Lemma~\ref{Lemma39}
\[D^\la\da_{\SSS_{n-1}}\sim \underbrace{D^\mu}_{\soc}|\underbrace{\ldots}_{\text{ no }D^\mu}|\underbrace{D^\mu}_{\head}.\]
Using Frobenius reciprocity, for any $\nu\in\Par_2(n-1)$, we have
$$
\dim\Hom_{\AAA_{n-1}}(E^\nu_{(\pm)},D^\la\da_{\AAA_{n-1}})=
\dim\Hom_{\SSS_{n-1}}(E^\nu_{(\pm)}\uparrow^{\SSS_{n-1}},D^\la\da_{\SSS_{n-1}})
$$
Moreover, $E^\nu_\pm\uparrow^{\SSS_{n-1}}\cong D^\nu$ if $\nu\in\Parinv_2(n-1)$, and 
$E^\nu\uparrow^{\SSS_{n-1}}= D^\nu|D^\nu$ otherwise. 
So $\soc(D^\la\da_{\AAA_{n-1}})\cong (E^\mu)^{\oplus k}$ for  $k\leq 2$. But $\la\in\Parinv_2(n)$, so $\soc(D^\la\da_{\AAA_{n-1}})\cong E^\mu\oplus E^\mu$. By Frobenius reciprocity again, we now conclude that $E^\mu\uparrow^{\SSS_{n-1}} = D^\mu|D^\mu$ is a submodule of $D^\la\da_{\SSS_{n-1}}$. Hence $D^\la\da_{\SSS_{n-1}}\cong D^\mu|D^\mu$ and $E^\la_\pm\da_{\AAA_{n-1}}$ is irreducible.
\end{proof}

We now deal with the case $k=2$ when $p=2$ and $\la$ is not basic spin (for $p>2$ the corresponding result is known, see \S\ref{SSTC} below). 

\begin{Theorem} \label{T221118} %{\rm \cite{}}
Let $p=2$ and $\la \in\Parinv_2(n)\setminus\{ \be_n\}$. Then $E^\la_\pm\da_{\AAA_{n-2,2}}$ is irreducible if and only if $\la$ is JS, in which case $E^\la_\pm\da_{\AAA_{n-2}}\cong E^{\tilde e_{1-i}\tilde e_i \la}$ where $i$ is the residue of $(1,\la_1)$.
\end{Theorem}
\begin{proof}
If $E^\la_\pm\da_{\AAA_{n-2,2}}$ is irreducible, then $\la$ is JS by Proposition \ref{P100518_2}. %So all parts have the same parity. Conditions (i), (ii) and (iii) then follow from \cite[Theorem 1.1]{Benson}.
Conversely, let $\la$ be JS.  Note that $h(\la)\geq 2$ and $\mu:=\tilde e_{1-i}\tilde e_i \la=(\la_1-1,\la_2-1,\la_3,\la_4,\ldots)$. In particular $\mu_1+\mu_2=\la_1+\la_2-2\equiv 2\pmod{4}$ and so $\mu\not\in\Parinv_2(n-2)$. Since $\la$ is JS it follows that $D^\la\da_{\SSS_{n-1}}\cong D^{\tilde{e}_i\la}$ by Lemma \ref{Lemma39}. Further $\eps_i(\tilde{e}_i\la)=0$ and $\eps_{1-i}(\tilde{e}_i\la)=2$. So
\[D^\la\da_{\SSS_{n-2}}\sim \underbrace{D^{\mu}}_{\soc}|\underbrace{\ldots}_{\text{no }D^{\mu}}|\underbrace{D^{\mu}}_{\head}.\]
Using Frobenius reciprocity as in the proof of Theorem~\ref{TAn-1}, we deduce that %$\soc(D^\la\da_{\AAA_{n-2}})\cong E^{\mu}\oplus E^{\mu}$. So there exists $M\subseteq D^\la\da_{\SSS_{n-2}}$ with 
$E^\mu\uparrow^{\SSS_{n-2}}= D^{\mu}|D^{\mu}$ is a submodule of $D^\la\da_{\SSS_{n-2}}$. Hence $D^\la\da_{\SSS_{n-2}}= D^{\mu}|D^{\mu}$, so $E^\la_\pm\da_{\AAA_{n-2}}$ is irreducible.
%In particular $D^\la\da_{\AAA_{n-2}}\sim E^{\mu}|E^{\tilde e_{1-i}\tilde e_i \la}$ and then $E^\la_\pm\da_{\AAA_{n-2}}\cong E^{\tilde e_{1-i}\tilde e_i \la}$ is irreducible and then $E^\la_\pm\da_{\AAA_{n-2,2}}$ is also irreducible.
%
%
% Further if $E^\la_\pm\da_{\AAA_{n-2,2}}$ is irreducible then $D^\la\da_{\SSS_{n-2}}$ has at most $2$ composition factors. Notice that $\eps_i(\tilde e_i\la)=0$ since $\la$ is JS and by Lemma \ref{Lemma39}. Further, since $h(\la)\geq 3$ by Lemma \ref{L080518}, it follows that $\tilde e_i\la$ is not a JS partition. In particular by Lemma \ref{Lemma39}, $\eps_{1-i}(\tilde e_i(\la))=2$ and (iv) holds.
%
\end{proof}

\subsection{Proof of Theorem~\ref{TB}}\label{SSTB}
Suppose first that $p>2$. By \cite[Proposition 3.7]{KSAlt} (see also \cite[Theorem 5.10]{BO2}), $E^\la_\pm\da_{A_{n-1}}$ is irreducible if and only if $\la$ is JS or $\la$ has two normal nodes of different residues. Under the assumption $\la\in\Parinv_p(n)$ this is equivalent to the requirement that the two normal nodes have residues different from $0$. On the other hand, if $p=2$ then by Theorem~\ref{TAn-1}, $E^\la_\pm\da_{A_{n-1}}$ is irreducible if and only if $\la$ is JS or $\la$ has two normal nodes and $\la_1=\la_2+1$ is even. It remains to show that the latter condition holds if and only if $\la$ has exactly two normal nodes of residue $1$. The `only-if' part is clear. For the `if' part, suppose that $\la$ has exactly two normal nodes of residue $1$. Since the top removable node is always normal it follows that $\la_1$ is even. Since $\la\in\Parinv_2(n)$ it then follows that  $\la_1=\la_2+1$.

\subsection{Proof of Theorem~\ref{TC}}\label{SSTC}
If $p>2$, the result is \cite[Theorem 3.6]{KSAlt}. If $p=2$ use Theorem~\ref{T221118}. 
%Otherwise use Proposition~\ref{P221118}. 

\subsection{Transitive imprimitive subgroups}
In this section we begin to investigate restrictions to the maximal transitive imprimitive subgroups $G\leq \AAA_n$, i.e. subgroups of   the form $G_{a,b}$ with $a,b\geq 2$ and $n=ab$.

\begin{Proposition}\label{P110518_2}%{\rm \cite{}}
Let $n=2b\geq 8$, $\la\in\Parinv_p(n)$. Then $E^\la_\pm\da_{G_{2,b}}$ are reducible. 
\end{Proposition}
\begin{proof}
For $p>3$ this is known, see \cite[Main Theorem]{KSAlt}. For $p=2$, this is clear since $G_{2,b}$ has a non-trivial normal $2$-subgroup, cf. \cite[Lemma 2.19]{KMTOne}. Let $p=3$. It suffices to prove that $D^\la\da_{\SSS_2\wr\SSS_b}$ has at least three composition factors. Let $\SSS_{(2^b)}\cong \SSS_2\times\dots\times \SSS_2$ be the base subgroup, with generators $g_1,\dots,g_b$ of order $2$. The irreducible $\F\SSS_{(2^b)}$-modules are of the form
$
\{L(\de_1,\dots,\de_b)\mid \de_1,\dots,\de_b\in\{0,1\}\}
$
where $L(\de_1,\dots,\de_b)=\F\cdot v$ and $g_rv=(-1)^{\de_r}v$ for $r=1,\dots,b$. Restriction of any irreducible $\F(\SSS_2\wr\SSS_b)$-module to $\SSS_{(2^b)}$ is a direct sum of modules of the form $L(\de_1,\dots,\de_b)$ with {\em fixed} $\de_1+\dots+\de_b$. So it suffices to prove that $D^\la\da_{\SSS_{(2^b)}}$ has three composition factors  $L(\de_1,\dots,\de_b)$ with three different sums $\de_1+\dots+\de_b$. Note that $D^{(3,1)}\da_{\SSS_{(2^b)}}\cong L(0,0)\oplus L(0,1)\oplus L(1,0)$ and $D^{(2,1,1)}\da_{\SSS_{(2^b)}}\cong L(1,1)\oplus L(0,1)\oplus L(1,0)$. Moreover, 
by Lemmas~\ref{L080518} and \ref{LKZ}, 
the restriction
$D^\la\da_{\SSS_{4}\times \SSS_{4}}$ has a composition factor of the form $L_1\boxtimes L_2$ with $L_1,L_2\in \{D^{(3,1)},D^{(2,1,1)}\}$.
Hence $D^\la\da_{\SSS_{(4,4,2^{b-4})}}$ has a a composition factor of the form $L_1\boxtimes L_2\boxtimes L(\de_5,\dots,\de_b)$. Restricting this module to  $\SSS_{(2^b)}$ yields composition factors of the form $L(\eta_1,\eta_2,\eta_3,\eta_4,\de_5,\dots,\de_b)$ with fixed $\de_5,\dots,\de_b$ and at least three different sums $\eta_1+\eta_2+\eta_3+\eta_4$. 
\end{proof}

\begin{Proposition}\label{P110518_3}
Let $n=2a$ with $a\geq 4$. Let $G=G_{a,2}$ and $\la\in\Parinv_p(n)$. If $E^\la_+\da_G$ or $E^\la_-\da_G$ is irreducible then 
%one of the following holds:
%\begin{enumerate}
%\item[{\rm (i)}] $p=3$, $\la$ is JS.
%\item[{\rm (i)}] 
$p=2$ and $\la$ has at most three normal nodes.
%\item[{\rm (ii)}] $p=2$ and $\la=\be_n$.
%\end{enumerate}
\end{Proposition}

\begin{proof}
For $p\geq 5$ the proposition holds by \cite[Proposition 3.12]{KSAlt}. 
So we may assume that $p=2$ or $p=3$. 
Without loss of generality we assume that $E^\la_+\da_G$ is irreducible.

Let $p=2$. We may assume that $\la$ has at least four normal nodes. Hence by Lemma~\ref{Lemma39}(v), 
$$
\dim \Hom_{\SSS_n}(M_1,\EE(\la))=\dim \End_{\SSS_{n-1}}(D^\la\da_{\SSS_{n-1}})\geq 4
$$
Since $M_1\sim D_0|S_1^*$ for example Theorem~\ref{Tp=2}, it follows that $\dim\Hom(S_1^*,\EE(\la))\geq 3$. So by  Lemma~\ref{Lemma4.13},
$$\dim\Hom(S_1^*\da_{\AAA_n},\EE_{+,+}(\la)\oplus\EE_{-,-}(\la))\geq 3.$$
But $S_1^*=D_1$ or $S_1^*=D_1|D_0$, so there exist $\zeta,\zeta'\in\Hom_{\AAA_n}(S_1^*\da_{\AAA_n},\EE_{+,+}(\la)\oplus \EE_{+,-}(\la))$ such that $\zeta|_{E_1}$ and $\zeta'|_{E_1}$ are linearly independent.

On the other hand, by Lemma~\ref{L100518_3}(i) there exists $\psi\in\Hom_{\AAA_n}(\J(G),S_1^*\da_{\AAA_n})$ with $[\im\psi:E_1]\neq 0$. Note also that there exists $\xi\in\Hom_{\AAA_n}(\J(G),\EE_{+,+}(\la)\oplus \EE_{+,-}(\la))$ with $\im\xi\cong E_0$. Note that $\xi$, $\zeta\circ\psi$ and $\zeta'\circ\psi$ are linearly independent %(any non-zero linear combination of $\zeta\circ\psi$ and $\zeta'\circ\psi$ has $E_1$ as composition factor of its image). T
which contradicts $E^\la_+$ being irreducible, due to Lemma \ref{LBasicAlt}.

If $p=3$ then $M_2\sim M_1|S_2^*$ by Theorem \ref{Tp=3}. Now,   $i_1(G)=1<2=i_2(G)$ imply that there exists $\psi\in\Hom_{\SSS_n}(\I(G),M_2)$ with $[\im\psi:D_2]\neq 0$.

Assume first that  $\la$ is not JS. Then by \cite[Theorem 3.3]{KS2Tran} and \cite[Lemmas 4.9, 4.11, 4.12]{M2} we have 
\begin{align*}
\dim\Hom_{\SSS_n}(M_2,\EE(\la))&=\dim\End_{\SSS_{n-2,2}}(D^\la\da_{\SSS_{n-2,2}})\\
&\geq
\dim\End_{\SSS_{n-1}}(D^\la\da_{\SSS_{n-1}})+2
\\&=
\dim\Hom_{\SSS_n}(M_1,\EE(\la))+2.
\end{align*}
So from Lemma \ref{Lemma4.13} 
\[\dim\Hom_{\AAA_n}(M_2\da_{\AAA_n},\EE_{+,+}(\la)\oplus \EE_{+,-}(\la))\geq\dim\Hom_{\AAA_n}(M_1\da_{\AAA_n},\EE_{+,+}(\la)\oplus \EE_{+,-}(\la))+2.\]
Since $M_2\sim S_2|M_1$ by Theorem \ref{Tp=3}, we now deduce that there exist homomorphisms $\zeta,\zeta'\in\Hom_{\AAA_n}(M_2\da_{\AAA_n},\EE_{+,+}(\la)\oplus \EE_{+,-}(\la))$ whose restrictions to $S_2\da_{\AAA_n}$ %$\zeta|_{S_2\da_{\AAA_n}}$ and $\zeta'|_{S_2\da_{\AAA_n}}$ 
are linearly independent. Let further $\xi\in\Hom_{\AAA_n}(\J(G),\EE_{+,+}(\la)\oplus \EE_{+,-}(\la))$ with $\im\xi\cong E_0$. Then $\xi$, $\zeta\circ\psi$ and $\zeta'\circ\psi$ are linearly independent. This contradicts $E^\la_+$ being irreducible, due to Lemma \ref{LBasicAlt}.

Assume now that $\la$ is JS. Since $E^\la_+\da_G$ is irreducible, so is $E^\la_-\da_G$. In particular $D^\la\da_{\SSS_{n/2}\wr\SSS_2}$ has at most $2$ composition factors. So by the classification of the irreducible $\SSS_{n/2}\wr\SSS_2$-modules, we have that $D^\la\da_{\SSS_{n/2,n/2}}$ has at most $4$ composition factors. From Lemma \ref{L240818_2} we have that $D^\la\da_{\SSS_{n/2}}$ has at least $5$ non-isomorphic composition factors, leading to a contradiction.

\end{proof}

\begin{Proposition} \label{P110518}%{\rm \cite{}}
Let $n\geq 8$ and exclude the case $(n,p)=(8,2)$. 
If $\la\in\Parinv_p(n)$, $a,b\geq 2$ with $ab=n\geq 8$, and $E^\la_+\da_{G_{a,b}}$ or $E^\la_-\da_{G_{a,b}}$ is irreducible then $p=2$ and one of the following holds:
\begin{enumerate}
%\item[{\rm (i)}] $p=3$, $b=2$ and $\la$ is JS.

\item[{\rm (i)}] $b=2$ and $\la$ has at most three normal nodes.

\item[{\rm (ii)}] $\la=\be_n$.
\end{enumerate} 
\end{Proposition}
\begin{proof}
By Propositions \ref{P110518_2} and \ref{P110518_3} we may assume that $a,b\geq 3$. Now by Proposition \ref{P100518}, we may assume that we are in the cases (iv) or (v) of that proposition. Since $i_2(G_{a,b})=2$, the proposition follows by Lemma~\ref{L100518_3}(ii).
\end{proof}

\subsection{Restrictions to $G_{a,2}$ for $p=2$}
The goal of this subsection is to eliminate the exceptional case which appears in Proposition~\ref{P110518}(i).

\begin{Lemma} \label{L200818}%{\rm \cite{}}
Let $n=2a\geq 4$, $\la\in\Parinv_2(n)$ and $G=G_{a,2}$. If  $E^\la_+\da_G$ or $E^\la_-\da_G$ is irreducible then $D^\la\da_{\SSS_a\wr\SSS_2}$ is either irreducible or it has exactly two composition factors which are isomorphic to each other. In particular, $D^\la\da_{\SSS_a}$ has at most two isomorphism classes of composition factors. 
\end{Lemma}
\begin{proof}
Note that 
  $E^\la_+\da_G$ is irreducible if and only if so is $E^\la_-\da_G$. Hence $D^\la\da_{\SSS_a\wr\SSS_2}$ is either irreducible or has exactly two composition factors. If $D^\la\da_{\SSS_a\wr\SSS_2}$ is irreducible then $D^\la\da_{\SSS_{a,a}}$ is of the form $D^\mu\boxtimes D^\mu$ or $D^\mu\boxtimes D^\nu\,\oplus\, D^\nu\boxtimes D^\mu$, and so all composition factors of $D^\la\da_{\SSS_a}$ are isomorphic to $D^\mu$ or $D^\nu$.

Suppose that $D^\la\da_{\SSS_a\wr\SSS_2}$ has two composition factors $L_+$ and $L_-$. We may assume that 
$L_\pm\da_G\cong E^\la_\pm\da_G$. Then, since $(1,2) \in \SSS_a \wr\SSS_2$, we get 
$$L_-\da_G \cong E^\la_-\da_G\cong (E^\la_+\da_G)^{(1,2)}\cong (L_+\da_G)^{(1,2)}
\cong L_+^{(1,2)}\da_G\cong L_+\da_G. 
$$
As $p=2=[\SSS_a\wr\SSS_2:G]$, it now follows from Clifford theory that $L_+\cong L_-$,
and we are done as in the previous paragraph. 
\end{proof}

Combining Lemmas~\ref{L200818} and \ref{L221118} allows us to assume that we are in one of the exceptional cases of  Lemma~\ref{L221118}. The next lemma deals with the exceptional case (ii) of Lemma~\ref{L221118}.

\begin{Lemma}\label{L221118_3}
Let $p=2\mid n\geq 6$ and $\la=(\be_{n-1},1)$. Then $E^\la_\pm\da_{G_{n/2,2}}$ is reducible.
\end{Lemma}

\begin{proof}
Assume that $E^\la_\pm\da_{G_{n/2,2}}$ is irreducible. 
On the other hand, by \cite[Theorem A]{KMTOne}, $D^\la\da_{\SSS_{n/2}\wr\SSS_2}$ is reducible. By Lemma~\ref{L200818}, we conclude that in the Grothendieck group we have either
\[[D^\la\da_{\SSS_{n/2,n/2}}]=2[D^\mu\boxtimes D^\nu]+2[D^\nu\boxtimes D^\mu]\]
for some distinct $\mu,\nu\in\Par_2(n/2)$ or
\[[D^\la\da_{\SSS_{n/2,n/2}}]=2[D^\mu\boxtimes D^\mu]\]
for some $\mu\in\Par_2(n/2)$. 
For $n\leq 10$, using \cite[Tables]{JamesBook}, one checks that neither of these ever happens.
Let now $n\geq 12$. It is easy to see by repeatedly applying Lemma \ref{Lemma39} that $D^{\be_{n/2}}$ and $D^{(\be_{n/2-1},1)}$ are composition factors of $D^\la\da_{\SSS_{n/2}}$. So
\[[D^\la\da_{\SSS_{n/2,n/2}}]=2[D^{\be_{n/2}}\boxtimes D^{(\be_{n/2-1},1)}]+2[D^{(\be_{n/2-1},1)}\boxtimes D^{\be_{n/2}}].\]
In particular 
\begin{equation}\label{E221118}
\dim D^\la=4\dim D^{\be_{n/2}}\dim D^{(\be_{n/2-1},1)}.
\end{equation}

For any $m$, let $\langle m\rangle$ be a basic spin representation and $\langle m-1,1\rangle$ be a second basic spin representation of $\SSS_m$ in characteristic 0. From \cite[Theorem 1.2]{Benson} we have that $D^{\be_m}$ is a composition factor of $\langle m\rangle$ and that $D^{(\be_{m-1},1)}$ is a composition factor of $\langle m-1,1\rangle$, provided $\be_m$ and $(\be_{m-1},1)$ are $2$-regular. This leads to a contradiction with (\ref{E221118}) using \cite[Tables III and IV]{Wales}.
\end{proof}

The next result, whose proof is similar to that of \cite[Lemma 7.20]{KMTOne}, treats the exceptional case (iii) of Lemma~\ref{L221118}.
 
\begin{Lemma}\label{L221118_4}
Let $p=2$, $n\geq 24$, $n\equiv 0\pmod{8}$ and $\la=(n/4+3,n/4+1,n/4-1,n/4-3)$. Then $E^\la_\pm\da_{G_{n/2,2}}$ is reducible.
\end{Lemma}

\begin{proof}
Assume that $E^\la_\pm\da_{G_{n/2,2}}$ is irreducible. Let 
\begin{align*}
\mu&:=(n/8+3,n/8+1,n/8-1,n/8-3),\\
\nu&:=(n/8+2,n/8+1,n/8-1,n/8-2).
\end{align*} 
By Lemma \ref{L110518} or \cite[Lemma 3.14]{KMTOne},   $D^\mu$ and $D^\nu$ are composition factors of $D^\la\da_{\SSS_{n/2}}$.  
It then follows from Lemma \ref{L200818} that all composition factors of $D^\la\da_{\SSS_{n/2,n/2}}$ are of the form $D^\mu\boxtimes D^\nu$ or $D^\nu\boxtimes D^\mu$.

Let 
\begin{align*}
\pi&:=(n/8+2,n/8+1,n/8,n/8-1),\\
\psi&:=(n/8+1,n/8,n/8-1,n/8-2).
\end{align*}
By \cite[Lemma 1.11]{BK}, we have that $D^\pi\boxtimes D^\psi$ is a composition factor of $D^\la\da_{\SSS_{n/2+2,n/2-2}}$. As $\nu=\tilde e_i^2 \pi$, by Lemma \ref{Lemma39}, we have  that $D^\nu\boxtimes \bone_{\SSS_{1,1}}\boxtimes D^\psi$ is a composition factor of $D^\la\da_{\SSS_{n/2,1,1,n/2-2}}$. So $D^\psi$ is a composition factor of $D^\mu\da_{\SSS_{n/2-2}}$, 
which contradicts \cite[Lemma~3.7]{KMTOne}. 
\end{proof}

The next two lemmas deal with the exceptional case (v) of Lemma~\ref{L221118}.

\begin{Lemma}\label{L221118_5}
Let $p=2$, $n\geq 24$, $n\equiv 0\pmod{8}$, and $\la=(n/4+2,n/4+1,n/4-1,n/4-2)$ then $E^\la_\pm\da_{G_{n/2,2}}$ is reducible.
\end{Lemma}

\begin{proof}
Assume that $E^\la_\pm\da_{G_{n/2,2}}$ is irreducible. Let
\begin{align*}
\mu&:=(n/8+2,n/8+1,n/8-1,n/8-2),\\
\nu&:=(n/8+3,n/8+1,n/8-1,n/8-3).
\end{align*} 
By Lemmas \ref{L200818} and \ref{L200818_2}(i), all composition factors of $D^\la\da_{\SSS_{n/2,n/2}}$ are of the form $D^\mu\boxtimes D^\nu$ or $D^\nu\boxtimes D^\mu$. Let $\pi:=(n/8+3,n/8+1,n/8-1,n/8-2)$. From Lemma \ref{L200818_2}(iv), $D^\pi$ is a composition factor of $D^\la\da_{\SSS_{n/2+1}}$. In particular there exists $\psi\in\Par_2(n/2-1)$ such that $D^\pi\boxtimes D^\psi$ is a composition factor of $D^\la\da_{\SSS_{n/2+1,n/2-1}}$. Restricting this module to $\SSS_{n/2,1,n/2-1}$ we have by Lemma \ref{Lemma39} that $D^\mu\boxtimes\bone_{\SSS_1}\boxtimes D^\psi$ and $D^\nu\boxtimes\bone_{\SSS_1}\boxtimes D^\psi$ are composition factors of $D^\la\da_{\SSS_{n/2,1,n/2-1}}$. In particular $D^\psi$ is a composition factor of $D^\mu\da_{\SSS_{n/2-1}}$ and $D^\nu\da_{\SSS_{n/2-1}}$. Since $\nu$ is JS, Lemma \ref{Lemma39} gives that $\psi=(n/8+2,n/8+1,n/8-1,n/8-3)$, contradicting $D^\mu\da_{\SSS_{n/2-1}}$ also having a composition factor isomorphic to $D^\psi$, again using Lemma \ref{Lemma39}.
\end{proof}

\begin{Lemma}\label{L221118_6}
Let $p=2$, $n\geq 20$, $n\equiv 4\pmod{8}$, and $\la=(n/4+2,n/4+1,n/4-1,n/4-2)$. Then $E^\la_\pm\da_{G_{n/2,2}}$ is reducible.
\end{Lemma}

\begin{proof}
Assume that $E^\la_\pm\da_{G_{n/2,2}}$ is irreducible. Let
\begin{align*}
\mu&:=(n/8+5/2,n/8+1/2,n/8-1/2,n/8-5/2),\\
\nu&:=(n/8+3/2,n/8+1/2,n/8-1/2,n/8-3/2).
\end{align*} 
By Lemmas \ref{L200818} and \ref{L200818_2}(iii), in the Grothendieck group we have for some $a\in\Z_{>0}$:
\begin{equation}\label{E221118_2}
[D^\la\da_{\SSS_{n/2,n/2}}]=a[D^\mu\boxtimes D^\nu]+a[D^\nu\boxtimes D^\mu].
\end{equation}
%all composition factors of $D^\la\da_{\SSS_{n/2,n/2}}$ are of the form $D^\mu\boxtimes D^\nu$ or $D^\nu\boxtimes D^\mu$, and 
%$$[D^\la\da_{\SSS_{n/2,n/2}}:D^\mu\boxtimes D^\nu]=[D^\la\da_{\SSS_{n/2,n/2}}:D^\nu\boxtimes D^\mu].$$
In particular, all composition factors of $D^\la\da_{\SSS_{n/2}}$ are of the form $D^\mu$ or $D^\nu$. It follows using Lemma~\ref{Lemma39} that all composition factors of $D^\la\da_{\SSS_{n/2+1}}$ are of the form $D^\kappa$ with $\tilde e_i\kappa=\mu$ or $\nu$ for some $i$. Since any composition factor of $D^\la\da_{\SSS_k}$ for any $k\leq n$ is indexed by a partition with at most $4$ rows (for example by \cite[Lemma 3.7]{KMTOne}), it then follows that any composition factor of $D^\la\da_{\SSS_{n/2+1}}$ is of the form $D^\pi$ or $D^\psi$, where
\begin{align*}
\pi&:=(n/8+5/2,n/8+3/2,n/8-1/2,n/8-5/2),\\
\psi&:=(n/8+5/2,n/8+1/2,n/8-1/2,n/8-3/2).
\end{align*} 
Then, in the Grothendieck group,
\begin{align*}
[D^\la\da_{\SSS_{n/2+1,n/2-1}}]&=[D^\pi\boxtimes M]+[D^\psi\boxtimes N],
\end{align*} 
for certain modules $M,N$ of $\SSS_{n/2-1}$. Comparing this to
(\ref{E221118_2}) and 
using Lemma \ref{Lemma39}, we deduce 
\begin{align*}
[D^\la\da_{\SSS_{n/2,1,n/2-1}}]&=2[D^\mu\boxtimes\bone_{\SSS_1}\boxtimes M]+2[D^\mu\boxtimes\bone_{\SSS_1}\boxtimes N]+[D^\nu\boxtimes\bone_{\SSS_1}\boxtimes N]\\
&=a[D^\mu\boxtimes\bone_{\SSS_1}\boxtimes D^\nu\da_{\SSS_{n/2-1}}]+a[D^\nu\boxtimes\bone_{\SSS_1}\boxtimes D^\mu\da_{\SSS_{n/2-1}}].
\end{align*}
Notice that from Lemma \ref{Lemma39}, $D^\mu\da_{\SSS_{n/2-1}}\cong D^\gamma\oplus D^\delta$ with
\begin{align*}
\gamma&:=(n/8+5/2,n/8+1/2,n/8-3/2,n/8-5/2),\\
\delta&:=(n/8+3/2,n/8+1/2,n/8-1/2,n/8-5/2).
\end{align*}
In particular $D^\mu\boxtimes\bone_{\SSS_1}\boxtimes D^\gamma$ is a composition factor of $D^\la\da_{\SSS_{n/2,1,n/2-1}}$ and then $D^\gamma$ is a composition factor of $D^\nu\da_{\SSS_{n/2-1}}$, which contradicts Lemma~\ref{Lemma39} by a block argument.
\end{proof}

\begin{Proposition}\label{P221118_9} %{\rm \cite{}}%{\bf ()}
Let $p=2\mid n\geq 10$ and $\la\in\Parinv_2(n)$. If  $E^\la_\pm\da_{G_{n/2,2}}$ is irreducible then $\la=\be_n$.
\end{Proposition}
\begin{proof}
By Lemma~\ref{L200818}, we may also assume that $D^\la\da_{\SSS_{n/2}}$ has at most two isomorphism classes of composition factors. So by Lemma~\ref{L221118}, we may assume that we are in one of the exceptional cases (i)-(vi) of that lemma. The case (i) does not need to be considered since this is the case $\la=\be_n$. In the cases (iv) and (vi), $\la$ has four normal nodes, so we can exclude them by Proposition \ref{P110518}. The case (ii) is treated in Lemma~\ref{L221118_3}, the case (iii) is treated in Lemma~\ref{L221118_4}, and the case  (v) is treated in Lemmas~\ref{L221118_5} and \ref{L221118_6}.
\end{proof}

\subsection{Proof of Theorem~\ref{TA}}
\label{SSTA}
For $p>3$ the result follows from \cite[Main Theorem]{KSAlt}. So we may assume that $p=3$ or $2$. 
In view of the exceptions (i) and (iv) in Theorem~\ref{TA}, we may assume that $G$ is imprimitive and $\la\neq \be_n$ if $p=2$. 

The theorem is easily checked for $n=5,6,7$. Indeed, in view of Theorems~\ref{TB} and \ref{TC}, we may assume that $G$ is one of the following: $G_{2,3}, G_{3,2}, \AAA_{4,3}$. Moreover, we only have to consider the cases where either $p=2$ and $\la=(3,2,1)$ or $p=3$ and $\la=(4,1^2)$ or $(4,2,1)$. In the exceptional cases the restriction $E^\la_\pm\da_G$ are reducible since $\sqrt{|G|}\leq \dim E^\la_\pm$.

The case $(n,p)=(8,2)$ is also easy since in this case $\la\in\Parinv_2(8)\setminus\{\be_8\}$ implies $\la=(4,3,1)$ and $\dim E^\la_\pm=20$ and $G$ is contained in one of $\{\AAA_7,\AAA_{6,2},\AAA_{5,3},\AAA_{4,4},G_{2,4},G_{4,2}\}$. The case $G\leq \AAA_7$ is covered by the exceptional case (ii)(b) of Theorem \ref{TA}. The case $G\leq \AAA_{6,2}$ is excluded by Theorem \ref{T221118}. The cases $G\leq \AAA_{5,3}$, $G\leq \AAA_{4,4}$ and $G\leq G_{2,4}$ are excluded since then $\sqrt{|G|}<\dim E^{(4,3,1)}_\pm$. The case $G\leq G_{4,2}$ is excluded by Lemma~\ref{L221118_3}. From now we assume that $(n,p)\neq (8,2)$. 

If $G$ is intransitive then $G\leq \AAA_{n-k,k}$ for some $1\leq k\leq n/2$. By Proposition~\ref{P100518_2}, we may assume that either $k=1$, or $k=2$ and $\la$ is JS. It remains to apply Theorem~\ref{TB}. 

If $G$ is transitive then $G\leq G_{a,b}$ for some $a,b>1$ with $ab=n$. By Proposition~\ref{P110518}, we may assume that $p=2=b$ (and $\la$ has at most three normal nodes). Now, we  apply Proposition~\ref{P221118_9}. 

\section{Basic spin case}

In this section, we assume that $p=2$. 
Recall e.g. from \cite{Wales} that
\begin{equation}
\dim D^{\be_n}=
\left\{
\begin{array}{ll}
2^{(n-2)/2} &\hbox{if $n$ is even,}\\
2^{(n-1)/2} &\hbox{if $n$ is odd.}
\end{array}
\right.
\end{equation}
Moreover, $\be_n\in\Parinv_2(n)$ if and only $n\not \equiv 2\pmod{4}$ (although we consider a general $n$ in this section). %So throughout the section we assume that $n\not \equiv 2\pmod{4}$. 

\subsection{Restricting basic spin module to intransitive subgroups}

\begin{Lemma} \label{LPrev} %{\rm \cite{}}
Let $\nu=(n_1,\dots,n_h)$ be a composition of $n$ with $n_1,\dots,n_h>1$, and $D=D^{\la_1}\boxtimes\dots \boxtimes D^{\la_h}$ be an irreducible $\F \SSS_\nu$-module. Then $D\da_{\AAA_\nu}$ splits if and only if $\la_r\in\Parinv_2(n_r)$ for all $r=1,\dots,h$. 
\end{Lemma}
\begin{proof}
Suppose $\la_r\in\Parinv_2(n_r)$ for all $r=1,\dots,h$. Then the number of composition factors of $D\da_{\AAA_{n_1}\times \dots\times \AAA_{n_h}}$ is $2^h$. On the other hand $\AAA_{n_1}\times \dots\times \AAA_{n_h}$ is a normal subgroup of $\AAA_\nu$ of index $2^{h-1}$, so $D\da_{\AAA_\nu}$ must have at least two composition factors. 

Conversely, suppose that $\la_r\not\in\Parinv_2(n_r)$ for some $r$. Without loss of generality, we may assume that  
$\la_1\in\Parinv_2(n_1), \dots, \la_s\in\Parinv_2(n_s)$ and $\la_{s+1}\not\in\Parinv_2(n_{s+1}), \dots, \la_h\not\in\Parinv_2(n_h)$ for some $0\leq s<h$. Then 
$$
D\da_{\AAA_{n_1}\times \dots\times \AAA_{n_h}}
=\bigoplus_{\eps_1,\dots\eps_s\in\{+,-\}} E^{\la_1}_{\eps_1}\boxtimes\dots \boxtimes E^{\la_s}_{\eps_s}\boxtimes 
E^{\la_{s+1}}\boxtimes\dots 
\boxtimes E^{\la_h}.
$$
So any submodule of  $D\da_{\AAA_{n_1}\times \dots\times \AAA_{n_h}}$ is the direct sum of some of the summands in the right hand side. But if one such summand lies in an $\F\AAA_\nu$-submodule of $D\da_{\AAA_\nu}$ then all of them do (this can be seen by conjugating with elements of $\AAA_\nu$ having odd components in some of the first $s$ positions and, if necessary, an odd component in one of the remaining positions). 
\end{proof}

\begin{Proposition} \label{P221118} %{\rm \cite{}}
Let $n\geq 5$, $\nu=(n_1,\dots,n_h)$ be a composition of $n$ with $h>1$. Then $E^{\be_n}_{(\pm)}\da_{\AAA_\nu}$ is irreducible if and only if one of the following conditions holds:
\begin{enumerate}
\item[{\rm (1)}] $n\equiv 0\pmod{4}$, $h=3$, $n_r\equiv 2\pmod{4}$ for exactly one $r$, and the other two parts of $\nu$ are odd;
\item[{\rm (2)}] $n\equiv 0\pmod{4}$, $h=2$, and $n_1$, $n_2$ are both odd;
\item[{\rm (3)}] $n\not\equiv 2\pmod{4}$, $h=2$, and $n_r\equiv 2\pmod{4}$ for at least one $r$.
\end{enumerate}
\end{Proposition}
\begin{proof}
We may assume that $n\not\equiv 2\pmod{4}$, since otherwise $E^{\be_n}=D^{\be_n}\da_{\AAA_n}$, and by \cite[Theorem C]{KMTOne}, $\AAA_\nu\leq \AAA_{n-k,k}$ with $n-k$ and $k$ odd, in which case $E^{\be_n}\da_{\AAA_{(n-k,k)}}$ is reducible by Lemma~\ref{LPrev}, hence  $E^{\be_n}\da_{\AAA_\nu}$ is also reducible.

All composition factors of $D^{\be_n}\da_{\SSS_\nu}$ are isomorphic to $D^{\be_{n_1}}\boxtimes\dots\boxtimes D^{\be_{n_h}}$, and by dimensions, we have
\begin{equation}\label{E230818}
[D^{\be_n}\da_{\SSS_\nu}:D^{\be_{n_1}}\boxtimes\dots\boxtimes D^{\be_{n_h}}]=2^{\lfloor\frac{n-1}{2}\rfloor - \lfloor\frac{n_1-1}{2}\rfloor-\dots- \lfloor\frac{n_h-1}{2}\rfloor}.
\end{equation}
If the last expression is greater than $2$, we must have that $E^{\be_n}_\pm\da_{\AAA_\nu}$ is reducible. So we may assume that it is at most $2$, which leaves us with the following cases:
\begin{enumerate}
\item[{\rm (a)}] $h=4$ and all $n_r$'s are odd;
\item[{\rm (b)}] $h=3$ and $n_r$ is even for at most one $r$;
\item[{\rm (c)}] $h=2$.
\end{enumerate}

In the case (a) the restriction $E^{\be_n}_{\pm}\da_{\AAA_\nu}$ is reducible, since by (\ref{E230818}), we have that $D^{\be_n}\da_{\SSS_\nu}$ has two composition factors, and these split when restricted to $\AAA_\nu$ by Lemma \ref{LPrev}.

In the case (b) $D^{\be_n}\da{\SSS_\nu}$ has exactly two composition factors by (\ref{E230818}). Suppose first that $n_r\not\equiv 2\pmod{4}$ for all $r$. In this case $E^{\be_n}_{\pm}\da_{\AAA_\nu}$ is reducible by the argument as in the previous paragraph. Without loss of generality we may then assume that $n_1\equiv 2\pmod{4}$ and that $n_2$ and $n_3$ are odd. In this case $(D^{\be_{n_1}}\boxtimes D^{\be_{n_2}}\boxtimes D^{\be_{n_3}})\da_{\AAA_\nu}$ does not split by Lemma \ref{LPrev}. So $D^{\be_n}\da_{\AAA_\nu}$ has exactly two composition factors and then $E^{\be_n}_{\pm}\da_{\AAA_\nu}$ is irreducible.

In the case (c) assume first that both $n_1$ and $n_2$ are odd. Then $D^{\be_n}\da{\SSS_\nu}$ is irreducible by (\ref{E230818}). So $D^{\be_n}\da{\AAA_\nu}$ has at most two composition factors and then $E^{\be_n}_{\pm}\da{\AAA_\nu}$ is irreducible. So we may assume that at least one of $n_1$, $n_2$ is even. In this case $D^{\be_n}\da{\SSS_\nu}$ has exactly two composition factors by (\ref{E230818}). If $n_1,n_2\not\equiv 2\pmod{4}$ then $(D^{\be_{n_1}}\boxtimes D^{\be_{n_2}})\da_{\AAA_\nu}$ splits by Lemma \ref{LPrev} and so $E^{\be_n}_{\pm}\da{\AAA_\nu}$ is reducible. Otherwise we may assume without loss of generality that $n_1\equiv 2\pmod{4}$. In this case $(D^{\be_{n_1}}\boxtimes D^{\be_{n_2}})\da_{\AAA_\nu}$ does not split by Lemma \ref{LPrev}. So $D^{\be_n}\da_{\AAA_\nu}$ has exactly two composition factors and then $E^{\be_n}_{\pm}\da_{\AAA_\nu}$ is irreducible.
\end{proof}

\subsection{Restricting basic spin module to transitive imprimitive subgroups}
Throughout this subsection, $a,b\in\Z_{\geq 2}$ with $ab=n$. We investigate when the restriction $E^{\be_n}_\pm\da_{G_{a,b}}$ is irreducible. %We begin with the analogous problem for symmetric groups.  

A special role will be played by the irreducible $\F (\SSS_a\wr \SSS_b)$-modules of the form $D^\mu\wr D^\nu$ for $\mu\in\Par_2(a)$ and $\nu\in\Par_2(b)$. As a vector space, $D^\mu\wr D^\nu=(D^\mu)^{\otimes b}\otimes D^\nu$, and the action on $v_1\otimes \dots\otimes v_b\otimes w\in (D^\mu)^{\otimes b}\otimes D^\nu$ is determined from the following requirements: $(g_1,\dots,g_b)\in \SSS_a\times\dots\times \SSS_a$ acts as 
$$
(g_1,\dots,g_b)\cdot (v_1\otimes \dots \otimes v_b\otimes w)=(g_1v_1)\otimes \dots \otimes (g_bv_b)\otimes w,
$$
and $h\in\SSS_b$ acts as
$$
h\cdot (v_1\otimes \dots \otimes v_b\otimes w)=v_{h^{-1}(1)}\otimes \dots \otimes v_{h^{-1}(b)}\otimes hw.
$$

\begin{Lemma} \label{L220818} %{\rm \cite{}}
All composition factors of the restriction $D^{\be_n}\da_{\SSS_a\wr\SSS_b}$ are of the form $D^{\be_a}\wr D^{\be_b}$, and
$$
[D^{\be_n}\da_{\SSS_a\wr\SSS_b}: D^{\be_a}\wr D^{\be_b}]=
\left\{
\begin{array}{ll}
2^{b/2} &\hbox{if $a$ is even and $b$ is even,}\\
2^{(b-1)/2} &\hbox{if $a$ is even and $b$ is odd,}\\
1 &\hbox{if $a$ is odd.}
\end{array}
\right.
$$
\end{Lemma}
\begin{proof}
The first statement is established in the course of proving \cite[Lemma 7.19]{KMTOne}. The second one follows by dimensions taking into account that $\dim D^{\be_a}\wr D^{\be_b}=\dim  D^{\be_b}(\dim D^{\be_a})^b$. 
\end{proof}

\begin{Proposition}\label{P261118} %{\rm \cite{}}
Suppose that $n\not \equiv 2\pmod{4}$. 
The restriction $E^{\be_n}_{\pm}\da_{G_{a,b}}$ is irreducible if and only if one of the following conditions holds:
\begin{enumerate}
\item[{\rm (i)}] $a$ is odd; 
\item[{\rm (ii)}] $a\equiv 2\pmod{4}$ and $b=2$. 
\end{enumerate}
\end{Proposition}
\begin{proof}
We consider the following cases.

{\sf Case 1.} {\em $a$ is odd}. Then  $b\not\equiv 2\pmod{4}$. By Lemma~\ref{L220818}, we have that $D^{\be_n}\da_{\SSS_a\wr\SSS_b}\cong D^{\be_a}\wr D^{\be_b}$. Since $D^{\be_n}$ splits and $G_{a,b}$ is an index $2$ subgroup of $\SSS_a\wr\SSS_b$, it follows that $D^{\be_a}\wr D^{\be_b}\da_{G_{a,b}}$ is a direct sum of two irreducible modules and so  $E^{\be_n}_{\pm}\da_{G_{a,b}}$ is irreducible, giving  case (i). 

{\sf Case 2.} {\em $a$ is even and  $b$ is even.} As $G_{a,b}$ is an index $2$ subgroup of $\SSS_a\wr\SSS_b$, by Lemma~\ref{L220818}, we may assume that $b=2$, in which case we have $[D^{\be_{n}}\da_{\SSS_{n/2}\wr\SSS_2}: D^{\be_{a}}\wr D^{\be_2}]=2$. Note that $D^{\be_2}\cong \bone_{\SSS_2}$. 

{\sf Case 2.1.} $a\equiv 2\pmod{4}$. In this case $D^{\be_{a}}\da_{\AAA_{a}}$ is irreducible, so the restriction $(D^{\be_{a}}\wr D^{\be_2})\da_{\AAA_{a}\times \AAA_{a}}\cong D^{\be_{a}}\boxtimes D^{\be_{a}}$ is irreducible. Hence 
$(D^{\be_{a}}\wr D^{\be_2})\da_{G_{a,2}}$ is irreducible, as $\AAA_{a}\times \AAA_{a}\leq G_{a,2}$. It follows that $E^{\be_{n}}_\pm \da_{G_{a,2}}$ is irreducible, giving  case (ii). 

{\sf Case 2.2.} $a\equiv 0\pmod{4}$.  We claim that in this case $E^{\be_n}_{\pm}\da_{G_{a,2}}$ is reducible. To prove this it suffices to show that $(D^{\be_{a}}\wr D^{\be_2})\da_{G_{a,2}}$ is reducible. If $(D^{\be_{a}}\wr D^{\be_2})\da_{G_{a,2}}$ was irreducible, restricting further to the subgroup $\AAA_{a}\wr S_2\leq G_{a,2}$ would give at most two composition factors, but we claim that $(D^{\be_{a}}\wr D^{\be_2})\da_{\AAA_{a}\wr S_2}$ has three. To see this, note that
$$
(D^{\be_{a}}\wr D^{\be_2})\da_{\AAA_{a}\times \AAA_{a}}\cong E^{\be_{a}}_+\boxtimes E^{\be_{a}}_+\,\oplus\, 
E^{\be_{a}}_-\boxtimes E^{\be_{a}}_-\,\oplus\,
E^{\be_{a}}_+\boxtimes E^{\be_{a}}_-\,\oplus\,
E^{\be_{a}}_-\boxtimes E^{\be_{a}}_+.
$$
It now follows from the classification of irreducible modules over wreath products that $(D^{\be_{a}}\wr D^{\be_2})\da_{\AAA_{a}\wr S_2}$ has composition factors $E_+,E_-,E$ such that 
$
E_{\pm}\da_{\AAA_{a}\times \AAA_{a}}\cong E^{\be_{a}}_\pm\boxtimes E^{\be_{a}}_\pm,
$
 and 
$
E\da_{\AAA_{a}\times \AAA_{a}}\cong E^{\be_{a}}_+\boxtimes E^{\be_{a}}_-\,\oplus\,
E^{\be_{a}}_-\boxtimes E^{\be_{a}}_+.
$

{\sf Case 3.} {\em $a$ is even and $b$ is odd.} In this case by the assumption $n\not \equiv 2\pmod{4}$ we have $a\equiv 0\pmod{4}$. As $G_{a,b}$ is an index $2$ subgroup of $\SSS_a\wr\SSS_b$, by Lemma~\ref{L220818}, we may assume that $b=3$, in which case we have $[D^{\be_n}\da_{\SSS_{a}\wr\SSS_3}: D^{\be_{a}}\wr D^{\be_3}]=2$. We claim that in this case $E^{\be_n}_{\pm}\da_{G_{a,b}}$ is reducible. To prove this it suffices to show that $(D^{\be_{a}}\wr D^{\be_3})\da_{G_{a,3}}$ is reducible. For that note first that
\[\AAA_{a}\times\AAA_{a}\times\AAA_{a}\unlhd\AAA_{a,a,a}\unlhd G_{a,3}\]
and that $[\AAA_{a,a,a}:\AAA_{a}\times\AAA_{a}\times\AAA_{a}]=4$ and $[G_{a,3}:\AAA_{a,a,a}]=6$. Also note that 
 $D^{\be_3}$ has dimension $2$, and $D^{\be_{a}}\da_{\AAA_{a}}$ splits since  
 $a\equiv 0\pmod{4}$, 
 so 
 \begin{align*}
 (D^{\be_{a}}\wr D^{\be_3})\da_{A_{a}\times A_{a}\times A_{a}}\cong &\bigoplus_{\eps_1,\eps_2,\eps_3\in\{+,-\}}(E^{\be_{a}}_{\eps_1}\boxtimes E^{\be_{a}}_{\eps_2}\boxtimes E^{\be_{a}}_{\eps_3})^{\oplus 2}.
\end{align*}
In particular $(D^{\be_{a}}\wr D^{\be_3})\da_{A_{a}\times A_{a}\times A_{a}}$ has $16$ composition factors all of the same dimension.

If $(D^{\be_{a}}\wr D^{\be_3})\da_{G_{a,3}}$ was irreducible, then $(D^{\be_{a}}\wr D^{\be_3})\da_{\AAA_{a,a,a}}$ would have $k$ composition factors all of the same dimensions with $k\mid 6$. From the previous paragraph it then follows that there exists $l\mid 4$ such that the restriction of any of these $k$ composition factors has $l$ composition factors. In particular $kl=16$, which leads to a contradiction, since $k\mid 6$ and $l\mid 4$. 
\end{proof}
 
% For completeness w
We consider the case $n\equiv 2\pmod{4}$ for completeness, even though it is not needed for the proof of the mains theorems.
%, where $D^{\be_n}$ does not split. 

\begin{Proposition}%\label{}%{\rm \cite{}}
Let $n\equiv 2\pmod{4}$. Then $E^{\be_n}\da_{G_{a,b}}$ is irreducible if and only of $a$ is odd. 
\end{Proposition}
\begin{proof}
If $a$ is even then even $D^{\be_n}\da_{\SSS_a\wr\SSS_b}$ is reducible by \cite[Theorem C]{KMTOne}, so %$E^{\be_n}\da_{G_{a,b}}$ since $G_{a,b}\leq \SSS_a\wr\SSS_b$. 
we may assume that $a$ is odd. Then by \cite[Theorem C]{KMTOne} again, 
$D^{\be_n}\da_{\SSS_a\wr\SSS_b}\cong D^{\be_a}\wr D^{\be_b}$, and we claim that $(D^{\be_a}\wr D^{\be_b})\da_{G_{a,b}}$ is irreducible. As vector spaces we can write
$$
D^{\be_a}\wr D^{\be_b}= (D^{\be_a})^{\otimes b}\otimes D^{\be_b}
=\bigoplus_{\eps_1,\dots,\eps_b\in\{+,-\}} E^{\be_a}_{\eps_1}\otimes\dots\otimes E^{\be_a}_{\eps_b}\otimes E^{\be_b}.
$$
Note that the direct summand $E^{\be_a}_{+}\otimes\dots\otimes E^{\be_a}_{+}\otimes E^{\be_b}$ is invariant under the action of the subgroup $\AAA_a\wr\AAA_b$, and forms a submodule of $(D^{\be_a}\wr D^{\be_b})\da_{\AAA_a\wr\AAA_b}$ isomorphic to 
$E^{\be_a}_+\wr E^{\be_b}$. Note that  
$[(D^{\be_a}\wr D^{\be_b})\da_{\AAA_a\wr\AAA_b}:E^{\be_a}_+\wr E^{\be_b}]=1$. 

If $(D^{\be_a}\wr D^{\be_b})\da_{G_{a,b}}$ is reducible, then restricting further to $\AAA_a\wr\AAA_b$, the submodule $E^{\be_a}_{+}\otimes\dots\otimes E^{\be_a}_{+}\otimes E^{\be_b}\cong E^{\be_a}_+\wr E^{\be_b}$ described in the previous paragraph, must lie in a proper submodule $V\subseteq (D^{\be_a}\wr D^{\be_b})\da_{G_{a,b}}$.  
Acting with elements of $\AAA_{a^b}$ with exactly two odd components, we see that all the subspaces $E^{\be_a}_{\eps_1}\otimes\dots\otimes E^{\be_a}_{\eps_b}\otimes E^{\be_b}$ with even $|\{k\mid \eps_k=-\}|$ lie in $V$. Next, taking into account the fact that $a$ is odd, there exists an element of $G_{a,b}\leq \SSS_a\wr\SSS_b$ with  exactly one odd component in the base group $\SSS_{a^b}$. Acting with this element we see that all the remaining  subspaces $E^{\be_a}_{\eps_1}\otimes\dots\otimes E^{\be_a}_{\eps_b}\otimes E^{\be_b}$ also lie in $V$. Thus $V=D^{\be_a}\wr D^{\be_b}$ giving a contradiction. 
\end{proof}

\subsection{Proof of Theorem~\ref{TD}}\label{SSTD}
We may assume that $G$ is not primitive. If $G$ is intransitive, then (up to conjugation) $G$ is contained in a subgroup of the form $\AAA_{n-k,k}$ for $1\leq k<n$, and we can apply Proposition~\ref{P221118}. If  $G$ is transitive then (up to conjugation) $G$ is contained in a subgroup of the form $G_{a,b}$ for $a,b\geq 2$ and $n=ab$. In this case we apply Proposition~\ref{P261118}.

\end{document}